\documentclass[10pt,reqno,hidelinks]{amsart}
\usepackage{amssymb}
\usepackage{amsmath, amssymb, verbatim, url, mathrsfs} 
\usepackage{stmaryrd}
\usepackage{tabularx}
\usepackage{mathtools}
\usepackage{verbatim}
\usepackage{amsthm}
\usepackage{framed}
\usepackage{cite}
\usepackage{wasysym}
\usepackage{upgreek}
\usepackage{color}
\usepackage[dvipsnames]{xcolor}
\usepackage{array}
\usepackage{tensor}
\usepackage{accents} 
\usepackage[T1]{fontenc}
\usepackage{dsfont}
\usepackage[colorlinks,linkcolor=blue,citecolor=blue]{hyperref}
\usepackage{enumerate}
\usepackage{enumitem}
\usepackage[normalem]{ulem}
\usepackage{longtable}
\usepackage{mathtools}

\numberwithin{equation}{section}

\theoremstyle{definition} 

\newtheorem{proposition}{Proposition}[section]
\newtheorem{lemma}[proposition]{Lemma}
\newtheorem{corollary}{Corollary}[section]
\newtheorem{theorem}{Theorem}[section]

\newtheorem{remark}{Remark}[section]

\newtheorem{definition}[proposition]{Definition}

\newtheorem*{theorem*}{Theorem}
\newtheorem*{mquestion*}{Main Question}
\newtheorem*{claim*}{Claim}

\newtheorem*{intuition*}{Intuition}

\newcommand{\ca}{\mathsf{a}}
\newcommand{\ba}{\bar{\mathsf{a}}}
\newcommand{\cb}{\mathsf{b}}
\newcommand{\cg}{\mathsf{g}}

\newcommand{\cc}{\mathsf{c}}
\newcommand{\mft}{\mathfrak{t}}
\newcommand{\mfg}{\mathfrak{g}}
\newcommand{\cm}{\mathsf{m}}
\newcommand{\bc}{\bar{\mathsf{c}}}

\newcommand{\ufo}{\underline{f_0}}

\newcommand{\ck}{\mathsf{k}}
\newcommand{\mf}{\beta}

\newcommand{\vertiiii}[1]{{\left\vert\kern-0.25ex\left\vert\kern-0.25ex\left\vert\kern-0.25ex\left\vert #1 \right\vert\kern-0.25ex\right\vert\kern-0.25ex\right\vert\kern-0.25ex\right\vert}}
\newcommand{\vertiii}[1]{{\left\vert\kern-0.25ex\left\vert\kern-0.25ex\left\vert #1 \right\vert\kern-0.25ex\right\vert\kern-0.25ex\right\vert}}

\newcommand{\mfu}{\mathfrak{u}}

\newcommand{\mfv}{\mathfrak{v}}

\newcommand{\uf}{\underline{f}}

\newcommand{\htau}{\hat{\tau}}
\newcommand{\ttau}{\tilde{\tau}}

\newcommand{\txi}{\tilde{\zeta}}

\newcommand{\Rbb}{\mathbb{R}}

\newcommand{\del}[1]{{\partial_{#1}}}

\newcommand{\AND}{{\quad\text{and}\quad}}

\newcommand{\supp}{\mathop{\mathrm{supp}}}

\newcommand{\tc}{\mathtt{c}}

\newcommand{\be}{\begin{equation}}
	\newcommand{\ee}{\end{equation}}

\allowdisplaybreaks

\begin{document}

	\title[Emergence of nonlinear Jeans  instabilities]{The emergence of nonlinear Jeans-type instabilities for quasilinear wave equations. II: Generalizations}

	\author{Chao Liu  \and Yiqing Shi}
	
	\address[Chao Liu]{Center for Mathematical Sciences and School of Mathematics and Statistics, Huazhong University of Science and Technology, Wuhan 430074, Hubei Province, China}
	\email{chao.liu.math@foxmail.com}

	\address[Yiqing Shi]{School of Mathematics and Statistics, Huazhong University of Science and Technology, Wuhan 430074, Hubei Province, China.}
\email{yiqngshi@hust.edu.cn}

\begin{abstract} 
This work extends the previous work by the first author  [arXiv:2409.02516] and [Math. Ann. 393 (2025), 317-363], analyzing the long-term behavior of solutions to a broader class of quasilinear wave equations with parameter $1<\ca\leq30$ and $\frac{1}{3}\leq\cb\leq\frac{2}{3}$:
\begin{equation*} 
	\partial^2_t  \varrho-  \biggl( \frac{ \mathsf{m}^2 (\partial_{t}\varrho  )^2}{(1+\varrho )^2}  + 4(\mathsf{k}-\mathsf{m}^2)(1+\varrho  )\biggr)  \Delta \varrho  =  F(t,\varrho,\partial_{\mu} \varrho)     
\end{equation*}   
where $F$ is given by
\begin{equation*}
	F(t,\varrho,\partial_{\mu} \varrho):=
	\cb  \varrho  (1+  \varrho  )  -(\ca-1) \partial_{t}\varrho  
	+ \frac{4}{3} \frac{(\partial_{t}\varrho )^2}{1+\varrho } +   \biggl(\mathsf{m}^2 \frac{ (\partial_{t}\varrho )^2}{(1+\varrho )^2}  + 4(\mathsf{k}-\mathsf{m}^2) (1+\varrho  )  \biggr)  q^i \partial_{i}\varrho  -  \mathtt{K}^{ij}   \partial_{i}\varrho\partial_{j}\varrho    .  
\end{equation*}
The results demonstrate that for this extensive family of quasilinear wave equations satisfying $1<\ca\leq30$ and $\frac{1}{3}\leq\cb\leq\frac{2}{3}$, self-increasing blowup solutions also exist, and self-increasing singularities emerge at certain future endpoints of null geodesics  provided the inhomogeneous perturbations of data are sufficiently small.

 \vspace{2mm}

{{\bf Keywords:} quasilinear wave equations, blowup, Jeans instability, self-increase blowup, Fuchsian system}

\vspace{2mm}

{{\bf Mathematics Subject Classification:} Primary 35L05; Secondary 35A01, 35L02, 35L10 %, 83F05
}
\end{abstract}

	%\date{Version of \today}
	\maketitle
	%	\tableofcontents
	
	\setcounter{tocdepth}{2}
	
	\pagenumbering{roman} \pagenumbering{arabic}

\section{Introduction}

Let $1<\ca\leq30$, $\frac{1}{3}\leq\cb\leq\frac{2}{3}$, $\ck>0$, $\cm^2 \leq \ck$, $\beta \in (0,+\infty)$, $\beta_0 \in (0,+\infty)$,  $t_0 \in (0,+\infty)$ and $\mathbf{t}_0=\ln t_0$ be given constants, and let $q^i$ be a given vector field. Suppose $\psi \in C^1_0(\Rbb^n)$ and $\psi_0 \in C^1_0(\Rbb^n)$ are given positive-valued functions with %\footnote{This is a simplification, and it can be generalized to $\supp \psi \subset \subset B_1(0)$ and $\supp \psi_0 \subset \subset B_1(0)$. }
$\supp \psi = \supp \psi_0 =B_1(0)$  (where $B_1(0)$ denotes the unit open ball centered at the origin). Let $\mathtt{K}_{ij}$ be analytic functions in all their variables. We \textit{aim} to identify self-increasing blowup solutions for the following type of non-covariant quasilinear wave equation, where the condition $1<\ca\leq30$, $\frac{1}{3}\leq\cb\leq\frac{2}{3}$ extends the case $\ca=\frac{4}{3}$, $\cb=\frac{2}{3}$ previously studied by Liu \cite{Liu2024}.
\begin{enumerate}[label={(Eq.\arabic*)}]
	\item\label{Eq1}  $
	\begin{gathered} 
		 \partial^2_t \varrho-  \cg \delta^{ij} \del{i}\del{j} \varrho  = 	\underbrace{\cb  \varrho  (1+  \varrho  ) }_{ \text{(i) self-increas.}} \underbrace{-(\ca-1) \del{t}\varrho  }_{ \text{(ii) damping}} 
		+ \underbrace{\frac{4}{3} \frac{(\del{t}\varrho )^2}{1+\varrho } }_{\text{(iii) Riccati}} +  \underbrace{ \cg q^i \del{i}\varrho }_{\text{(iv) convec.}} -  \mathtt{K}^{ij}(t,\varrho,\del{\mu}\varrho)    \del{i}\varrho\del{j}\varrho     , \quad  \text{in}\; [\mathbf{t}_0,\mathbf{t}^\star) \times \Rbb^n,    \label{e:eq3}       \\
		  \varrho|_{t=\mathbf{t}_0}= \beta+ \psi(x^k)   \AND 	\del{t} \varrho|_{t=\mathbf{t}_0}=  \beta_0 +  \psi_0(x^k)   ,  \quad   \text{in}\; \{\mathbf{t}_0 \} \times \Rbb^n ,  %\label{e:eq3dt}
	\end{gathered}
	$\\
	where\footnote{In this article, we use the index convention given in \S\ref{s:AIN}, i.e., $\mu=0,\cdots,n$ and $i=1,\cdots,n$, $x_0=t$, and the Einstein summation convention. } 
	\begin{equation*}%\label{e:Fdef0}
	 \cg= \cg(   \varrho,\del{t}\varrho )  =  \cm^2 \frac{ (\del{t}\varrho )^2}{(1+\varrho )^2}  + 4(\ck-\cm^2) (1+\varrho  )    .
	\end{equation*}  \end{enumerate}

To study \ref{Eq1}, this article first applies a simple exponential-time coordinate transformation (where $e^t$ is treated as the new time) to convert it into \ref{Eq2}, and then analyzes \ref{Eq2}. 	A detailed proof demonstrating that \ref{Eq1} and \ref{Eq2} are mutually convertible through a simple time transformation can be found in \cite[Appendix A]{Liu2024} . Moreover, the conclusions derived from \ref{Eq2} in this study are equally applicable to \ref{Eq1}.
	\begin{enumerate}[label={(Eq.\arabic*)}] 
	\setcounter{enumi}{1}
	\item\label{Eq2}$ 
	\begin{gathered}
		\partial^2_t \varrho-  	\cg \delta^{ij} \del{i}\del{j} \varrho = 
		\frac{\cb}{t^2}  \varrho  (1+  \varrho )-\frac{\ca}{ t} \del{t}\varrho   + \frac{4}{3} \frac{(\del{t}\varrho )^2}{1+\varrho } + \cg q^i \del{i}\varrho   - \frac{1}{ t^2} \mathtt{K}^{ij}(t,\varrho,\del{\mu}\varrho ) \del{i}\varrho\del{j}\varrho      , \quad    \text{in}\; [t_0,t^\star) \times \Rbb^n,    %\label{e:maineq0}  
		 \\
	  \varrho|_{t=t_0}=   \beta + \psi(x^k)  \AND 	\del{t} \varrho|_{t=t_0}=     \beta_0 + \psi_0(x^k)   ,  \quad  \text{in}\; \{t_0 \} \times \Rbb^n ,  %\label{e:maineq1}
	\end{gathered}$ \\
	where  
	\begin{equation*}%\label{e:Fdef}
		\cg =\cg(t, \varrho,\del{t}\varrho )  :=     \cm^2\frac{ (\del{t}\varrho )^2}{(1+\varrho )^2}  + 4(\ck-\cm^2) \frac{1+\varrho  }{t^2} . 
	\end{equation*} 
\end{enumerate}
For convenience, we denote $\mathtt{g}^{\alpha\beta} =-\delta^\alpha_0\delta^\beta_0+\cg\delta^{ij}\delta_i^\alpha\delta_j^\beta$ and view $\mathtt{g}_{\alpha\beta}$ as a Lorentzian metric.

  In the first author’s work \cite{Liu2024}, equations \ref{Eq1} and \ref{Eq2} adopted parameters $\ca=\frac{4}{3}$, $\cb=\frac{2}{3}$, and $\cc=\frac{4}{3}$, derived from simplified models of the nonlinear Jeans instability in the Euler--Poisson system \cite[\S II.9]{Peebles2020} and Einstein--Euler system \cite{Noh2004, Hwang2013a, Hwang2005, Hwang2007, Hwang2006, Hwang1999, Noh2005}. Here, we \textit{aim} to extend the parameter range to
  \begin{equation}\label{A:2}
  	1 < \ca \leq 30, \quad \frac{1}{3} \leq \cb \leq \frac{2}{3}, \quad \cc = \frac{4}{3},
  \end{equation}
  and study the long-time dynamics of solutions under these conditions. That is, this work \textit{extends} the results of \cite{Liu2024} to a more general setting, encompassing a broader range of parameters. Consequently, the derivation inevitably involves more complex computations and more delicate and challenging issues in \textit{parameter selection}. These constitute the main problems that this paper aims to address.
  
  As shown in \cite{Liu2023}, neglecting rotation, shear, and tidal forces links the case $\ca=\frac{4}{3}$, $\cb=\frac{2}{3}$, $\cc=\frac{4}{3}$ to the Euler–Poisson equations; hence, after appropriate transformations, \ref{Eq1} and \ref{Eq2} serve as simplified models of that system. We prove that when the inhomogeneous perturbation of the initial data is sufficiently small (i.e., long-wavelength), the solution can grow unboundedly, producing self-growth singularities along certain null geodesics. The growth rate of these blow-up solutions is estimated in Theorems \ref{t:mainthm1} and \ref{t:mainthm2}.
  
  Linearized Jeans theory predicts a slow growth rate $\sim t^{2/3}$ \cite{Bonnor1957, Zeldovich1971, ViatcehslavMukhanov2013, Liu2022}, inadequate for describing high-density dynamics and inconsistent with observed cosmic structure formation. This motivates the nonlinear analysis pursued here and in related studies \cite{Liu2022, Liu2022b, Liu2023a, Liu2023, Liu2024}, which reveal that $\varrho$ grows substantially faster than predicted by linear theory.
  
  Equation~\ref{Eq1} contains competing nonlinear terms: the self-growth, damping, Riccati, and convection terms (see \cite{Liu2024} for details). Retaining only the Riccati term with $\cm=0$ reduces the system to one akin to shock-formation models \cite{Speck2016, Holzegel2016, Speck2016a}. However, the interplay of self-growth, damping, and convection terms can still drive unbounded amplification of $\varrho$ along null geodesics, forming self-growth singularities when initial inhomogeneities are small. Once the Riccati term dominates and $\varrho$ becomes large, the system transitions to the regime of large $\varrho$ (cf. Remark~1.9 in \cite{Speck2016}).

\subsection{Methodology, Assumptions and Main Theorems}\label{s:mainthm}
\subsubsection{Basic Assumptions and Main Theorems}
Before stating the main theorems, we review a class of ordinary differential equations (ODEs) studied in \cite{Liu2022b}, namely \eqref{e:feq0b}--\eqref{e:feq1b}, as the reference equations
\begin{gather}
	\partial^2_\mft f(\mft) + \frac{\ca}{\mft} \del{\mft} f(\mft) - \frac{\cb}{\mft^2} f(\mft)(1 + f(\mft)) - \frac{\cc (\del{\mft} f(\mft))^2}{1 + f(\mft)} = 0, \label{e:feq0b} \\
	f(t_0) = \beta > 0, \quad \del{t} f(t_0) = \beta_0 > 0. \label{e:feq1b}
\end{gather}
The solution to this ODE serves as the reference solution (i.e., the solution of equation \ref{Eq2} is expected to asymptotically approach this reference solution). The following quantities, which depend only on the initial value $t_0$ and the constants $\beta, \beta_0$, are used to describe the self-growth behavior of both the reference solution and the solution to equation \ref{Eq2}.
\begin{align}
	\mathtt{A} &:= \frac{t_0^{-\frac{\ba - \triangle}{2}}}{\triangle} \biggl( \frac{t_0 \mf_0}{(1 + \mf)^2} - \frac{\ba + \triangle}{2} \frac{\mf}{1 + \mf} \biggr), \quad
	\mathtt{B} := \frac{t_0^{-\frac{\ba + \triangle}{2}}}{\triangle} \Bigl( \frac{\ba - \triangle}{2} \frac{\mf}{1 + \mf} - \frac{t_0 \mf_0}{(1 + \mf)^2} \Bigr) < 0, \label{e:ttA} \\
	\mathtt{C} &:= \frac{2}{2 + \ba + \triangle} \left( \max \left\{ \frac{\ba + \triangle}{2 \cb}, 1 \right\} \right)^{-1} \Bigl( \ln(1 + \mf) + \frac{\ba + \triangle}{2\cb} \frac{t_0 \mf_0}{1 + \mf} \Bigr) t_0^{-\frac{\ba + \triangle}{2}} > 0, \label{e:ttC} \\
	\mathtt{D} &:= \frac{\ba + \triangle}{2 + \ba + \triangle} \Bigl( \ln(1 + \mf) - \frac{1}{\cb} \frac{t_0 \mf_0}{1 + \mf} \Bigr) t_0. \label{e:ttD}
\end{align}
Here, $\triangle := \sqrt{(1 - \ca)^2 + 4\cb} > -\ba$ and $\ba = 1 - \ca < 0$. Furthermore, within the parameter range \eqref{A:2} for $\ca, \cb, \cc$ in equations \ref{Eq1} and \ref{Eq2}, and to simplify subsequent calculations, we introduce the following assumptions \ref{A:1}--\ref{A:4} regarding the initial values $\beta, \beta_0$ and the parameters $\ck$ and $q^i$
\begin{enumerate}[label=(A\arabic*)]
	\item \label{A:1} The initial values satisfy $\beta_0^2 \geq 6\cb\beta (1 + \beta)^{2} t_0^{-2}$.
	\item \label{A:3} The convection direction is assumed constant and can be expressed as $q^i = |q| \delta^i_1$, with $|q| = \frac{606}{2 - (2 - 3\cb)\cm^2} \in [303, \frac{606}{2 - m^2}]$.
	\item \label{A:4} $\ck = 1$.
\end{enumerate}

\begin{figure}[h]
	\centering
	\includegraphics[width=9cm]{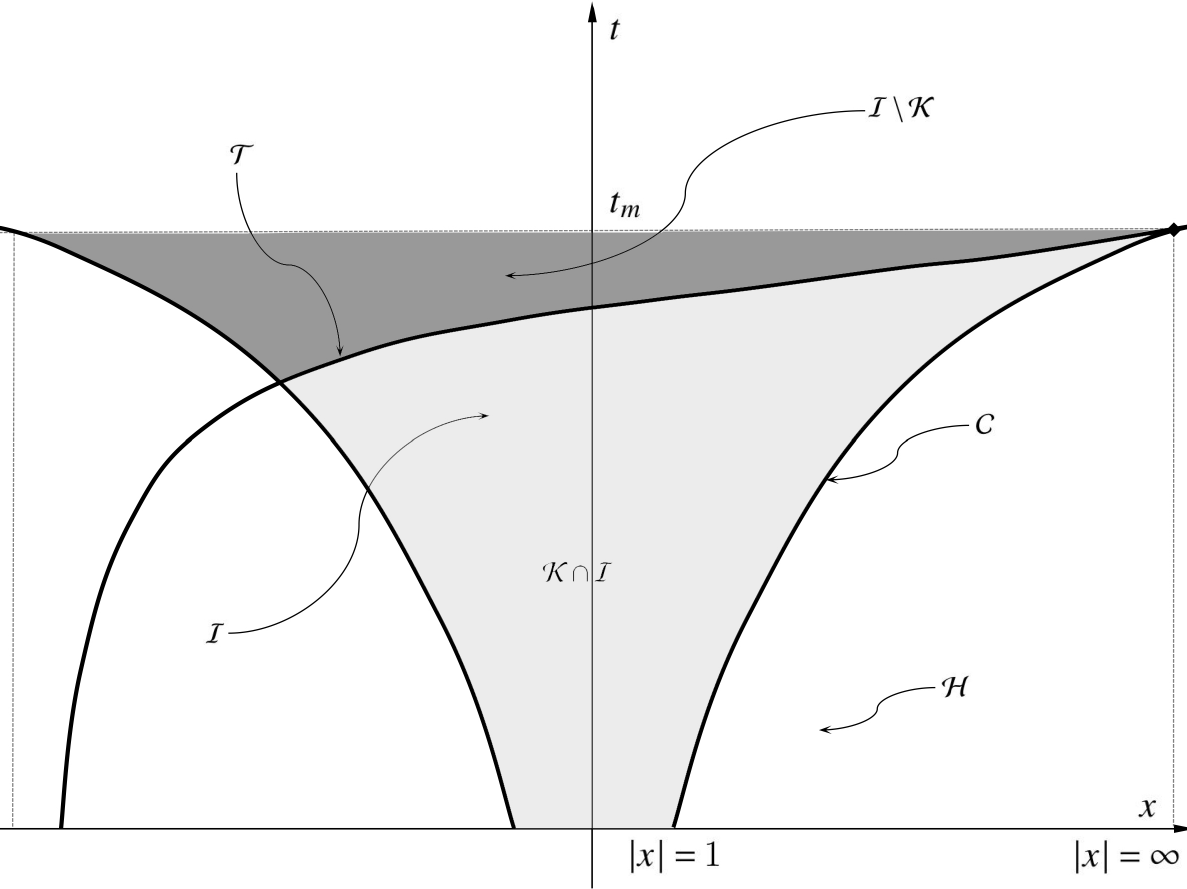}
	\caption{Regions for the solution in the main theorems.}
	\label{f:fig0}
\end{figure}

We now describe the regions appearing in the main theorems. As shown in Fig.~\ref{f:fig0}, the region $\mathcal{I}$ corresponds to the inhomogeneous solution, while $\mathcal{H}$ corresponds to the homogeneous solution. The domain of influence of the inhomogeneous initial data is bounded by the characteristic cone $\mathcal{C}$. These regions are defined by
\begin{align}
	\mathcal{I} &:= \left\{ (t, x) \in [t_0, t_m) \times \mathbb{R}^n \;\middle|\; |x| < 1 + \int_{t_0}^{t} \sqrt{ \cg(y, f(y), f_0(y)) } \, dy \right\}, \label{e:cai} \\
	\mathcal{H} &:= \left\{ (t, x) \in [t_0, t_m) \times \mathbb{R}^n \;\middle|\; |x| \geq 1 + \int_{t_0}^{t} \sqrt{ \cg(y, f(y), f_0(y)) } \, dy \right\}, \label{e:cah} \\
	\mathcal{C} &:= \left\{ (t, x) \in [t_0, t_m) \times \mathbb{R}^n \;\middle|\; |x| = 1 + \int_{t_0}^{t} \sqrt{ \cg(y, f(y), f_0(y)) } \, dy \right\}. \label{e:char1}
\end{align}

The main results are illustrated in Fig.~\ref{f:fig0}. A hypersurface $\mathcal{T}$, determined by the inhomogeneous initial data, divides the inhomogeneous region $\mathcal{I}$ into two parts: the lower light-shaded region $\mathcal{K}$ and the upper dark-shaded region $\mathcal{I}\setminus\mathcal{K}$. Theorems \ref{t:mainthm1} and \ref{t:mainthm2} rigorously define $\mathcal{K}$ and establish that the solution is fully determined in $\mathcal{K}$ and in the homogeneous region $\mathcal{H}$ (i.e., $\mathcal{K}\cup\mathcal{H}$, the light-shaded plus white domains), while it remains undetermined in the dark-shaded region $\mathcal{I}\setminus\mathcal{K}$, hereafter the unknown region. The solution is inhomogeneous within $\mathcal{K}\cap\mathcal{I}$ but homogeneous in $\mathcal{H}$. In $\mathcal{K}\cup\mathcal{H}$, the solution exhibits self-growth behavior, developing a self-growth singularity at $p_m=(t_m,+\infty,0,\dots,0)$, whereas its spatial derivatives stay small, thus $\mathcal{K}\cup\mathcal{H}$ is referred to as the long-wavelength domain.

\begin{remark}\label{s:excel}
	The white region $\mathcal{H}$ represents the homogeneous self-growing solution; the light-shaded region $\mathcal{K}\cap\mathcal{I}$ corresponds to the inhomogeneous self-growing solution; and the dark-shaded region $\mathcal{I}\setminus\mathcal{K}$ is the unknown region for future investigation.
\end{remark}

The main theorems are stated below. For detailed proofs, see \S\ref{s:pfmthm}.

\begin{theorem}\label{t:mainthm1}
	For \ref{Eq1}, assume $k \in \mathbb{Z}_{\geq \frac{n}{2} + 3}$, the parameters $\ca, \cb, \cc$ satisfy condition \eqref{A:2}, the constants $\mathtt{A}, \mathtt{B}, \mathtt{C}, \mathtt{D}$ defined in \eqref{e:ttA}--\eqref{e:ttD} depend on the initial values $\beta$ and $\beta_0$, and Assumptions \ref{A:1}--\ref{A:4} hold. Given functions $(\psi, \psi_0) \in C^1_0(\mathbb{R}^n)$ with $\supp(\psi, \psi_0) = B_1(0)$, and let $f(\mft)$ be the solution to \eqref{e:feq0b}--\eqref{e:feq1b} (provided by Theorem~\ref{t:mainthm0}). Then there exist sufficiently small constants $\sigma_0 > 0$ and $\delta_0 > 0$ such that if the initial data satisfy
	\begin{equation*}
		\| \psi \|_{H^k(B_1(0))} + \| \partial_i \psi \|_{H^{k}(B_1(0))} + \| \psi_0 \|_{H^k(B_1(0))} \leq e^{-\frac{303}{\delta_0}} \sigma_0^2,
	\end{equation*}
	then there exists a hypersurface $t = \mathcal{T}(x, \delta_0)$ satisfying
	\begin{align*}
		\Gamma_{\delta_0} &:= \{(t, x) \in [\mathbf{t}_0, t_m) \times \mathbb{R}^n \mid t = \mathcal{T}(x, \delta_0)\} \subset \mathcal{I}, \\
		\lim_{a \to +\infty} \mathcal{T}(a \delta^i_1, \delta_0) &= t_m \quad \text{and} \quad \lim_{\delta_0 \to 0+} \mathcal{T}(x, \delta_0) = t_m,
	\end{align*}
	so that \ref{Eq1} admits a solution $\varrho \in C^2(\mathcal{K} \cup \mathcal{H})$, where $\mathcal{K} := \{(t, x) \in [\mathbf{t}_0, t_m) \times \mathbb{R}^n \mid t < \mathcal{T}(x, \delta_0)\}$, and $\varrho$ satisfies the following
	\begin{enumerate}[label=(\arabic*)]
		\item As $x^1 \to +\infty$, define
		\begin{equation*}
			\mathbf{1}_{-}(x^1) := 1 - C \sigma_0^2 e^{-\frac{100}{\delta_0}} e^{-\frac{3x^1}{4}} \searrow 1 \quad \text{and} \quad \mathbf{1}_{+}(x^1) := 1 + C \sigma_0^2 e^{-\frac{100}{\delta_0}} e^{-\frac{3x^1}{4}} \searrow 1.
		\end{equation*}
		Then for all $(t, x) \in \mathcal{K} \cap \mathcal{I}$,
		\begin{gather*}
			\mathbf{1}_{-}(x^1) f_0\bigl( e^{\mathbf{t}_0} + \mathbf{1}_{-}(x^1) (e^{t} - e^{\mathbf{t}_0}) \bigr) \leq \varrho_0(t, x) \leq \mathbf{1}_{+}(x^1) f_0\bigl( e^{\mathbf{t}_0} + \mathbf{1}_{+}(x^1) (e^{t} - e^{\mathbf{t}_0}) \bigr), \\
			-C \sigma_0^2 e^{-\frac{100}{\delta_0}} e^{-\frac{3x^1}{4}} \bigl(1 + f\bigl( e^{\mathbf{t}_0} + \mathbf{1}_{-}(x^1) (e^{t} - e^{\mathbf{t}_0}) \bigr) \bigr) \leq \varrho_i(t, x) \\
			\quad \text{and} \quad \leq C \sigma_0^2 e^{-\frac{100}{\delta_0}} e^{-\frac{3x^1}{4}} \bigl(1 + f\bigl( e^{\mathbf{t}_0} + \mathbf{1}_{+}(x^1) (e^{t} - e^{\mathbf{t}_0}) \bigr) \bigr), \\
			\mathbf{1}_{-}(x^1) f\bigl( e^{\mathbf{t}_0} + \mathbf{1}_{-}(x^1) (e^{t} - e^{\mathbf{t}_0}) \bigr) \leq \varrho(t, x) \leq \mathbf{1}_{+}(x^1) f\bigl( e^{\mathbf{t}_0} + \mathbf{1}_{+}(x^1) (e^{t} - e^{\mathbf{t}_0}) \bigr).
		\end{gather*}
		Moreover, $\varrho_0$ and $\varrho$ develop a self-growth singularity at the point $p_m := (t_m, +\infty, 0, \cdots, 0)$, i.e.,
		\begin{gather*}
			\lim_{\mathcal{K} \ni (t, x) \to p_m} \varrho = \lim_{\mathcal{K} \ni (t, x) \to p_m} f = +\infty, \\
			\lim_{\mathcal{K} \ni (t, x) \to p_m} \varrho_0 = \lim_{\mathcal{K} \ni (t, x) \to p_m} f_0 = +\infty \quad \text{and} \quad \lim_{\mathcal{K} \ni (t, x) \to p_m} \varrho_i = 0.
		\end{gather*}
		
		\item On $\mathcal{H}$, $\varrho \equiv f$, where $\mathcal{H}$ is defined in \eqref{e:cah}.
		
		\item For $(t, x) \in \mathcal{K} \cap \mathcal{I}$, the growth rate of $\varrho$ can be estimated by exponential functions as follows.
		\begin{align*}
			\varrho(t, x) &\geq \mathbf{1}_{-}(x^1) f\bigl( e^{\mathbf{t}_0} + \mathbf{1}_{-}(x^1) (e^{t} - e^{\mathbf{t}_0}) \bigr) > \mathbf{1}_{-}(x^1) \bigl( e^{\mathtt{C} (e^{\mathbf{t}_0} + \mathbf{1}_{-}(x^1) (e^{t} - e^{\mathbf{t}_0}))^{\frac{\ba + \triangle}{2}}} - 1 \bigr) \\
			\intertext{and}
			\varrho(t, x) &\leq \mathbf{1}_{+}(x^1) f\bigl( e^{\mathbf{t}_0} + \mathbf{1}_{+}(x^1) (e^{t} - e^{\mathbf{t}_0}) \bigr) \\
			&< \frac{3}{2} \left( \frac{1}{1 + \mathtt{A} (e^{\mathbf{t}_0} + \mathbf{1}_{+}(x^1) (e^{t} - e^{\mathbf{t}_0}))^{\frac{\ba - \triangle}{2}} + \mathtt{B} (e^{\mathbf{t}_0} + \mathbf{1}_{+}(x^1) (e^{t} - e^{\mathbf{t}_0}))^{\frac{\ba + \triangle}{2}}} - 1 \right).
		\end{align*}
		
		\item If the initial data satisfy $\breve{\beta} := \frac{e^{\mathbf{t}_0} \mf_0}{3(\ca - 1)(1 + \mf)} - 1 > 0$, then for all $(t, x^k) \in \mathcal{K} \cap \mathcal{I}$, the lower bound for $\varrho$ becomes
		\begin{align*}
			\varrho(t, x) &\geq \mathbf{1}_{-}(x^1) f\bigl( e^{\mathbf{t}_0} + \mathbf{1}_{-}(x^1) (e^{t} - e^{\mathbf{t}_0}) \bigr) \\
			&> \mathbf{1}_{-}(x^1) \left( \frac{1 + \mf}{\left( \frac{\beta_0 e^{\mathbf{t}_0 \ca}}{3(\ca - 1)(1 + \beta)} \bigl( e^{\mathbf{t}_0} + \mathbf{1}_{-}(x^1) (e^{t} - e^{\mathbf{t}_0}) \bigr)^{1 - \ca} - \breve{\beta} \right)^3} - 1 \right),
		\end{align*}
		and $\varrho$ blows up in finite time.
	\end{enumerate}
\end{theorem}

\begin{theorem}\label{t:mainthm2}
	For \ref{Eq2}, assume $k \in \mathbb{Z}_{\geq \frac{n}{2} + 3}$, the parameters $\ca, \cb, \cc$ satisfy condition \eqref{A:2}, the constants $\mathtt{A}, \mathtt{B}, \mathtt{C}, \mathtt{D}$ defined in \eqref{e:ttA}--\eqref{e:ttD} depend on the initial values $\beta$ and $\beta_0$, and Assumptions \ref{A:1}--\ref{A:4} hold. Given functions $(\psi, \psi_0) \in C^1_0(\mathbb{R}^n)$ with $\supp(\psi, \psi_0) = B_1(0)$, and let $f(\mft)$ be the solution to \eqref{e:feq0b}--\eqref{e:feq1b} (provided by Theorem~\ref{t:mainthm0}). Then there exist sufficiently small constants $\sigma_0 > 0$ and $\delta_0 > 0$ such that if the initial data satisfy
	\begin{equation*}
		\| \psi \|_{H^k(B_1(0))} + \| \partial_i \psi \|_{H^{k}(B_1(0))} + \| \psi_0 \|_{H^k(B_1(0))} \leq e^{-\frac{303}{\delta_0}} \sigma_0^2,
	\end{equation*}
	then there exists a hypersurface $t = \mathcal{T}(x, \delta_0)$ satisfying
	\begin{align*}
		\Gamma_{\delta_0} &:= \{(t, x) \in [t_0, t_m) \times \mathbb{R}^n \mid t = \mathcal{T}(x, \delta_0)\} \subset \mathcal{I}, \\
		\lim_{a \to +\infty} \mathcal{T}(a \delta^i_1, \delta_0) &= t_m \quad \text{and} \quad \lim_{\delta_0 \to 0+} \mathcal{T}(x, \delta_0) = t_m,
	\end{align*}
	so that \ref{Eq2} admits a solution $\varrho \in C^2(\mathcal{K} \cup \mathcal{H})$, where $\mathcal{K} := \{(t, x) \in [t_0, t_m) \times \mathbb{R}^n \mid t < \mathcal{T}(x, \delta_0)\}$, and $\varrho$ satisfies the following
	\begin{enumerate}[label=(\arabic*)]
		\item\label{b1} As $x^1 \to +\infty$, define
		\begin{equation*}
			\mathbf{1}_{-}(x^1) := 1 - C \sigma_0^2 e^{-\frac{100}{\delta_0}} e^{-\frac{3x^1}{4}} \searrow 1 \quad \text{and} \quad \mathbf{1}_{+}(x^1) := 1 + C \sigma_0^2 e^{-\frac{100}{\delta_0}} e^{-\frac{3x^1}{4}} \searrow 1.
		\end{equation*}
		Then for all $(t, x) \in \mathcal{K} \cap \mathcal{I}$,
		\begin{gather*}
			\mathbf{1}_{-}(x^1) f_0\bigl( t_0 + \mathbf{1}_{-}(x^1) (t - t_0) \bigr) \leq \varrho_0(t, x) \leq \mathbf{1}_{+}(x^1) f_0\bigl( t_0 + \mathbf{1}_{+}(x^1) (t - t_0) \bigr), \\
			-C \sigma_0^2 e^{-\frac{100}{\delta_0}} e^{-\frac{3x^1}{4}} \bigl(1 + f\bigl( t_0 + \mathbf{1}_{-}(x^1) (t - t_0) \bigr) \bigr) \leq \varrho_i(t, x) \\
			\quad \text{and} \quad \leq C \sigma_0^2 e^{-\frac{100}{\delta_0}} e^{-\frac{3x^1}{4}} \bigl(1 + f\bigl( t_0 + \mathbf{1}_{+}(x^1) (t - t_0) \bigr) \bigr), \\
			\mathbf{1}_{-}(x^1) f\bigl( t_0 + \mathbf{1}_{-}(x^1) (t - t_0) \bigr) \leq \varrho(t, x) \leq \mathbf{1}_{+}(x^1) f\bigl( t_0 + \mathbf{1}_{+}(x^1) (t - t_0) \bigr).
		\end{gather*}
		Moreover, $\varrho_0$ and $\varrho$ develop a self-growth singularity at the point $p_m := (t_m, +\infty, 0, \cdots, 0)$, i.e.,
		\begin{gather*}
			\lim_{\mathcal{K} \ni (t, x) \to p_m} \varrho = \lim_{\mathcal{K} \ni (t, x) \to p_m} f = +\infty, \\
			\lim_{\mathcal{K} \ni (t, x) \to p_m} \varrho_0 = \lim_{\mathcal{K} \ni (t, x) \to p_m} f_0 = +\infty \quad \text{and} \quad \lim_{\mathcal{K} \ni (t, x) \to p_m} \varrho_i = 0.
		\end{gather*}
		
		\item\label{b2} On $\mathcal{H}$, $\varrho \equiv f$, where $\mathcal{H}$ is defined in \eqref{e:cah}.
		
		\item\label{b3} For $(t, x) \in \mathcal{K} \cap \mathcal{I}$, the growth rate of $\varrho$ can be estimated by exponential functions as follows
		\begin{align*}
			\varrho(t, x) &\geq \mathbf{1}_{-}(x^1) f\bigl( t_0 + \mathbf{1}_{-}(x^1) (t - t_0) \bigr) > \mathbf{1}_{-}(x^1) \bigl( e^{\mathtt{C} (t_0 + \mathbf{1}_{-}(x^1) (t - t_0))^{\frac{\ba + \triangle}{2}}} - 1 \bigr) \\
			\intertext{and}
			\varrho(t, x) &\leq \mathbf{1}_{+}(x^1) f\bigl( t_0 + \mathbf{1}_{+}(x^1) (t - t_0) \bigr) \\
			&< \frac{3}{2} \left( \frac{1}{1 + \mathtt{A} (t_0 + \mathbf{1}_{+}(x^1) (t - t_0))^{\frac{\ba - \triangle}{2}} + \mathtt{B} (t_0 + \mathbf{1}_{+}(x^1) (t - t_0))^{\frac{\ba + \triangle}{2}}} - 1 \right).
		\end{align*}
		
		\item\label{b4} If the initial data satisfy $\breve{\beta} := \frac{t_0 \mf_0}{3(\ca - 1)(1 + \mf)} - 1 > 0$, then for all $(t, x^k) \in \mathcal{K} \cap \mathcal{I}$, the lower bound for $\varrho$ becomes
		\begin{align*}
			\varrho(t, x) &\geq \mathbf{1}_{-}(x^1) f\bigl( t_0 + \mathbf{1}_{-}(x^1) (t - t_0) \bigr) \\
			&> \mathbf{1}_{-}(x^1) \left( \frac{1 + \mf}{\left( \frac{\beta_0 t_0^{\ca}}{3(\ca - 1)(1 + \beta)} \bigl( t_0 + \mathbf{1}_{-}(x^1) (t - t_0) \bigr)^{1 - \ca} - \breve{\beta} \right)^3} - 1 \right),
		\end{align*}
		and $\varrho$ blows up in finite time.
	\end{enumerate}
\end{theorem}

\subsubsection{Methods and Ideas}\label{s:oview}

The analysis in this work is based on the framework developed in \cite{Liu2024}; 
a detailed exposition of the method can be found in \cite[\S1.4]{Liu2024}. 
Only a brief outline of the main ideas is provided here. 
Compared with \cite{Liu2024}, the \textit{main difficulty} of this paper lies in the more complex computations and the more delicate and challenging parameter choices, which require further refinement of the previous proof and more precise parameter selection.

The approach combines the reference ODE system \eqref{e:feq0b}--\eqref{e:feq1b} introduced in \cite{Liu2022b} 
with the Fuchsian formulation originally developed by Oliynyk~\cite{Oliynyk2016a}, 
which has since been extensively applied and generalized in the analysis of nonlinear PDEs 
(see, e.g., \cite{Liu2018,Liu2018a,Beyer2020b,Fajman2021,LeFloch2021,Ames2022,Beyer2023b,Oliynyk2024}).  

The essential idea is to employ the ODE solution $f $  as a \emph{reference profile} 
and to establish that the PDE solution $ \varrho $ asymptotically approaches $f $ . 
The reference solution $f $ is self-growing and exhibits blow-up at a time $t_m$, 
which may be finite or infinite depending on the initial data. 
Accordingly, if $ \varrho $ can be shown to remain sufficiently close to $ f $ 
on the domain $ [t_0,t_m)\times\mathbb{R}^n $, 
a corresponding self-growth estimate for $ \varrho $ follows.

To achieve this, the original PDE is transformed into a \emph{Fuchsian-type singular symmetric hyperbolic system} of the form
\begin{equation}\label{e:Fucmodel}
	B^{\mu}\partial_{\mu}u = \frac{1}{\tau}\mathbf{B}\mathbf{P}u + H,
\end{equation}
where the coefficient matrices satisfy the structural and positivity assumptions specified in 
Conditions~\ref{c:2}--\ref{c:7}. 
Once the system is cast in this standard Fuchsian form, 
the \emph{Fuchsian global initial value problem} (GIVP) ensures global existence of solutions 
up to the singular time $\tau=0$, corresponding to $t=+\infty$.

Simply speaking, the transformation into Fuchsian form proceeds through three main steps.
\begin{enumerate}[label=(\arabic*)]
	\item \textbf{Time compactification:} the physical time $t\in[t_0,t_m)$ is mapped to a compact interval $\tau\in[-1,0)$ via a transformation depending on $f$;
	\item \textbf{Fuchsian variable construction:} variables are chosen as appropriately scaled deviations between $\varrho$ and $f$;
	\item \textbf{System normalization:} singular and regular components are separated, and the principal matrix is adjusted to satisfy the required positivity conditions.
\end{enumerate}

However, since equation~\ref{Eq2} lacks a synchronous source term, 
the blow-up times of $f$ and $\varrho$ generally do not coincide, 
introducing substantial analytical difficulties. 
To resolve this, the time compactification is applied separately to both functions  $f$ and $\varrho$, 
thereby aligning their asymptotic regimes. 
Additional transformations  (including the introduction of zoom-in coordinates, 
zoom-in variable rescaling, cutoff procedures, and spatial compactification) are then performed 
as in \cite[\S1.4]{Liu2024} (see it for more details), 
ultimately converting the system into a standard Fuchsian form on the compact manifold \(\mathbb{T}^n\), 
to which the Fuchsian GIVP theory applies.

In essence, the analysis relies on the Fuchsian GIVP framework, 
the fundamental properties of the reference ODE system \eqref{e:feq0b}--\eqref{e:feq1b}, 
the time compactification transformation, 
and the auxiliary quantities \(\chi_\uparrow\), \(1/(f\mfg)\), and \(1/(\mfg f^{1/2})\), 
together with the identities summarized in \S\ref{s:iden1}.  

As a result, it is shown that for sufficiently small long-wavelength perturbations, 
the solution \(\varrho(t,x^k)\) remains close to the reference profile \(f(\mft)\) 
and inherits its self-growth behavior up to the singular time.

\subsection{Summaries and outlines} 
This article employs three successive coordinate transformations to reformulate the wave equation into standard Fuchsian form.
\begin{enumerate}[label=(\arabic*)]
	\item $(t,x)\!\to\!(\tau,\zeta)$ compactifies time so that the blow-up time corresponds to $\tau=0$, enabling the Fuchsian formulation;
	\item $(\tau,\zeta)\!\to\!(\ttau,\txi)$ with $\mathfrak{U}=e^{\theta\txi^1}U$ renders the matrix $\mathbf{B}$ in \eqref{e:Fucmodel} positive definite;
	\item $(\ttau,\txi)\!\to\!(\htau,\hat{\zeta})$ compactifies space onto the torus $\mathbb{T}^n$, as required by the Fuchsian GIVP theory.
\end{enumerate}

\S\ref{s:2} introduces the time-compactification transformation and related identities essential for converting equation~\ref{Eq2} into Fuchsian form.  
\S\ref{s:3} derives a singular symmetric hyperbolic system from \ref{Eq2} using suitable variables and analyzes the domain of influence of inhomogeneous data.  
\S\ref{s:4} revises this system through zoom-in coordinates, variable rescalings, and cutoff modifications, extends it to the compact manifold $\mathbb{T}^n$, and verifies that the resulting formulation satisfies the standard Fuchsian structure.  
Finally, \S\ref{s:pfmthm} establishes Theorems~\ref{t:mainthm1} and \ref{t:mainthm2}, with supporting lemmas proved in Appendix~\ref{s:pfmthm1}.

\subsection{Notations}\label{s:AIN}
Unless stated otherwise, the following conventions are used throughout this article.

\subsubsection{Indices and coordinates}\label{iandc}
Latin indices ($i,j,k$) denote spatial components ($1\le i\le n$), and Greek indices ($\alpha,\beta,\gamma$) denote spacetime components ($0\le \alpha\le n$). The Einstein summation convention is adopted.  
Spatial coordinates on $\Rbb^n$ or $\mathbb{T}^n$ are denoted by $x^i$ (or $\zeta^i$), and time by $t=x^0$ (or $\tau=\zeta^0$). For $\zeta=(\zeta^1,\dots,\zeta^n)$, we set $|\zeta|^2=\delta_{ij}\zeta^i\zeta^j$.

Four coordinate systems are introduced: $(t,x)$, $(\tau,\zeta)$, $(\ttau,\txi)$, and $(\htau,\hat{\zeta})$.  
An object $Q$ in these coordinates is denoted respectively by $\underline{Q}$, $\widetilde{Q}$, and $\widehat{Q}$; for instance,
\begin{equation*}
\underline{\varrho}(\tau,\zeta)=\varrho(t(\tau,\zeta),x(\tau,\zeta)), \quad 
\widetilde{\mathbf{A}^0}(\ttau,\txi)=\mathbf{A}^0(\tau(\ttau,\txi),\zeta(\ttau,\txi)), \quad
\widehat{\mathfrak{U}}(\htau,\hat{\zeta})=\mathfrak{U}(\ttau(\htau,\hat{\zeta}),\txi(\htau,\hat{\zeta})).
\end{equation*}

\subsubsection{Remainder terms}\label{s:rmdrs}
A function $f(x,y)$ is said to vanish to $n$-th order in $y$ if $|f(x,y)|\le C|y|^n$ as $y\to0$.  
To represent analytic remainders whose explicit form is irrelevant, we use script symbols such as $\mathscr{S}(\tau,U;V)$, $\mathscr{Z}(\tau,U;V)$, and $\mathscr{H}(\tau,U;V)$, all analytic in $(\tau,U,V)$ and vanishing to first order in $V$.

\subsubsection{Derivatives}\label{s:der}
Partial derivatives with respect to $x^\mu=(t,x^i)$ are denoted $\partial_\mu=\partial/\partial x^\mu$.  
We write $Du=(\partial_j u)$ for the spatial gradient and $\partial u=(\partial_\mu u)$ for the spacetime gradient.  
Greek letters may also denote multi-indices $\alpha=(\alpha_1,\dots,\alpha_n)\in\mathbb{Z}_{\ge0}^n$, with
$D^\alpha=\partial_1^{\alpha_1}\cdots\partial_n^{\alpha_n}$; the context clarifies the meaning.

\subsubsection{Function spaces and matrix inequalities}\label{s:funsp}
For a finite-dimensional vector space $V$, $H^s(\Rbb^n,V)$ denotes the space of $V$-valued functions with $s$ derivatives in $L^2(\Rbb^n)$; when $V=\Rbb^N$, we simply write $H^s(\Rbb^n)$.  
The $L^2$ inner product and $H^s$ norm are defined as
\begin{equation*}
\langle u,v\rangle=\int_{\Rbb^n}(u(x),v(x))\,d^n x, \qquad
\|u\|_{H^s}^2=\sum_{|\alpha|\le s}\langle D^\alpha u, D^\alpha u\rangle,
\end{equation*}
where $(\xi,\zeta)=\xi^T\zeta$ is the Euclidean product on $\Rbb^N$.  
For symmetric matrices $A,B\in\mathbb{M}_{N\times N}$, we write
\begin{equation*}
	A\le B \quad \Leftrightarrow \quad (\zeta,A\zeta)\le (\zeta,B\zeta),\ \forall\,\zeta\in\Rbb^N.
\end{equation*}

%%%--------------NEW SEC--------------

\section{Compactified time transformations and analysis of the reference equation}
\label{s:2} 
In \cite{Liu2022b}, the presence of a synchronous source term ensures that $\varrho$ and $f$ blow up simultaneously at $t=t_m$, so a single compactification $\tau=\mathfrak{g}(t)$ suffices to map $t=t_m$ to $\tau=0$.  
In the present case, the lack of this synchronization requires two independent time transformations to compactify the coordinates $t$ for $\varrho(t,x)$ and $\mft$ for $f(\mft)$, respectively.

Since our analysis focuses on the long-wavelength region $\mathcal{K}\cup\mathcal{H}$, where $\varrho$ closely approximates $f$, we adopt the similar compactification form as in \cite{Liu2022b}, with modified parameters:
\begin{align}
	\tau &= \mathfrak{g}(\mft)= -\Bigl(1+\cb B\!\int_{t_0}^{\mft} s^{\ca-2} f(s)(1+f(s))^{1-\cc}\,ds\Bigr)^{-\frac{A}{\cb}} \in [-1,0)  , \label{e:ctm2}\\
	\tau &= g(t,x)= -\Bigl(1+\cb B\!\int_{t_0}^{t} s^{\ca-2} \varrho(s,x)(1+\varrho(s,x))^{1-\cc}\,ds\Bigr)^{-\frac{A}{\cb}}. \label{e:ctm2b}
\end{align}
When $\varrho=f$, the two coincide; in general, each maps the blow-up time to $\tau=0$.  
Because $g$ depends on $\varrho$, the coordinate change is governed by an evolution equation,
\begin{equation*}%\label{e:dtgtrs}
	\partial_t g(t,x)=\frac{AB\,\varrho(t,x)\bigl(-g(t,x)\bigr)^{\frac{\mathsf{b}}{A}+1}}{t^{2-\mathsf{a}}(1+\varrho(t,x))^{\mathsf{c}-1}}, 
	\qquad g(t_0,x)=-1,
\end{equation*}
which is solved jointly with the main equation~\ref{Eq2}.  
For convenience, we also consider its inverse $\mathsf{h}=\mathsf{h}(\tau,\zeta)$, whose derivatives $\mathsf{h}_\zeta=\partial_\zeta\mathsf{h}$ satisfy a corresponding evolution equation, used in converting \ref{Eq2} into a singular hyperbolic system in $(\tau,\zeta)$ coordinates (see \S\ref{s:2}).

After transformation, both $\varrho$ and $f$ share the blow-up time $\tau=0$, making it natural to compare them in $\tau$-coordinates.  
Hence, new variables such as 
$u(\tau,\zeta):=(\underline{\varrho}(\tau,\zeta) - \uf(\tau))/\uf(\tau)$
are introduced to measure their relative deviation (see \S\ref{s:sngexp} for details).

In summary, two independent time compactifications, $\tau=\mathfrak{g}(\mft)$ for $f$ and $\tau=g(t,x)$ for $\varrho$, are essential to synchronize their blow-up times and to cast equation~\ref{Eq2} into the Fuchsian framework. 
In what follows, Lemmas~\ref{t:gb2} and~\ref{t:gb1} are analogous to \cite[Lemmas~2.1–2.2]{Liu2024}. Their statements are presented without proof, as the arguments are essentially identical to those in \cite[Lemmas~2.1–2.2]{Liu2024}.

\subsection{Time Compactification for the Reference Solution}\label{s:2.2}
We introduce the  time compactification transformation $	\tau = \mathfrak{g}(\mft)$ for the reference equation \eqref{e:feq0b}--\eqref{e:feq1b}. 

\begin{lemma}\label{t:gb2}
	Let $f\in C^2([t_0,t_1))$, $t_1>t_0$, be the solution to the reference ODE \eqref{e:feq0b}--\eqref{e:feq1b}, let $\mathfrak{g}(\mft)$ be defined as in \eqref{e:ctm2}, and let $f_0(\mft):=\partial_{\mft} f(\mft)$. Then the following statements hold:
	\begin{enumerate}[label=(\arabic*)]
		\item\label{l:2.1} $\tau=\mathfrak{g}(\mft)$ can be expressed as
		\begin{equation*}
			\tau=\mathfrak{g}(\mft)=-\exp\Bigl(-A\int^{\mft}_{t_0} \frac{f(s)(f(s)+1)}{s^2 f_0(s)} ds \Bigr)<0.
		\end{equation*}
		\item\label{l:2.2} The function $\mathfrak{g}(\mft)$ satisfies
		\begin{align*}%\label{e:dtg1}
			\partial_{\mft}\mathfrak{g}(\mft) = -A \mathfrak{g}(\mft)   \frac{f(\mft)(f(\mft)+1)}{\mft^2f_0(\mft)} = \frac{A B f(\mft) (-\mathfrak{g}(\mft) )^{1+\frac{\cb}{A}}  }{  \mft^{2-\ca}  (1+f(\mft))^{\cc-1}} \AND
			\mathfrak{g}(t_0) =   -1 .
		\end{align*}
		\item\label{l:2.3} The inverse transformation of \eqref{e:ctm2} is
		\begin{equation*}%\label{e:ctmi2}
			\mft=\mathsf{h}_\uparrow(\tau) ,
		\end{equation*}
		where $\mathsf{h}_\uparrow(\tau)$ satisfies the ODE
		\begin{align*}%\label{e:dtbup1}
			\partial_{\tau}\mathsf{h}_\uparrow (\tau)  =  \frac{\mathsf{h}_\uparrow^{2-\ca}(\tau) (1+\underline{f}(\tau) )^{\cc-1}}{A B  \underline{f}(\tau)  (-\tau)^{\frac{\cb}{A}+1}} \AND
			\mathsf{h}_\uparrow(-1) = t_0  .
		\end{align*}
		\item\label{l:2.4} Substituting the compactified time coordinate $\tau$ into the reference ODE \eqref{e:feq0b}--\eqref{e:feq1b} yields
		\begin{equation*}\label{e:dtf0eq1} 	\partial_{\tau}\underline{f_0} + \ca  \frac{\mathsf{h}_\uparrow^{1-\ca} (1+ \uf)^{\cc-1}}{A B  \uf (-\tau)^{\frac{\cb}{A}+1}} \underline{f_0}  -\cb  \frac{\mathsf{h}_\uparrow^{-\ca} (1+ \uf )^{\cc} }{A B  (-\tau)^{\frac{\cb}{A}+1}} -\cc \frac{\mathsf{h}_\uparrow^{2-\ca}  \underline{f_0}^2 (1+ \uf )^{\cc-2}}{A B  \uf (-\tau)^{\frac{\cb}{A}+1}}= 0,
		\end{equation*}
		where $\uf(\tau):=f\circ\mathsf{h}_\uparrow(\tau)$ and $\ufo(\tau):=f_0\circ\mathsf{h}_\uparrow(\tau)$.
		\item\label{l:2.5} $\partial_{\tau}\uf $ and $\underline{f_0}$ satisfy
		\begin{equation*}%\label{e:dtfeq1}
			\partial_{\tau} \uf  = \frac{\mathsf{h}_\uparrow^{2-\ca} (1+ \uf )^{\cc-1}}{A B  \uf   (-\tau)^{\frac{\cb}{A}+1}}	\underline{f_0}  .
		\end{equation*}
		\item\label{l:2.6} In the compactified time coordinate $\tau$, $\underline{f_0}(\tau)$ is given by
		\begin{equation*}%\label{e:f0frl}
			\underline{f_0}(\tau) =B^{-1} (\mathsf{h}_{\uparrow}(\tau))^{-\ca} (-\tau)^{-\frac{\cb}{A}} (1+ \uf(\tau))^{\cc } >0  .
		\end{equation*}
	\end{enumerate}
\end{lemma}

\subsection{Time Compactification for the Perturbed Solution}\label{s:2.1}
To study $\varrho(t,x)$, we introduce the following compactified time coordinate system $(\tau, \zeta^i)$, which is a prerequisite for obtaining the Fuchsian equation.
\begin{equation}\label{e:coord2}
	\tau = g(t,x^i) \quad \text{and} \quad \zeta^i = x^i,
\end{equation}
where the function $g(t, x^i)$ satisfies the following equations,
\begin{align}
	\partial_{t}g(t,x^i) &= 	\frac{A B   \varrho(t,x^i) \left(-g(t,x^i)\right)^{\frac{\mathsf{b}}{ A}+1}}{t^{2-\mathsf{a}}   (  \varrho(t,x^i)+1)^{\mathsf{c}-1}} ,  \label{e:tmeq1}   \\
	g(t_0,x^i) &= -1 ,  \label{e:tmeq2}
\end{align}
with constants $0 < A < \frac{2\cb}{3-2\cc}$ and $B := (1+\mf)^{\cc} / ( t_0^{\ca} \mf_0) > 0$. Note that the coordinates $(\tau,\zeta)$ and the solution $\varrho$ are determined simultaneously by solving equations \ref{Eq2} and \eqref{e:tmeq1}--\eqref{e:tmeq2} together.

\begin{lemma}\label{t:gb1}
	If the coordinates $(\tau,\zeta^i)$ are defined by \eqref{e:coord2}, the function $g$ satisfies \eqref{e:tmeq1}--\eqref{e:tmeq2}, and $0<\varrho \in C^2([t_0,t_\star]\times \mathbb{R})$ is a solution to the main equation \ref{Eq2}, then the following statements hold:
	\begin{enumerate}[label=(\arabic*)]
		\item\label{l:1.1} $\tau = g(t,x^i)$ can be expressed as
		\begin{equation*}\label{e:tmdef}
			\tau = g(t,x^i) = - \left(1+ \cb B  \int^t_{t_0} s^{\ca-2} \varrho(s,x^i)(1+\varrho(s,x^i))^{1-\cc}  ds \right)^{-\frac{A}{\cb}} \in [-1,0).
		\end{equation*}
		\item\label{l:1.2} The inverse transformation of \eqref{e:coord2} is
		\begin{equation}\label{e:coordi2}
			t = \mathsf{h}(\tau,\zeta^i) \quad \text{and} \quad x^i = \zeta^i,
		\end{equation}
		where $\mathsf{h}(\tau,\zeta^i)$ satisfies the ODE,
		\begin{align*}%\label{e:tmeqi1}
			\partial_{\tau} \mathsf{h} (\tau, \zeta^i) = \frac{\mathsf{h}^{2-\ca} (\tau, \zeta^i)  ( 1 + \underline{\varrho}(\tau,\zeta^i))^{\cc-1}}{A B  \underline{\varrho}(\tau,\zeta^i) \left(-\tau\right)^{\frac{\cb}{A}+1}} \AND
			\mathsf{h}(-1,\zeta^i) = t_0.
		\end{align*}
		\item\label{l:1.3} Let $\mathsf{h}_i := \partial_{\zeta^i}\mathsf{h}$. Then $\mathsf{h}_i$ satisfies
		\begin{align*}%\label{e:bzeq1}
			\partial_{\tau} \mathsf{h}_{i}	
			=&  \frac{(2-\ca)(-\tau )^{-\frac{\cb}{A}-1} \mathsf{h}_i \left(\underline{\varrho } +1\right)^{\cc-1}}{A B  \mathsf{h}^{\ca-1} \underline{\varrho } }       -\frac{(-\tau )^{-\frac{\cb}{A}-1} \mathsf{h}^{2-\ca} \partial_{\zeta^i} \underline{\varrho }   \left(\underline{\varrho } +1\right)^{\cc-1}}{A B  \underline{\varrho }^2} \nonumber\\
			&+\frac{(\cc-1)(-\tau )^{-\frac{\cb}{A}-1} \mathsf{h}^{2-\ca} \partial_{\zeta^i} \underline{\varrho }  }{A B  \underline{\varrho }  \left(\underline{\varrho } +1\right)^{2-\cc}}.
		\end{align*}
		\item\label{l:1.4} For any function $F(t,x^i)$, the following relations hold:
		\begin{align*}
			\underline{\partial_{t}F} = \frac{A B  \underline{\varrho}  \left(-\tau\right)^{\frac{\cb}{ A}+1}}{\mathsf{h}^{2-\ca}   ( \underline{\varrho} +1)^{\cc-1}} \partial_{\tau} \underline{F} \AND
			\underline{\partial_{x^i}F}  = -  \frac{A B  \underline{\varrho}  \left(-\tau\right)^{\frac{\cb}{A}+1} \mathsf{h}_i  } {\mathsf{h}^{2-\ca}   ( \underline{\varrho} +1)^{\cc-1}} \partial_{\tau} \underline{F} + \partial_{\zeta^i} \underline{F}.
		\end{align*}
	\end{enumerate}
\end{lemma}

\section{Singular Symmetric Hyperbolic System}\label{s:3}

Following the approach in \cite{Liu2022b}, we rewrite equation~\ref{Eq2} in the compactified coordinates $(\tau,\zeta)$ using a suitable set of Fuchsian variables, obtaining a first-order singular symmetric hyperbolic system with a singularity at $\tau=0$ (see Lemma~\ref{t:mainsys1}).  
Unlike \cite{Liu2022b}, the parameters $\ca$, $\cb$, $\cc$ are not fixed, and an undetermined parameter $\ell_0$ appears in \eqref{e:v4}, which will later be specified in \eqref{e:l0def}. Hence, the calculations in this section need to be redone. While the calculations in this section are more involved, they introduce no new ideas. We thus present only the key results without proofs; for details, see \cite{Liu2024}. The derivation relies on key identities from \S\ref{s:2} and Appendix~\ref{t:refsol}.

\subsection{Derivation of the Singular Form}\label{s:sngexp}
We introduce the following  variables. 
\begin{align}
	u(\tau,\zeta^k) = 	\frac{\underline{\varrho}(\tau,\zeta^k)- \uf(\tau) }{\uf(\tau )}  \quad \Leftrightarrow \quad & 	\underline{\varrho}(\tau,\zeta^k)= \uf(\tau) + \uf(\tau )u(\tau,\zeta^k),  \label{e:v1}  \\
	u_0(\tau,\zeta^k)   =  	\frac{ \underline{\varrho_0}(\tau,\zeta^k)  - \underline{f_0}(\tau) }{ \underline{f_0}(\tau) }   \quad \Leftrightarrow \quad & 	\underline{\varrho_0}(\tau,\zeta^k)  = \underline{f_0}(\tau)  +\underline{f_0}(\tau)   u_0(\tau,\zeta^k),  \label{e:v2}  \\
	u_i(\tau,\zeta^k)  =  \frac{ \underline{\varrho_i} (\tau,\zeta^k) }{1+\underline{f}(\tau)}   \quad \Leftrightarrow \quad & 	\underline{\varrho_i} (\tau,\zeta^k) = (1+\underline{f}(\tau)) u_i(\tau,\zeta^k),  \label{e:v3}   \\
	z(\tau,\zeta^k) =  \biggl(\frac{\mathsf{h}(\tau,\zeta^k)}{\mathsf{h}_\uparrow(\tau)}\biggr)^{\frac{1}{\ell_0}}-1 \quad \Leftrightarrow \quad & 	\mathsf{h}(\tau,\zeta^k)  =  \bigl(1+ z (\tau,\zeta^k) \bigr)^{\ell_0} \mathsf{h}_\uparrow(\tau),  \label{e:v4}
\end{align}
and
\begin{align}\label{e:v5} 
	\mathcal{B}_{j} (\tau,\zeta^k)   =  \frac{(\mathsf{h}_{\uparrow}(\tau))^{\ca-2}(1+\uf(\tau))^{2-\cc}\mathsf{h}_j(\tau,\zeta^k)}{  B^{-1}  (-\tau)^{-\frac{\cb}{A}}  }       \quad \Leftrightarrow \quad & 	\mathsf{h}_j(\tau,\zeta^k)= \frac{  B^{-1}  (-\tau)^{-\frac{\cb}{A}}  \mathcal{B}_j (\tau,\zeta^k)}  {(\mathsf{h}_{\uparrow}(\tau))^{\ca-2}(1+\uf(\tau))^{2-\cc}},
\end{align}
where $\varrho_\mu:=\partial_{x^\mu}\varrho$.

The following lemma collects several identities for these variables, which will be used throughout.  
Since their proofs are analogous to those in \cite[\S3]{Liu2024}, we state them without proof.

\begin{lemma}\label{t:id0}
	If $\ell_1,\ell_2,\ell_3,\ell_4$ are constants, then the following identities hold:
	\begin{gather}
		\frac{    \mathsf{h}^{\ell_2} (1+\uf+\uf u)^{\ell_3}}{A B  (\uf+\uf u)^{\ell_1} (-\tau)^{\frac{\ell_4}{3A}+1}}
		=  \frac{   \mathsf{h}_\uparrow^{\ell_2}   (1+\uf)^{\ell_3}   }{A B \uf^{\ell_1} (-\tau)^{\frac{\ell_4}{3A}+1}}  \frac{ (   1+ z  )^{\ell_0\ell_2} (1+\frac{\uf}{1+\uf} u)^{\ell_3}  }{ (1+ u)^{\ell_1}  },  \label{e:id1} \\
		\frac{  B^{-1}  (-\tau)^{-\frac{\cb}{A}}  }  {(\mathsf{h}_{\uparrow}(\tau))^{\ca-2}(1+\uf(\tau))^{2-\cc}}=\frac{\mathsf{h}_\uparrow^2(\tau) \underline{f_0}(\tau)}{(1+\uf(\tau))^2}	 = \frac{\underline{\chi_\uparrow}(\tau) }{B  } \frac{ \uf (\tau) }{\ufo (\tau )}.  \label{e:id3}
	\end{gather}
	Moreover, from \eqref{e:id3} it follows that
	\begin{equation}\label{e:id2}
		\mathsf{h}_j(\tau,\zeta^k)= 	\frac{\mathsf{h}_\uparrow^2(\tau) \underline{f_0}(\tau)}{(1+\uf(\tau))^2}	 \mathcal{B}_j (\tau,\zeta^k) = \frac{\underline{\chi_\uparrow}(\tau) }{B  } \frac{ \uf (\tau) }{\ufo (\tau )}\mathcal{B}_j (\tau,\zeta^k).
	\end{equation}
\end{lemma}

We define frequently used quantities,
\begin{gather}
	\mathscr{R}^j := \underline{\cg} \delta^{ij} \mathsf{h}_{i} \frac{(1+\underline{f} ) }{\underline{f_0} } \overset{\eqref{e:id2}}{=} \underline{\cg} \delta^{ij} \frac{\uf(1+\underline{f} ) }{\underline{f_0}^2 } \frac{\underline{\chi_\uparrow} }{B } \mathcal{B}_i , \quad 
	H^{ij} := \underline{\cg} \delta^{ij} \frac{\mathsf{h}_\uparrow^{2} }{\uf} \frac{ (1 + z )^{\ell_0(2-\ca)} (1+\frac{\uf}{1+\uf} u )^{\cc-1} }{ (1+ u) } , \label{e:Rdef}  
	\\
	\mathscr{S}=\mathscr{S}(\tau):= \frac{ (\ck-\cm^2) }{\tau} \frac{B}{\underline{\chi_\uparrow}} \biggl( \frac{4}{ \uf } - \frac{\underline{\mathfrak{G}}}{ B} \biggr) , \quad 
	S=S(\tau):= \ck +\frac{(\ck-\cm^2)(12-2\cb-8\cc)}{2\cb+(3-2\cc)\frac{\underline{\mathfrak{G}}}{ B}} + \tau \mathscr{S} , \label{e:S1} \\  
	\mathscr{L}=\mathscr{L}(\tau;u_0,u,z) := \cm^2 \left( \frac{ (1 + u_0 )^2}{ (1+ \frac{\uf}{1+\uf } u )^2} -1\right) + 4(\ck-\cm^2) \frac{B}{\underline{\chi_\uparrow}} \biggl(1+\frac{1}{\uf}\biggr) \left(\frac{(1+\frac{\uf}{1+\uf} u ) }{(1+ z )^{2\ell_0}}-1\right) . \label{e:L1}   
\end{gather} 

\begin{lemma}\label{t:coef1} 
	Substituting the variables \eqref{e:v1}--\eqref{e:v5} into \eqref{e:Rdef}--\eqref{e:L1} yields,
	\begin{align*}
		\mathscr{R}^j &= \mathscr{R}^j(\tau,u_0,u,z;\mathcal{B}_k ):= R \mathcal{B}_i \delta^{ij} = (S+\mathscr{L}) \frac{\uf}{1+\uf} \frac{\underline{\chi_\uparrow}}{B} \mathcal{B}_i \delta^{ij}, \\
		H^{ij} &= H^{ij}(\tau,u_0,u,z) := S \delta^{ij} \frac{\underline{\chi_\uparrow}}{B} + \mathscr{H} \delta^{ij},  
	\end{align*}
	where
	\begin{align}		
		R &:= R(\tau,u_0,u,z)= (S+\mathscr{L} ) \frac{\uf}{1+\uf} \frac{\underline{\chi_\uparrow}}{B}, \label{e:R1}\\
		\mathscr{H} &:= \mathscr{H}(\tau; u_0,u, z)= \mathscr{L} \frac{\underline{\chi_\uparrow}}{B} + (S+\mathscr{L} ) \frac{\underline{\chi_\uparrow}}{B} \left( \frac{ (1 + z )^{\ell_0(2-\ca)} (1+\frac{\uf}{1+\uf} u)^{\cc-1} }{ (1+ u) }-1\right). 
	\end{align}
	Furthermore, the following inequalities hold,
	\begin{align*}
		\mathscr{S}(\tau) & < -\frac{(\ck-\cm^2)}{\tau}, \\
		\text{and} \quad \cm^2+\frac{(\ck-\cm^2)(12-2\cb-8\cc)}{2\cb+(3-2\cc)\frac{\underline{\mathfrak{G}}}{ B}} & < S(\tau) \leq \ck\bigl(1+ \frac{1}{\beta}\bigr)+\frac{(\ck-\cm^2)(12-2\cb-8\cc)}{2\cb+(3-2\cc)\frac{\underline{\mathfrak{G}}}{ B}}.
	\end{align*}
\end{lemma}

We also frequently use the identity later (from Proposition \ref{t:limG}):
 \begin{align}\label{e:chig}
 	\frac{\underline{\chi_\uparrow}}{B} = \frac{2\cb}{3-2\cc}+\frac{\underline{\mathfrak{G}}}{B} . 
 \end{align}

 Let $\mathscr{R}_k := \delta_{jk} \mathscr{R}^j$, $H^i_k := \delta_{jk} H^{ij}$. Then, using Lemmas \ref{t:id0} and \ref{t:coef1}, analogous but more involved calculations as in \cite[\S 3]{Liu2024} (omitting details) yield the following singular symmetric hyperbolic system. Note that \eqref{e:mainsys1} does not satisfy Condition \ref{c:5} of Appendix \ref{s:fuch}, so Theorem \ref{t:fuch} cannot be directly applied; modifications are discussed in the next section.

\begin{lemma}\label{t:mainsys1}
	Equation \ref{Eq2} can be reformulated as the following first-order singular symmetric hyperbolic system,
	\begin{align}\label{e:mainsys1}
		\mathbf{A}^0 \partial_{\tau} U + \frac{1}{A\tau} \mathbf{A}^i \partial_{\zeta^i} U = \frac{1}{A\tau} \mathbf{A} U + \mathbf{F},
	\end{align}
	where $U := (u_0, u_j, u, \mathcal{B}_l, z)^T$, and
	\begin{align*}
		\mathbf{A}^0  = \begin{pmatrix}
			1 & \mathscr{R}^j & 0 & 0 & 0 \\
			\mathscr{R}_k & (S + \mathscr{L}) \delta^j_k & 0 & 0 & 0 \\
			0 & 0 & 2 & 0 & 0 \\
			0 & 0 & 0 & \delta^l_s & 0 \\
			0 & 0 & 0 & 0 & 1
		\end{pmatrix},  \quad
		\mathbf{A}^i  = \begin{pmatrix}
			0 & H^{ij} & 0 & 0 & 0 \\
			H^i_k & 0 & 0 & 0 & 0 \\
			0 & 0 & 0 & 0 & 0 \\
			0 & 0 & 0 & 0 & 0 \\
			0 & 0 & 0 & 0 & \mathtt{d}^i
		\end{pmatrix}, \quad (\mathtt{d}^i \text{ are arbitrary constants})
	\end{align*}
	\begin{align*}
		\mathbf{A} &= \begin{pmatrix}
			\mathbf{D}_{11} + \mathscr{Z}_{11} & \mathbf{D}^j_{12} + \mathscr{Z}^j_{12} & \mathbf{D}_{13} + \mathscr{Z}_{13} & 0 & \mathbf{D}_{15} + \mathscr{Z}_{15} \\
			0 & (\mathbf{D}_{22} + \mathscr{Z}_{22}) \delta^j_k & 0 & (\mathbf{D}_{24} + \mathscr{Z}_{24}) \delta^l_k & 0 \\
			\mathbf{D}_{31} + \mathscr{Z}_{31} & 0 & \mathbf{D}_{33} + \mathscr{Z}_{33} & 0 & \mathbf{D}_{35} + \mathscr{Z}_{35} \\
			0 & (\mathbf{D}_{42} + \mathscr{Z}_{42}) \delta^j_{s} & 0 & (\mathbf{D}_{44} + \mathscr{Z}_{44}) \delta_{s}^l & 0 \\
			0 & 0 & 0 & 0 & 0
		\end{pmatrix}.
	\end{align*}
	and $\mathbf{F} = (\mathfrak{F}_{u_0}, \mathfrak{F}_{u_j}, \mathfrak{F}_{u}, \mathfrak{F}_{\mathcal{B}_i}, \mathfrak{F}_{z})^T$. 
	Here, the constant matrices $\mathbf{D}_{mn}$ are given by
	\begin{align*}
		\mathbf{D}_{11} &= \frac{3\cb - 4\cb\cc}{3 - 2\cc}, &
		\mathbf{D}^j_{12} &= \left(-4\ck + 4\cm^2 - \frac{2\cb}{3-2\cc}\cm^2\right) q^j, &
		\mathbf{D}_{13} &= \frac{3\cb\cc}{3-2\cc}, \\
		\mathbf{D}_{15} &= \frac{\ell_0 \cb (3\ca - 4\cc)}{3-2\cc}, &
		\mathbf{D}_{22} &= \frac{2\cb}{3-2\cc} \ck, &
		\mathbf{D}_{24} &= \frac{2\cb\ck}{3-2\cc} \left(\cb + \frac{2\ca\cb}{3-2\cc}\right), \\
		\mathbf{D}_{31} &= -\frac{4\cb}{3-2\cc}, &
		\mathbf{D}_{33} &= \frac{4\cb(3-\cc)}{3-2\cc}, &
		\mathbf{D}_{35} &= -\frac{4\cb\ell_0(2-\ca)}{3-2\cc}, \\
		\mathbf{D}_{42} &= 2 - \cc, &
		\mathbf{D}_{44} &= \cb . 
	\end{align*}
and
	\begin{align*}
		\mathfrak{F}_{u_0} = & \mathfrak{F}_{u_0}(\tau, u_0,u, z) := \frac{\underline{\mathfrak{G}}}{\tau} \mathscr{S}_{11}( \tau; u_0,u, z ) + \frac{1}{\tau \uf^{\frac{1}{2}}} \mathscr{S}_{12}(\tau,u_0; u, z ),  \\
		\mathfrak{F}_{u_k} =&\mathfrak{F}_{u_k}(\tau, u_0,u, z,  u_k,\mathcal{B}_k) := \frac{\underline{\mathfrak{G}}}{\tau} \mathscr{S}_{k;21}(\tau; u_k,\mathcal{B}_k) + \frac{1}{\tau \uf^{\frac{1}{2}}} \mathscr{S}_{k;22}(\tau,u_0,u,z; u_k,\mathcal{B}_k); \\
		\mathfrak{F}_{\mathcal{B}_i} =&\mathfrak{F}_{\mathcal{B}_i}(\tau, u_0,u, z, u_i,\mathcal{B}_k) := \frac{1}{\tau \uf^{\frac{1}{2}}} \mathscr{S}_{i;42}(\tau, u_0,u, z; u_i,\mathcal{B}_k); \\
		\mathfrak{F}_{u} =&\mathfrak{F}_{u}(\tau, u_0,u, z) := \frac{\underline{\mathfrak{G}}}{\tau} \mathscr{S}_{31}(\tau;u_0,u,z) + \frac{1}{\tau \uf^{\frac{1}{2}}} \mathscr{S}_{32}(\tau; u_0, u, z); \\
		\mathfrak{F}_z =&\mathfrak{F}_{z}(\tau; u, \mathcal{B}_i, z) := \frac{1}{\tau \uf^{\frac{1}{2}}} \mathscr{S}_{52}(\tau; u, \mathcal{B}_i, z) \notag \\
		=& - \frac{1}{\ell_0} \frac{1}{A \tau \uf^{\frac{1}{2}}} \biggl( \frac{\underline{\chi_\uparrow}}{B} \biggr)^{\frac{1}{2}} \Biggl( \frac{(1+z)^{\ell_0(1-\ca)+1} (1+\frac{\uf}{1+\uf} u)^{\cc-1}}{(1+ u)} - 1 - z - \frac{\uf}{(1+\uf)} \frac{\mathtt{d}^i \mathcal{B}_i}{(1+z)^{\ell_0-1}} \Biggr).
	\end{align*}  
	with $\mathscr{Z}^j_{12}(\tau; u_0,u,z; 0,0) = 0$, $\mathscr{Z}_{1\ell}(\tau; 0,0,0) = 0$ for $\ell=1,3,5$, and the notation explained in \S \ref{s:rmdrs}.
\end{lemma}

\subsection{Domain of Influence of inhomogeneous Initial Data}\label{s:DoI}
This section gives the representation of the inhomogeneous domain in various coordinate systems. Following \S $3.2$ of \cite{Liu2024}, we redo the calculations due to different parameter choices. We first present the $(\tau,\zeta)$ form, which can be directly converted to $(\ttau,\txi)$ and $(\htau,\hat{\zeta})$. For simplicity, we set $\ck=1$ and $\mathtt{d}^i=0$.

\begin{lemma}\label{t:char1}
	If equation \eqref{e:mainsys1} admits a classical solution $U$ in some spacetime region, with initial data satisfying $\supp U(\zeta)|_{\tau=-1}=B_1(0)$, then the domain of influence of $B_1(0)$ is contained within (or on the boundary of) the characteristic cone $\underline{\mathcal{C}}$, where the hypersurface $\underline{\mathcal{C}}$ is given by
	\begin{equation*}
		|\zeta| = 1 + \frac{1}{A} \int_{-1}^{\tau} \frac{1}{-s} \sqrt{S} \frac{\underline{\chi_\uparrow}}{B} ds = 1 - \frac{2b}{A} \ln(-\tau) + \frac{1}{A} \int_{-1}^{\tau} \frac{\Xi(s)}{-s} ds,
	\end{equation*}
	where
	\begin{equation}\label{e:bdef!}
		b := \sqrt{9\cb^2 m^2 - 6\cb m^2 + 6\cb}
	\end{equation}
	and
	\begin{equation*}
		\Xi := 2\sqrt{\left(1+\frac{\underline{\mathfrak{G}}}{6\cb B}\right)\left(b^2 + \frac{3\cb \cm^2}{2} \frac{\underline{\mathfrak{G}}}{B} + \frac{6\cb(1-\cm^2)}{\uf}\right)} - 2b \geq 0,
	\end{equation*}
	with $\lim_{\tau \to 0} \Xi(\tau) = 0$. Moreover, for $\tau \in [-1,0)$, the term $\frac{1}{A} \int_{-1}^{\tau} \frac{\Xi(s)}{-s} ds$ is non-negative and bounded. Denote $\theta := \max_{\tau \in [-1,0]} \frac{1}{A} \int_{-1}^{\tau} \frac{\Xi(s)}{-s} ds$.
\end{lemma}

\begin{proof}
The proof follows from a modification of \cite[Lemma 3.9]{Liu2024}. The only essential change is to replace equation (3.43) in \cite[Lemma 3.9]{Liu2024} by
\begin{equation*}
	\frac{d\zeta^i}{d\tau} = -\frac{1}{A\tau} \sqrt{S} \frac{\chi_\uparrow}{B} \frac{\zeta^i}{|\zeta|} = -\frac{1}{A\tau} \big(2b + \Xi(\tau)\big) \frac{\zeta^i}{|\zeta|} \quad \overset{\text{multiply } \zeta_i/|\zeta|}{\Longrightarrow} \quad \frac{d|\zeta|}{d\tau} = -\frac{1}{A\tau} \sqrt{S} \frac{\chi_\uparrow}{B} = -\frac{1}{A\tau} \big(2b + \Xi(\tau)\big).
\end{equation*}
This completes the proof.
\end{proof}

The proofs of the following corollaries are analogous to \cite[Corollary 3.1 and 3.2]{Liu2024}, so we omit the details.
\begin{corollary}
	The characteristic cone $\underline{\mathcal{C}}$ can be expressed in the $(t,x^i)$ coordinates as $\mathcal{C}$,
	\begin{equation*}
		\mathcal{C} := \left\{ (t,x) \in [t_0,t_m) \times \mathbb{R}^n \;\middle|\; |x| = 1 + \int_{t_0}^{t} \sqrt{ \cg(y,f(y),f_0(y)) } \, dy \right\},
	\end{equation*}
	which is equation \eqref{e:char1}.
\end{corollary}

Alternatively, the expression for $\mathcal{C}$ can be derived directly from the wave operator $\partial^2_t - \cg \delta^{ij} \partial_i \partial_j$, where
\[
\cg = \cg(t, \varrho, \partial_t \varrho) := \cm^2 \frac{ (\partial_t \varrho)^2}{(1+\varrho)^2} + 4(\ck-\cm^2) \frac{1+\varrho}{t^2}.
\]
This is because when $|x_0| \geq 1$, the function $-|x| + |x_0| + \int^t_{t_0} \sqrt{\cg} \, dy$ is a null function, and $x^1 = x^1_0 + \int^t_{t_0} \sqrt{\cg} \, dy$ represents a null geodesic.

\begin{corollary}\label{t:homdom}
	If $\mathcal{H}$ is defined as in \eqref{e:cah}, i.e.,
	\begin{equation*}
		\mathcal{H} := \left\{ (t,x) \in [t_0,t_m) \times \mathbb{R}^n \;\middle|\; |x| \geq 1 + \int_{t_0}^{t} \sqrt{ \cg(y,f(y),f_0(y)) } \, dy \right\},
	\end{equation*}
	then on $\mathcal{H}$, we have $\varrho = f$.
\end{corollary}

%%%--------------NEW SEC-----------------

\section{Modification and Extension to Fuchsian Form}\label{s:4} 
Since the singular equation \eqref{e:mainsys1} is not in Fuchsian form, we apply a sequence of transformations to partially convert it. Following the approach in \cite[\S1.4.2]{Liu2024}, this process yields a standard Fuchsian form while preserving the equivalence of solutions within the specified region. Unlike \cite{Liu2024}, the relaxed constraints on parameters $\ca$, $\cb$, and $\cc$ make their selection more involved. Moreover, \S \ref{s:reorg} introduces a new hypersurface $\widehat{\Gamma}_{\delta_0}$ to define the updated lens-shaped region.

\subsection{The Zoom-in Coordinate}\label{s:riv1} 
The first step in transforming equation \eqref{e:mainsys1} into a Fuchsian equation is to introduce a zoom-in coordinate system $(\ttau,\txi)$ defined by
\begin{equation}\label{e:coord5}
	\ttau = \ttau(\tau,\zeta) = \tau
	\quad \text{and} \quad
	\txi^i = \txi^i(\tau,\zeta) = \frac{\mathtt{p} \tc^i}{A} \ln(-\tau) + \zeta^i,
\end{equation}
where $\tc^i$ and $\mathtt{p}$ are given constants, and $\txi^i \in (-\infty, \infty)$. The purpose of this transformation is to make the main diagonal elements of $\mathbf{A}^i$ non-zero.

The inverse transformation of \eqref{e:coord5} is
\begin{equation}\label{e:coordi5}
	\tau = \tau(\ttau,\txi) = \ttau
	\quad \text{and} \quad
	\zeta^i = \zeta^i(\ttau,\txi) = \txi^i - \frac{\mathtt{p} \tc^i}{A} \ln(-\ttau).
\end{equation}

\begin{figure}[htbp]
	\begin{minipage}[t]{0.5\linewidth}
		\centering
		\includegraphics[width=1\textwidth]{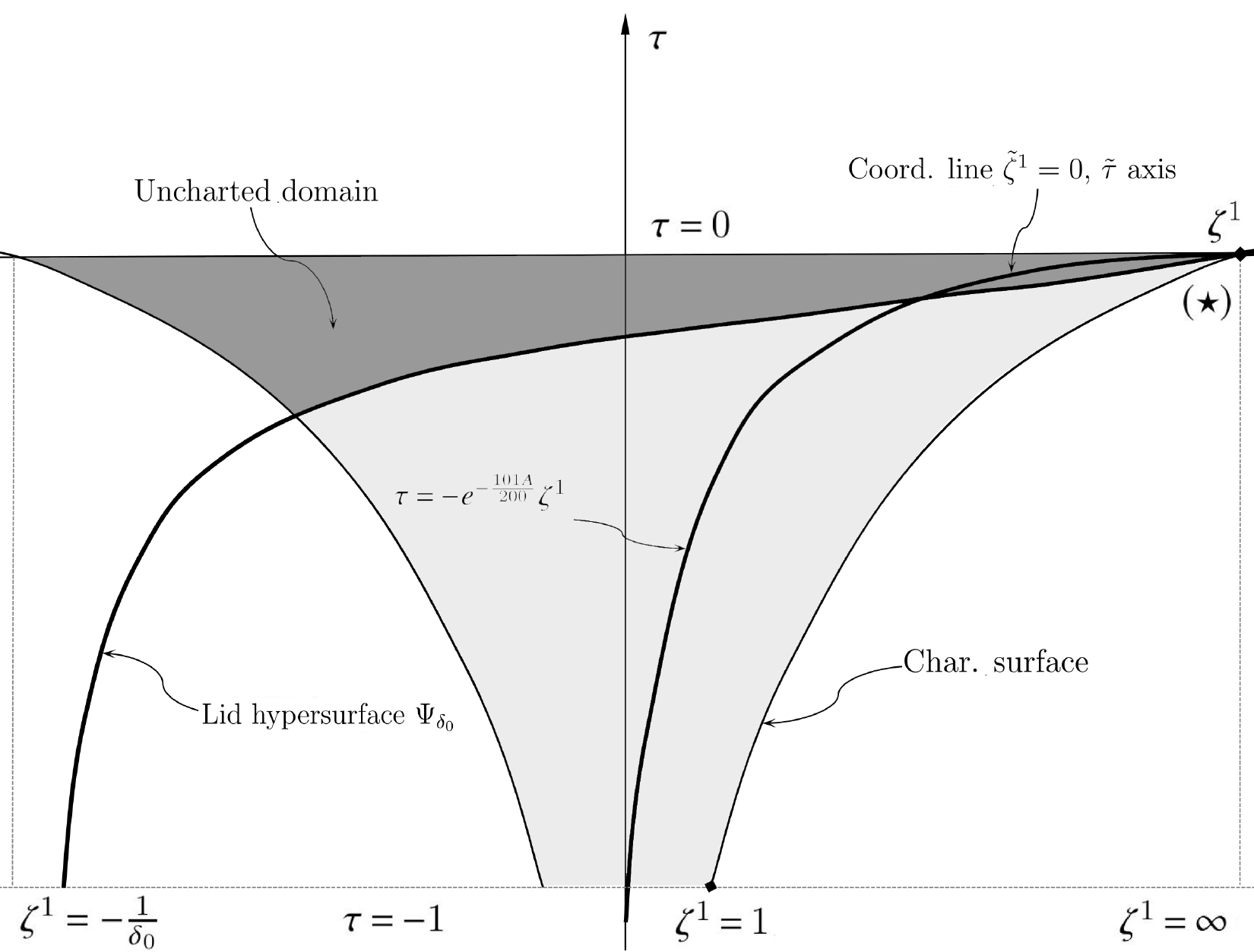}
		\caption{The compactified coordinate system $(\tau,\zeta)$ after time compactification.}
		\label{f:fig3a}
	\end{minipage}%
	\begin{minipage}[t]{0.5\linewidth}
		\centering
		\includegraphics[width=1\textwidth]{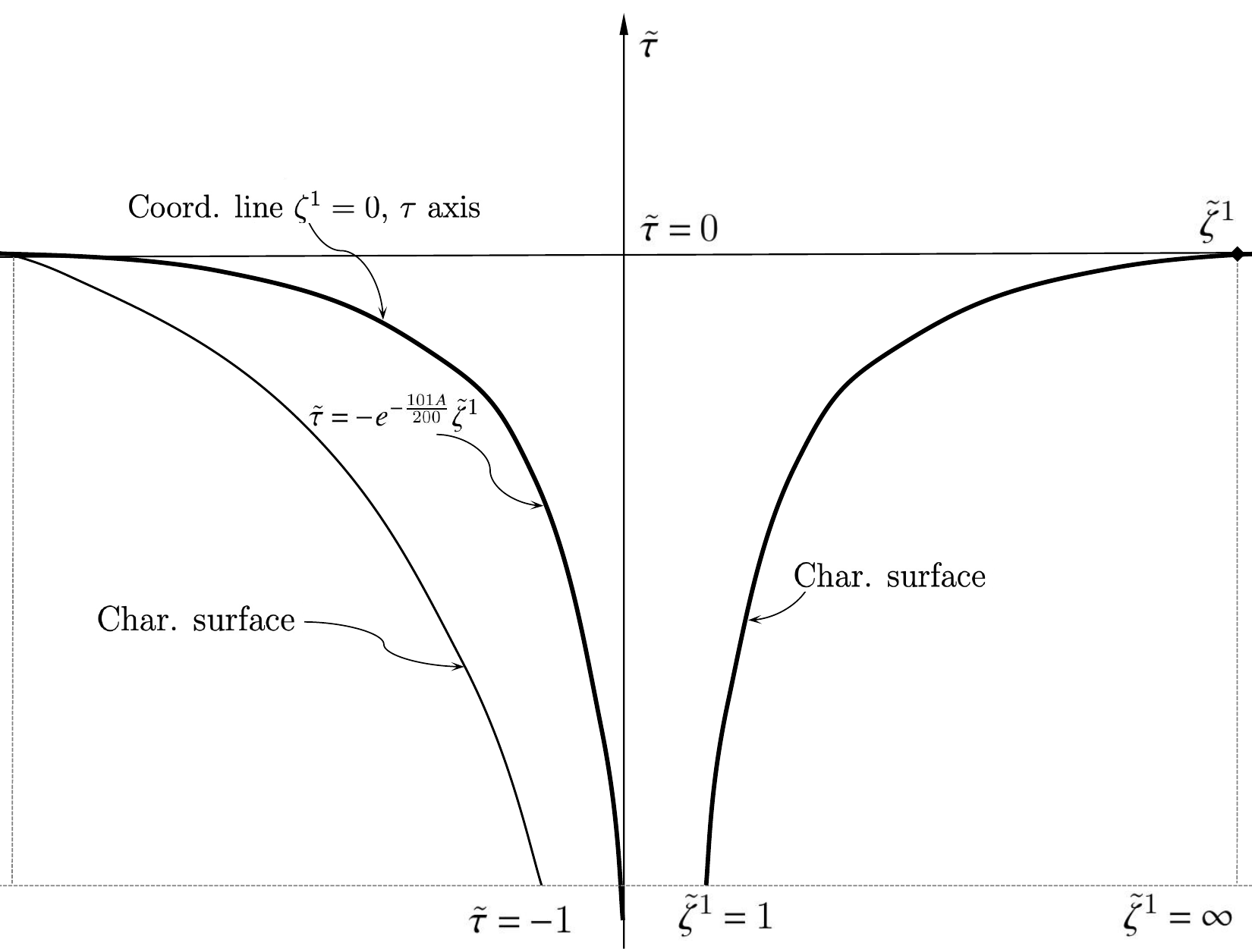}
		\caption{The zoom-in coordinate system $(\ttau,\txi)$.}
		\label{f:fig3b}
	\end{minipage}
\end{figure}

From \eqref{e:coord5} and \eqref{e:coordi5}, the Jacobian matrix of the coordinate transformation is
\begin{equation}\label{e:Jb2}
	\begin{pmatrix}
		\widetilde{\frac{\partial \ttau}{\partial \tau}} & \widetilde{\frac{\partial \ttau}{\partial \zeta^j}} \\
		\widetilde{\frac{\partial \txi^i}{\partial \tau}} & \widetilde{\frac{\partial \txi^i}{\partial \zeta^j}}
	\end{pmatrix}
	= \begin{pmatrix}
		1 & 0 \\
		\frac{\mathtt{p} \tc^i}{A \ttau} & \delta^i_j
	\end{pmatrix},
	\quad \text{and} \quad
	\det \begin{pmatrix}
		\widetilde{\frac{\partial \ttau}{\partial \tau}} & \widetilde{\frac{\partial \ttau}{\partial \zeta^j}} \\
		\widetilde{\frac{\partial \txi^i}{\partial \tau}} & \widetilde{\frac{\partial \txi^i}{\partial \zeta^j}}
	\end{pmatrix} = 1 > 0.
\end{equation}
For any function $z(\tau,\zeta)$, let $\tilde{z}(\ttau,\txi) := z(\tau(\ttau,\txi), \zeta(\ttau,\txi))$. Using the Jacobian matrix \eqref{e:Jb2}, we obtain
\begin{align}
	\widetilde{\partial_{\tau} z} &= \widetilde{\frac{\partial \ttau}{\partial \tau}} \partial_{\ttau} \tilde{z} + \widetilde{\frac{\partial \txi^i}{\partial \tau}} \partial_{\txi^i} \tilde{z} = \partial_{\ttau} \tilde{z} + \frac{\mathtt{p} \tc^i}{A \ttau} \partial_{\txi^i} \tilde{z}, \label{e:transf1} \\
	\widetilde{\partial_{\zeta^i} z} &= \widetilde{\frac{\partial \ttau}{\partial \zeta^i}} \partial_{\ttau} \tilde{z} + \widetilde{\frac{\partial \txi^j}{\partial \zeta^i}} \partial_{\txi^j} \tilde{z} = \partial_{\txi^i} \tilde{z}. \label{e:transf2}
\end{align}

\begin{lemma}\label{t:mainsys2}
	In the zoom-in coordinate system \eqref{e:coord5}, the singular equation \eqref{e:mainsys1} becomes
	\begin{equation}\label{e:mainsys2a}
		\widetilde{\mathbf{A}^0} \partial_{\ttau} \tilde{U} + \frac{1}{A \ttau} \tilde{\mathbf{A}}^{i} \partial_{\txi^i} \tilde{U} = \frac{1}{A\ttau} \widetilde{\mathbf{A}} \tilde{U} + \tilde{\mathbf{F}},
	\end{equation}
	where
	\begin{align*}
		\tilde{\mathbf{A}}^{i} &:= \mathtt{p} \tc^i \widetilde{\mathbf{A}^0} + \widetilde{\mathbf{A}^i} \\
		&= \begin{pmatrix}
			\mathtt{p} \tc^i & \mathtt{p} \tc^i \tilde{\mathscr{R}}^j + \tilde{H}^{ij} & 0 & 0 & 0 \\
			\mathtt{p} \tc^i \tilde{\mathscr{R}}_k + \tilde{H}^i_k & \mathtt{p} \tc^i (\ck + \mathscr{L} + \tau\mathscr{S}) \delta^j_k & 0 & 0 & 0 \\
			0 & 0 & 2 \mathtt{p} \tc^i & 0 & 0 \\
			0 & 0 & 0 & \mathtt{p} \tc^i \delta^l_s & 0 \\
			0 & 0 & 0 & 0 & \mathtt{p} \tc^i + \mathtt{d}^i
		\end{pmatrix}.
	\end{align*}
\end{lemma}

\begin{proof}
	This result follows directly from \eqref{e:mainsys1} and the transformation rules \eqref{e:transf1}--\eqref{e:transf2}.
\end{proof}

\subsection{The Zoom-in Variables}\label{s:riv2}
The second modification introduces a set of new zoom-in variables based on \eqref{e:v1}--\eqref{e:v5}.
\begin{align}
	\mfu_0(\ttau,\txi) = \frac{1}{\mu(\txi)} \widetilde{u_0}(\ttau,\txi)
	\quad &\Leftrightarrow \quad \widetilde{u_0}(\ttau,\txi) = \mu(\txi) \mfu_0(\ttau,\txi), \label{e:fv1} \\
	\mfu_i(\ttau,\txi) = \frac{1}{\mu(\txi)} \widetilde{u_i}(\ttau,\txi)
	\quad &\Leftrightarrow \quad \widetilde{u_i}(\ttau,\txi) = \mu(\txi) \mfu_i(\ttau,\txi), \label{e:fv2} \\
	\mfu(\ttau,\txi) = \tilde{u}(\ttau,\txi)
	\quad &\Leftrightarrow \quad \tilde{u}(\ttau,\txi) = \mfu(\ttau,\txi), \label{e:fv3} \\
	\mfv(\ttau,\txi) = \frac{1}{\mu(\txi)} \tilde{u}(\ttau,\txi)
	\quad &\Leftrightarrow \quad \tilde{u}(\ttau,\txi) = \mu(\txi) \mfv(\ttau,\txi), \label{e:fv7} \\
	\mathfrak{B}_j(\ttau,\txi^I) = \frac{1}{\mu(\txi)} \tilde{\mathcal{B}}_j(\ttau,\txi)
	\quad &\Leftrightarrow \quad \tilde{\mathcal{B}}_j(\ttau,\txi) = \mu(\txi) \mathfrak{B}_j(\ttau,\txi), \label{e:fv4} \\
	\mathfrak{z}(\ttau,\txi) = \frac{1}{\eta(\txi)} \tilde{z}(\ttau,\txi)
	\quad &\Leftrightarrow \quad \tilde{z}(\ttau,\txi) = \eta(\txi) \mathfrak{z}(\ttau,\txi), \label{e:fv6}
\end{align}
where
\begin{equation}\label{e:eta1}
	\mu(\txi) := \mu_\star e^{\omega \kappa_i \txi^i}, \quad
	\eta(\txi) := \eta_\star e^{\theta \kappa_i \txi^i}, \quad
	\Rightarrow \quad
	\frac{\partial_{\txi^i} \mu}{\mu} = \omega \kappa_i < 0, \quad
	\frac{\partial_{\txi^i} \eta}{\eta} = \theta \kappa_i,
\end{equation}
with $\mu_\star$, $\eta_\star$, $\omega > 0$, $\kappa_i < 0$ being given constants satisfying the condition
\begin{equation}\label{e:ck}
	\tc^i \kappa_i = 1.
\end{equation}

\begin{lemma}\label{t:sigsys}
	After applying the zoom-in variables \eqref{e:fv1}--\eqref{e:fv6}, equation \eqref{e:mainsys2a} can be rewritten as
	\begin{align}\label{e:mainsys3}
		\mathfrak{A}^0 \partial_{\ttau} \mathfrak{U} + \frac{1}{A\ttau} \mathfrak{A}^{i} \partial_{\txi^i} \mathfrak{U} = \frac{1}{A\ttau} \mathfrak{A} \mathfrak{U} + \mathfrak{F},
	\end{align}
	where $\mathfrak{U} := (\mfu_0, \mfu_j, \mfu, \mathfrak{B}_l, \mathfrak{z}, \mfv)^T$, and $\mathfrak{A}^0$, $\mathfrak{A}^{i}$, $\mathfrak{A}$, and $\mathfrak{F}$ are given by
	\begin{align*}
		\mathfrak{A}^0 &:= 
		\begin{pmatrix}
			1 & \mu \tilde{R} \mathfrak{B}_i \delta^{ij} & 0 & 0 & 0 & 0 \\
			\mu \tilde{R} \mathfrak{B}_k & (\tilde{S} + \widetilde{\mathscr{L}}) \delta^j_k & 0 & 0 & 0 & 0 \\
			0 & 0 & 2 & 0 & 0 & 0 \\
			0 & 0 & 0 & \delta^l_s & 0 & 0 \\
			0 & 0 & 0 & 0 & 1 & 0 \\
			0 & 0 & 0 & 0 & 0 & 2
		\end{pmatrix}, \\
		\mathfrak{A}^{i} &:= 
		\begin{pmatrix}
			\mathtt{p} \tc^i & \tilde{S} \frac{\widetilde{\underline{\chi_\uparrow}}}{B} \delta^{ij} + \widetilde{\mathscr{Z}}^{ij} & 0 & 0 & 0 & 0 \\
			(\tilde{S} \frac{\widetilde{\underline{\chi_\uparrow}}}{B} \delta^{ij} + \widetilde{\mathscr{Z}}^{ij}) \delta_{jk} & \mathtt{p} \tc^i (\tilde{S} + \widetilde{\mathscr{L}}) \delta^j_k & 0 & 0 & 0 & 0 \\
			0 & 0 & 2 \mathtt{p} \tc^i & 0 & 0 & 0 \\
			0 & 0 & 0 & \mathtt{p} \tc^i \delta^l_s & 0 & 0 \\
			0 & 0 & 0 & 0 & \mathtt{p} \tc^i + \mathtt{d}^i & 0 \\
			0 & 0 & 0 & 0 & 0 & 2 \mathtt{p}\tc^i
		\end{pmatrix},
	\end{align*}
	\begin{align*}
		\mathfrak{A} &:= 
		\begin{pmatrix}
			\mathfrak{D}_{11} + \widetilde{\mathscr{Z}_{11}} & \mathfrak{D}_{12}^j + \widetilde{\mathscr{Z}^j_{-}} & 0 & 0 & \frac{\eta}{\mu} (\mathfrak{D}_{15} + \widetilde{\mathscr{Z}_{15}}) & \mathfrak{D}_{16} + \widetilde{\mathscr{Z}_{13}} \\
			\mathfrak{D}_{21,k} + \widetilde{\mathscr{Z}_k^+} & (\mathfrak{D}_{22} + \widetilde{\mathscr{Z}}_\star)\delta^j_k & 0 & (\mathfrak{D}_{24} + \widetilde{\mathscr{Z}_{24}}) \delta^l_k & 0 & 0 \\
			\mu(\mathfrak{D}_{31} + \widetilde{\mathscr{Z}_{31}}) & 0 & \mathfrak{D}_{33} + \widetilde{\mathscr{Z}_{33}} & 0 & \eta(\mathfrak{D}_{35} + \widetilde{\mathscr{Z}_{35}}) & 0 \\
			0 & (\mathfrak{D}_{42} + \widetilde{\mathscr{Z}_{42}})\delta^j_{s} & 0 & (\mathfrak{D}_{44} + \widetilde{\mathscr{Z}_{44}}) \delta^l_s & 0 & 0 \\
			0 & 0 & 0 & 0 & \mathfrak{D}_{55} & 0 \\
			\mathfrak{D}_{31} + \widetilde{\mathscr{Z}_{31}} & 0 & 0 & 0 & \frac{\eta}{\mu} (\mathfrak{D}_{35} + \widetilde{\mathscr{Z}_{35}}) & \mathfrak{D}_{66} + \widetilde{\mathscr{Z}_{33}}
		\end{pmatrix},
	\end{align*}
	\begin{equation*}
		\mathfrak{F} := 
		\begin{pmatrix}
			- \frac{\omega}{A} \widetilde{\mathscr{S}}^j \mfu_j + \mu^{-1} \widetilde{\mathfrak{F}_{u_0}} \\
			- \frac{\omega \mathtt{p}}{A} \widetilde{\mathscr{S}} \mfu_k - \frac{\omega}{A} \widetilde{\mathscr{S}}^j \delta_{jk} \mfu_0 + \mu^{-1} \widetilde{\mathfrak{F}_{u_k}} \\
			\widetilde{\mathfrak{F}_{u}} \\
			\mu^{-1} \widetilde{\mathfrak{F}_{\mathcal{B}_i}} \\
			\eta^{-1} \widetilde{\mathfrak{F}_{z}} \\
			\mu^{-1} \widetilde{\mathfrak{F}_u}
		\end{pmatrix},
	\end{equation*}
	with
	\begin{gather}
		\widetilde{\mathscr{Z}}^{ij} := \widetilde{\mathscr{Z}}^{ij}(\ttau; \mu \mfu_0, \mu \mfv, \eta \mathfrak{z}, \mu \mathfrak{B}_i) = \frac{\mu \mathtt{p} \tc^i (\tilde{S} + \widetilde{\mathscr{L}}) \tilde{\uf}}{1 + \tilde{\uf}} \frac{\widetilde{\underline{\chi_\uparrow}}}{B} \mathfrak{B}_k \delta^{kj} + \widetilde{\mathscr{H}} (\ttau; \mu \mfu_0, \mu \mfv, \eta \mathfrak{z}) \delta^{ij}, \label{e:ZIj} \\
		\widetilde{\mathscr{Z}}^j := \kappa_i \widetilde{\mathscr{Z}}^{ij}, \quad
		\widetilde{\mathscr{S}}^j := \ck \kappa_i \delta^{ij} \frac{\tilde{\underline{\mathfrak{G}}}}{B} + \kappa_i \delta^{ij} \frac{\tilde{\underline{\chi}}_\uparrow}{B} \ttau \widetilde{\mathscr{S}}, \label{e:Zj} \\
		\widetilde{\mathscr{Z}_k^+} = - \omega \widetilde{\mathscr{Z}}^j \delta_{jk}, \quad
		\widetilde{\mathscr{Z}^j_{-}} = \widetilde{\mathscr{Z}^j_{12}} - \omega \widetilde{\mathscr{Z}}^j, \quad
		\widetilde{\mathscr{Z}}_\star = - \omega \mathtt{p} \widetilde{\mathscr{L}} + \widetilde{\mathscr{Z}_{22}}, \label{e:Z+-}
	\end{gather}
	and
	\begin{align*}
		\mathfrak{D}_{11} &= -\omega \mathtt{p} + \frac{3\cb - 4\cb\cc}{3 - 2\cc}, \\
		\mathfrak{D}_{12}^j &= \left(-4\ck + 4\cm^2 - \frac{2\cb}{3-2\cc}\cm^2\right) q^j - 4 \omega \ck \kappa_i \delta^{ij}, \\
		\mathfrak{D}_{15} &= \frac{\ell_0 \cb (3\ca - 4\cc)}{3 - 2\cc}, \quad
		\mathfrak{D}_{16} = \frac{3\cb\cc}{3 - 2\cc}, \\
		\mathfrak{D}_{21,k} &= -4 \omega \ck \kappa_i \delta^{i}_k, \quad
		\mathfrak{D}_{22} = \frac{2\cb}{3-2\cc} \ck - \ck \omega \mathtt{p}, \\
		\mathfrak{D}_{24} &= \frac{2\cb\ck}{3-2\cc}\left(\cb + \frac{2\ca\cb}{3-2\cc}\right), \quad
		\mathfrak{D}_{31} = -\frac{4\cb}{3-2\cc}, \\
		\mathfrak{D}_{33} &= \frac{4\cb(3-\cc)}{3-2\cc}, \quad
		\mathfrak{D}_{35} = -\frac{4\cb\ell_0(2-\ca)}{3-2\cc}, \\
		\mathfrak{D}_{42} &= 2 - \cc, \quad
		\mathfrak{D}_{44} = \cb - \omega \mathtt{p}, \\
		\mathfrak{D}_{55} &= -\theta (\mathtt{p} + \kappa_i \mathtt{d}^i), \quad
		\mathfrak{D}_{66} = \frac{4\cb(3-\cc)}{3-2\cc} - 2 \omega \mathtt{p},
	\end{align*}
	and the components of $\mathfrak{F}$ are given by
	\begin{align*}
		\mu^{-1} \widetilde{\mathfrak{F}_{u_0}} &= \frac{\tilde{\underline{\mathfrak{G}}}}{\ttau} \widehat{\mathscr{S}_{11}}(\ttau,\mu,\eta,\eta \mu^{-1}; \mfu_0, \mfv, \mathfrak{z}) + \frac{1}{\ttau \tilde{\uf}^{\frac{1}{2}}} \widehat{\mathscr{S}_{12}}(\ttau,\mu,\eta,\eta \mu^{-1}; \mfu_0, \mfv, \mathfrak{z}); \\
		\mu^{-1} \widetilde{\mathfrak{F}_{u_k}} &= \frac{\tilde{\underline{\mathfrak{G}}}}{\ttau} \widehat{\mathscr{S}_{k;21}}(\ttau,\mu,\eta; \mfu_i, \mathfrak{B}_k) + \frac{1}{\ttau \tilde{\uf}^{\frac{1}{2}}} \widehat{\mathscr{S}_{k;22}}(\ttau,\mu,\eta,\mu \mfu_0,\mfu,\eta \mathfrak{z}; \mfu_i, \mathfrak{B}_k); \\
		\mu^{-1} \widetilde{\mathfrak{F}_{\mathcal{B}_i}} &= \frac{1}{\ttau \tilde{\uf}^{\frac{1}{2}}} \widehat{\mathscr{S}_{l;42}}(\ttau,\mu,\eta,\mu \mfu_0,\mfu,\eta \mathfrak{z}; \mfu_i, \mathfrak{B}_k); \\
		\eta^{-1} \widetilde{\mathfrak{F}_{z}} &= \frac{1}{\ttau \tilde{\uf}^{\frac{1}{2}}} \widehat{\mathscr{S}_{52}}(\ttau,\mu\eta^{-1}; \mfv, \mathfrak{B}_i, \mathfrak{z}); \\
		\mu^{-1} \widetilde{\mathfrak{F}_u} &= \frac{\underline{\widetilde{\mathfrak{G}}}}{\ttau} \widehat{\mathscr{S}_{31}}(\ttau,\mu,\eta,\eta \mu^{-1}; \mfu_0, \mfv, \mathfrak{z}) + \frac{1}{\ttau \tilde{\uf}^{\frac{1}{2}}} \widehat{\mathscr{S}_{32}}(\ttau,\mu,\eta,\eta \mu^{-1}; \mfu_0, \mfv, \mathfrak{z}).
	\end{align*}
\end{lemma}
\begin{proof}
The proof follows the same arguments as in Lemma~4.2 of \cite{Liu2024} and is therefore omitted.
\end{proof}

\subsection{Parameter Determination and Simplification of the Computation}\label{s:revFuc}

In equation \eqref{e:mainsys3}, the functions $\mu$ and $\eta$ become singular as $\txi^1 \to -\infty$. To eliminate this singularity, we introduce a cutoff function $\phi(\txi^1)$ that removes the problematic region, an essential step toward achieving the Fuchsian form. Before this modification, suitable parameters must be chosen so that the matrix $\mathbf{B}$ in \eqref{e:Fucmodel} is positive definite and the normalized matrix $\mathbf{B}^1$ has eigenvalues with the desired signs.
We assume $\ca > 1$, $\cb > 0$, and $1 < \cc < \tfrac{3}{2}$ as required by Appendix \ref{s:ODE0}, and leave $\ell_0$ in \eqref{e:v4} undetermined. The parameters $\ck$, $\kappa_i$, $\tc$, $\omega$, $\theta$, $\mathtt{p}$, and $\mathtt{d}^i$ are then selected to satisfy the positivity and eigenvalue conditions, considering their interdependence.
Specifically, for $1 < \ca \le 30$, $\tfrac{1}{3} \le \cb \le \tfrac{2}{3}$, and $\cc = \tfrac{4}{3}$, we take
\begin{equation}\label{e:para}
	\ck = 1, \quad
	\kappa_i = -101 \delta^1_i, \quad
	\tc^i = -\tfrac{1}{101} \delta_1^i, \quad
	\omega = \theta = \tfrac{3}{2}, \quad
	\mathtt{p} = -200, \quad
	\mathtt{d}^i = 0.
\end{equation}
For $\sigma_0, \delta_0 \in (0,1)$, set $\mu_\star = \sigma_0 e^{-303/\delta_0}$ and $\eta_\star = \sigma_0^2 e^{-303/\delta_0}$. Then, from \eqref{e:eta1},
\begin{align}\label{e:eta2}
	\mu(\txi^1) &:= \sigma_0 e^{-\frac{303}{\delta_0}} e^{-\frac{303}{2} \txi^1}, \AND 
	\eta(\txi^1) := \sigma_0^2 e^{-\frac{303}{\delta_0}} e^{-\frac{303}{2} \txi^1}, \nonumber \\
	\Rightarrow \quad \frac{\partial_{\txi^i} \mu}{\mu} &= -\frac{303}{2} \delta^1_i < 0, \quad
	\frac{\partial_{\txi^i} \eta}{\eta} = -\frac{303}{2} \delta^1_i < 0.
\end{align}     
If $\txi^1 > -\tfrac{2}{\delta_0}$, then
\begin{equation}\label{e:mueta2}
	\mu(\txi^1) \le \sigma_0, \quad \eta(\txi^1) \le \sigma_0^2.
\end{equation}
Moreover, we compute $\frac{\eta}{\mu}$, which will be used in subsequent derivations,
\begin{equation*}\label{e:eonm}
	\frac{\eta}{\mu} = \frac{\sigma_0^2 e^{-\frac{303}{\delta_0}} e^{-\frac{303}{2} \txi^1}}{\sigma_0 e^{-\frac{303}{\delta_0}} e^{-\frac{303}{2} \txi^1}} = \sigma_0.
\end{equation*}

We now introduce a smooth cutoff function $\phi(\txi^1)$ satisfying the following conditions,
\begin{equation}\label{e:phi1}
	\phi \in C^\infty\bigl(\mathbb{R}; [0,1]\bigr), \quad \phi|_{[-\delta_0^{-1}, +\infty)} = 1, \quad \text{and} \quad \supp \phi \subset [-2\delta_0^{-1}, +\infty) \subset \mathbb{R}.
\end{equation}
The equation modified by this cutoff function $\phi(\txi^1)$ becomes
\begin{equation}\label{e:mainsys5}
	\mathfrak{A}^0_\phi \partial_{\ttau} \mathfrak{U} + \frac{1}{A\ttau} \mathfrak{A}^{i}_\phi \partial_{\txi^i} \mathfrak{U} = \frac{1}{A\ttau} \mathfrak{A}_\phi \mathfrak{U} + \mathfrak{F}_\phi,
\end{equation}
where
\begin{align}\label{e:A0}
	\mathfrak{A}^0_\phi := &
	\begin{pmatrix}
		1 & \phi\mu \tilde{R} \mathfrak{B}_i \delta^{ij} & 0 & 0 & 0 & 0 \\
		\phi\mu \tilde{R} \mathfrak{B}_k & (\tilde{S} + \widetilde{\mathscr{L}}_\phi) \delta^j_k & 0 & 0 & 0 & 0 \\
		0 & 0 & 2 & 0 & 0 & 0 \\
		0 & 0 & 0 & \delta^l_s & 0 & 0 \\
		0 & 0 & 0 & 0 & 1 & 0 \\
		0 & 0 & 0 & 0 & 0 & 2
	\end{pmatrix},
\end{align}
\begin{align}\label{e:Ai}
	\mathfrak{A}^{i}_\phi :=
	\begin{pmatrix}
		\frac{200}{101} \delta^i_1 & \tilde{S} \frac{\widetilde{\underline{\chi_\uparrow}}}{B} \delta^{ij} + \widetilde{\mathscr{Z}}_\phi^{ij} & 0 & 0 & 0 & 0 \\
		(\tilde{S} \frac{\widetilde{\underline{\chi_\uparrow}}}{B} \delta^{ij} + \widetilde{\mathscr{Z}}_\phi^{ij}) \delta_{jk} & \frac{200}{101} \delta^i_1 (S + \widetilde{\mathscr{L}}_\phi) \delta^j_k & 0 & 0 & 0 & 0 \\
		0 & 0 & \frac{400}{101} \delta^i_1 & 0 & 0 & 0 \\
		0 & 0 & 0 & \frac{200}{101} \delta^i_1 \delta^l_s & 0 & 0 \\
		0 & 0 & 0 & 0 & \frac{200}{101} \delta^i_1 & 0 \\
		0 & 0 & 0 & 0 & 0 & \frac{400}{101} \delta^i_1
	\end{pmatrix},
\end{align}
\begin{align}\label{e:Aph}
	\mathfrak{A}_\phi := 
	{\footnotesize
		\begin{pmatrix}
			\mathfrak{D}_{11} + \widetilde{\mathscr{Z}_{11,\phi}} & \mathfrak{D}_{12}^j + \widetilde{\mathscr{Z}^j_{-,\phi}} & 0 & 0 & \frac{\eta}{\mu} (\mathfrak{D}_{15} + \widetilde{\mathscr{Z}_{15,\phi}}) & \mathfrak{D}_{16} + \widetilde{\mathscr{Z}_{13,\phi}} \\
			\mathfrak{D}_{21,k} + \widetilde{\mathscr{Z}_{k,\phi}^+} & (\mathfrak{D}_{22} + \widetilde{\mathscr{Z}}_{\star,\phi})\delta^j_k & 0 & (\mathfrak{D}_{24} + \widetilde{\mathscr{Z}_{24,\phi}}) \delta^l_k & 0 & 0 \\
			\phi\mu(\mathfrak{D}_{31} + \widetilde{\mathscr{Z}_{31,\phi}}) & 0 & \mathfrak{D}_{33} + \widetilde{\mathscr{Z}_{33,\phi}} & 0 & \phi\eta(\mathfrak{D}_{35} + \widetilde{\mathscr{Z}_{35,\phi}}) & 0 \\
			0 & (\mathfrak{D}_{42} + \widetilde{\mathscr{Z}_{42,\phi}})\delta^j_{s} & 0 & (\mathfrak{D}_{44} + \widetilde{\mathscr{Z}_{44,\phi}}) \delta^l_s & 0 & 0 \\
			0 & 0 & 0 & 0 & \mathfrak{D}_{55} & 0 \\
			\mathfrak{D}_{31} + \widetilde{\mathscr{Z}_{31,\phi}} & 0 & 0 & 0 & \frac{\eta}{\mu} (\mathfrak{D}_{35} + \widetilde{\mathscr{Z}_{35,\phi}}) & \mathfrak{D}_{66} + \widetilde{\mathscr{Z}_{33,\phi}}
	\end{pmatrix}},
\end{align}
\begin{equation}\label{e:Fphi}
	\mathfrak{F}_\phi := 
	\begin{pmatrix}
		- \frac{3}{2A} \widetilde{\mathscr{S}}^j\mfu_j + \mu^{-1} \widetilde{\mathfrak{F}_{u_0,\phi}} \\
		\frac{300}{A} \widetilde{\mathscr{S}}(\ttau) \mfu_k - \frac{3}{2A} \widetilde{\mathscr{S}}^j\delta_{jk} \mfu_0 + \mu^{-1} \widetilde{\mathfrak{F}_{u_k,\phi}} \\
		\widetilde{\mathfrak{F}_{u,\phi}} \\
		\mu^{-1} \widetilde{\mathfrak{F}_{\mathcal{B}_i,\phi}} \\
		\eta^{-1} \widetilde{\mathfrak{F}_{z,\phi}} \\
		\mu^{-1} \widetilde{\mathfrak{F}_{u,\phi}}
	\end{pmatrix},
\end{equation}
with
\begin{align}\label{D}
	\mathfrak{D}_{11} = &300-7\cb   ,\quad	\mathfrak{D}_{12}^j=(-4+(4-6\cb)\cm^2)q^j + 606\delta^{1j}  
	, \quad \mathfrak{D}_{15}=\ell_0\cb(9\ca-16)  ,\notag \\ 
	\mathfrak{D}_{16}= & 12\cb,\quad
	\mathfrak{D}_{21,k}=  606 \delta^{1}_k  ,\quad	\mathfrak{D}_{22}= 300 + 6\cb    ,\quad \mathfrak{D}_{24}=6\cb(\cb+6\ca\cb)  ,\quad \mathfrak{D}_{31}=- 12 \cb   ,\notag\\
	\mathfrak{D}_{33}=&  20\cb     ,\quad \mathfrak{D}_{35}=-12\cb\ell_0(2-\ca) ,\quad  
	\mathfrak{D}_{42}=\frac{2}{3}   ,\quad	\mathfrak{D}_{44}=\cb +300     ,\quad \mathfrak{D}_{55}=300     
	,  \notag\\
	\mathfrak{D}_{66}= & 20\cb +  600.
\end{align}

It should be noted that $\widetilde{\mathfrak{F}_{u,\phi}}$ denotes $\widetilde{\mathfrak{F}_{u}}$ with $\mu$, $\eta$, and $\eta\mu^{-1}$ replaced by $\phi\mu$, $\phi\eta$, and $\phi\eta\mu^{-1}$, respectively. The same notation applies to $\widetilde{\mathfrak{F}_{\mathcal{B}_i,\phi}}$, $\widetilde{\mathfrak{F}_{z,\phi}}$, $\widetilde{\mathfrak{F}_{u_k,\phi}}$, $\widetilde{\mathfrak{F}_{u_0,\phi}}$, $\widetilde{\mathscr{Z}_\phi}$, and $\widetilde{\mathscr{L}}_\phi$.

\subsection{Fuchsian System on   $\mathbb{T}^n$}\label{s:revFuc1}
 As noted earlier, equation \eqref{e:mainsys5} is nearly in Fuchsian form, except that its domain is noncompact. To resolve this, we compactify the domain by introducing new coordinates $(\hat{\tau}, \hat{\zeta})$, which map $\mathbb{R}^n$ to the torus $\mathbb{T}^n_{[-\pi/2, \pi/2]}$. The torus boundary is defined by $\hat{\zeta}^i = \pm \frac{\pi}{2}$, and $\mathbb{T}^n_{[-\pi/2, \pi/2]}$  an $n$-dimensional torus determined by the endpoints of each interval $[-\frac{\pi}{2}, \frac{\pi}{2}]$.
 
 The transformed equation \eqref{e:mainsys5} can thus be formulated on this torus. As $\hat{\zeta}^1 \to \pm \frac{\pi}{2}$, both $\hat{\phi}\hat{\mu}$ and $\hat{\phi}\hat{\eta}$ vanish, yielding a Fuchsian form defined on a closed manifold. The detailed verification of this form will be presented in \S \ref{s:verfuc}.  
 
 The coordinate transformation is given by  
 \begin{equation}\label{e:coord6b}
 	\hat{\tau} = \ttau \in [-1,0), \qquad
 	\hat{\zeta}^i = \arctan(\gamma \txi^i) \in \left(-\frac{\pi}{2}, \frac{\pi}{2}\right),
 \end{equation}
 where $\gamma > 0$ is a fixed constant (see \S\ref{s:verfuc}, equation~\eqref{e:bt1}).  
 The inverse map is  
 \begin{equation}\label{e:coordi6}
 	\ttau = \hat{\tau}, \qquad
 	\txi^i = \frac{1}{\gamma} \tan \hat{\zeta}^i.
 \end{equation}

The Jacobian matrix of this transformation is
\begin{equation*}\label{e:Jb3}
	\begin{pmatrix}
		\widehat{\frac{\partial \hat{\tau}}{\partial \ttau}} & \widehat{\frac{\partial \hat{\tau}}{\partial \txi^j}} \\
		\widehat{\frac{\partial \hat{\zeta}^i}{\partial \ttau}} & \widehat{\frac{\partial \hat{\zeta}^i}{\partial \txi^j}}
	\end{pmatrix}
	= \begin{pmatrix}
		1 & 0 \\
		0 & \gamma \cos^2(\hat{\zeta}^i) \delta^i_j
	\end{pmatrix},
	\quad \text{and} \quad
	\det \begin{pmatrix}
		\widehat{\frac{\partial \hat{\tau}}{\partial \ttau}} & \widehat{\frac{\partial \hat{\tau}}{\partial \txi^j}} \\
		\widehat{\frac{\partial \hat{\zeta}^i}{\partial \ttau}} & \widehat{\frac{\partial \hat{\zeta}^i}{\partial \txi^j}}
	\end{pmatrix}
	= \gamma^n \prod_{i=1}^n \cos^2 \hat{\zeta}^i > 0.
\end{equation*}

For a function $\tilde{z}(\ttau, \txi)$, let $\hat{z}(\hat{\tau}, \hat{\zeta}) := \tilde{z}(\ttau(\hat{\tau}, \hat{\zeta}), \txi(\hat{\tau}, \hat{\zeta}))$. Using the Jacobian matrix above, we obtain
\begin{equation}\label{e:transf3}
	\widehat{\partial_{\ttau} \tilde{z}} = \widehat{\frac{\partial \hat{\tau}}{\partial \ttau}} \partial_{\hat{\tau}} \hat{z} + \widehat{\frac{\partial \hat{\zeta}^i}{\partial \ttau}} \partial_{\hat{\zeta}^i} \hat{z} = \partial_{\hat{\tau}} \hat{z},
	\quad \text{and} \quad
	\widehat{\partial_{\txi^i} \tilde{z}} = \widehat{\frac{\partial \hat{\tau}}{\partial \txi^i}} \partial_{\hat{\tau}} \hat{z} + \widehat{\frac{\partial \hat{\zeta}^j}{\partial \txi^i}} \partial_{\hat{\zeta}^j} \hat{z} = \gamma \cos^2(\hat{\zeta}^i) \partial_{\hat{\zeta}^i} \hat{z}.
\end{equation}

As $\hat{\zeta}^1$ approaches $\pi/2$ or $-\pi/2$, the quantities $\cos^2(\hat{\zeta}^i)$, $\hat{\phi}\hat{\mu}$, $\hat{\phi}\hat{\eta}$, and all their derivatives tend to zero. Therefore, this coordinate transformation is well-suited for compactifying the domain of equation \eqref{e:mainsys5} appropriately to $[-1,0) \times \mathbb{T}^n_{[-\frac{\pi}{2}, \frac{\pi}{2}]}$.

We now present the Fuchsian equation on $\mathbb{T}^n$.
\begin{equation}\label{e:mainsys6}
	\widehat{\mathfrak{A}}^0_\phi \partial_{\hat{\tau}} \widehat{\mathfrak{U}} + \frac{1}{A\hat{\tau}} \gamma \cos^2 \hat{\zeta}^i \widehat{\mathfrak{A}}^{i}_\phi \partial_{\hat{\zeta}^i} \widehat{\mathfrak{U}} = \frac{1}{A\hat{\tau}} \widehat{\mathfrak{A}}_\phi \widehat{\mathfrak{U}} + \widehat{\mathfrak{F}}_\phi,
\end{equation}
where
\begin{align}\label{e:A0b}
	\widehat{\mathfrak{A}}^0_\phi := &
	\begin{pmatrix}
		1 & \hat{\phi} \hat{\mu} \hat{R} \hat{\mathfrak{B}}_i \delta^{ij} & 0 & 0 & 0 & 0 \\
		\hat{\phi} \hat{\mu} \hat{R} \hat{\mathfrak{B}}_k & (\hat{S} + \widehat{\mathscr{L}}_\phi) \delta^j_k & 0 & 0 & 0 & 0 \\
		0 & 0 & 2 & 0 & 0 & 0 \\
		0 & 0 & 0 & \delta^l_s & 0 & 0 \\
		0 & 0 & 0 & 0 & 1 & 0 \\
		0 & 0 & 0 & 0 & 0 & 2
	\end{pmatrix},
\end{align}
\begin{align}\label{e:Aib}
	\widehat{\mathfrak{A}}^{i}_\phi :=
	\begin{pmatrix}
		\frac{200}{101} \delta^i_1 & \hat{S} \frac{\widehat{\underline{\chi_\uparrow}}}{B} \delta^{ij} + \widehat{\mathscr{Z}}_\phi^{ij} & 0 & 0 & 0 & 0 \\
		(\hat{S} \frac{\widehat{\underline{\chi_\uparrow}}}{B} \delta^{ij} + \widehat{\mathscr{Z}}_\phi^{ij}) \delta_{jk} & \frac{200}{101} \delta^i_1 (\hat{S} + \widehat{\mathscr{L}}_\phi) \delta^j_k & 0 & 0 & 0 & 0 \\
		0 & 0 & \frac{400}{101} \delta^i_1 & 0 & 0 & 0 \\
		0 & 0 & 0 & \frac{200}{101} \delta^i_1 \delta^l_s & 0 & 0 \\
		0 & 0 & 0 & 0 & \frac{200}{101} \delta^i_1 & 0 \\
		0 & 0 & 0 & 0 & 0 & \frac{400}{101} \delta^i_1
	\end{pmatrix},
\end{align}
\begin{align}\label{e:Aphb}
	\widehat{\mathfrak{A}}_\phi :=
	{\footnotesize
		\begin{pmatrix}
			\mathfrak{D}_{11} + \widehat{\mathscr{Z}_{11,\phi}} & \mathfrak{D}_{12}^{j} + \widehat{\mathscr{Z}^j_{-,\phi}} & 0 & 0 & \frac{\hat{\eta}}{\hat{\mu}} (\mathfrak{D}_{15} + \widehat{\mathscr{Z}_{15,\phi}}) & \mathfrak{D}_{16} + \widehat{\mathscr{Z}_{13,\phi}} \\
			\mathfrak{D}_{21,k} + \widehat{\mathscr{Z}_{k,\phi}^+} & (\mathfrak{D}_{22} + \widehat{\mathscr{Z}}_{\star,\phi})\delta^j_k & 0 & (\mathfrak{D}_{24} + \widehat{\mathscr{Z}_{24,\phi}}) \delta^l_k & 0 & 0 \\
			\hat{\phi}\hat{\mu}(\mathfrak{D}_{31} + \widehat{\mathscr{Z}_{31,\phi}}) & 0 & \mathfrak{D}_{33} + \widehat{\mathscr{Z}_{33,\phi}} & 0 & \hat{\phi}\hat{\eta}(\mathfrak{D}_{35} + \widehat{\mathscr{Z}_{35,\phi}}) & 0 \\
			0 & (\mathfrak{D}_{42} + \widehat{\mathscr{Z}_{42,\phi}})\delta^j_{s} & 0 & (\mathfrak{D}_{44} + \widehat{\mathscr{Z}_{44,\phi}}) \delta^l_s & 0 & 0 \\
			0 & 0 & 0 & 0 & \mathfrak{D}_{55} & 0 \\
			\mathfrak{D}_{31} + \widehat{\mathscr{Z}_{31,\phi}} & 0 & 0 & 0 & \frac{\hat{\eta}}{\hat{\mu}} (\mathfrak{D}_{35} + \widehat{\mathscr{Z}_{35,\phi}}) & \mathfrak{D}_{66} + \widehat{\mathscr{Z}_{33,\phi}}
	\end{pmatrix}},
\end{align}
\begin{equation}\label{e:Fphib}
	\widehat{\mathfrak{F}}_\phi :=
	\begin{pmatrix}
		- \frac{3}{2A} \widehat{\mathscr{S}}^j \hat{\mfu}_j + \hat{\mu}^{-1} \widehat{\mathfrak{F}_{u_0,\phi}} \\
		\frac{300}{A} \widehat{\mathscr{S}} \hat{\mfu}_k - \frac{3}{2A} \widehat{\mathscr{S}}^j \delta_{jk} \hat{\mfu}_0 + \hat{\mu}^{-1} \widehat{\mathfrak{F}_{u_k,\phi}} \\
		\widehat{\mathfrak{F}_{u,\phi}} \\
		\hat{\mu}^{-1} \widehat{\mathfrak{F}_{\mathcal{B}_i,\phi}} \\
		\hat{\eta}^{-1} \widehat{\mathfrak{F}_{z,\phi}} \\
		\hat{\mu}^{-1} \widehat{\mathfrak{F}_{u,\phi}}
	\end{pmatrix},
\end{equation}
and the constants $\mathfrak{D}_{ij}$ are given in equation \eqref{D}.

For future reference, we denote the equation without the cutoff function $\phi$ (i.e., equation \eqref{e:mainsys3}) as
\begin{equation}\label{e:mainsys8}
	\widehat{\mathfrak{A}}^0 \partial_{\hat{\tau}} \widehat{\mathfrak{U}} + \frac{1}{A\hat{\tau}} \gamma \cos^2 \hat{\zeta}^i \widehat{\mathfrak{A}}^{i} \partial_{\hat{\zeta}^i} \widehat{\mathfrak{U}} = \frac{1}{A\hat{\tau}} \widehat{\mathfrak{A}} \widehat{\mathfrak{U}} + \widehat{\mathfrak{F}}.
\end{equation}

After the above construction, we have the following result for the final singular equation \eqref{e:mainsys6} (the detailed proof is given in \S \ref{s:verfuc}).

\begin{proposition}\label{t:verfuc}
	Let $k \in \mathbb{Z}_{> \frac{n}{2} + 3}$, $\widehat{\mathfrak{U}}|_{\hat{\tau} = -1} = \widehat{\mathfrak{U}}_0$, and $\widehat{\mathfrak{U}}_0 \in H^{k}(\mathbb{T}^n_{[-\frac{\pi}{2}, \frac{\pi}{2}]})$. Then the following hold:
	\begin{enumerate}[label=(\arabic*)]
		\item\label{t:verfuc.1} Equation \eqref{e:mainsys6} is a Fuchsian equation as defined in Appendix \ref{s:fuch}.
		\item\label{t:verfuc.2} If there exists a constant $0 < \sigma < \sigma_1$ such that
	$	\|\widehat{\mathfrak{U}}_0\|_{H^k} \leq \sigma$, 
		then the initial value problem for equation \eqref{e:mainsys6} has a unique solution
		\begin{equation}\label{e:solreg}
			\widehat{\mathfrak{U}} \in C^0([-1,0), H^k(\mathbb{T}^n_{[-\frac{\pi}{2},\frac{\pi}{2}]})) \cap C^1([-1,0), H^{k-1}(\mathbb{T}^n_{[-\frac{\pi}{2},\frac{\pi}{2}]})) \cap L^\infty([-1,0), H^k(\mathbb{T}^n_{[-\frac{\pi}{2},\frac{\pi}{2}]})),
		\end{equation}
		with $\widehat{\mathfrak{U}}|_{\hat{\tau} = -1} = \widehat{\mathfrak{U}}_0$. Moreover, for all $-1 \leq \hat{\tau} < 0$, the solution $\widehat{\mathfrak{U}}$ satisfies the energy estimate
		\begin{equation*}
			\|\widehat{\mathfrak{U}}(\hat{\tau})\|_{H^k(\mathbb{T}^n_{[-\frac{\pi}{2},\frac{\pi}{2}]})}^2 - \int^{\hat{\tau}}_{-1} \frac{1}{s} \| \widehat{\mathfrak{U}}(s)\|^2_{H^k(\mathbb{T}^n_{[-\frac{\pi}{2},\frac{\pi}{2}]})} ds \leq C(\sigma_1, \sigma_1^{-1}) \|\widehat{\mathfrak{U}}_0\|^2_{H^k(\mathbb{T}^n_{[-\frac{\pi}{2},\frac{\pi}{2}]})}.
		\end{equation*}
	\end{enumerate}
\end{proposition}

\subsection{Construction of the Lens-Shaped Region}\label{s:reorg}

\begin{figure}[h]
	\centering
	\includegraphics[width=10cm]{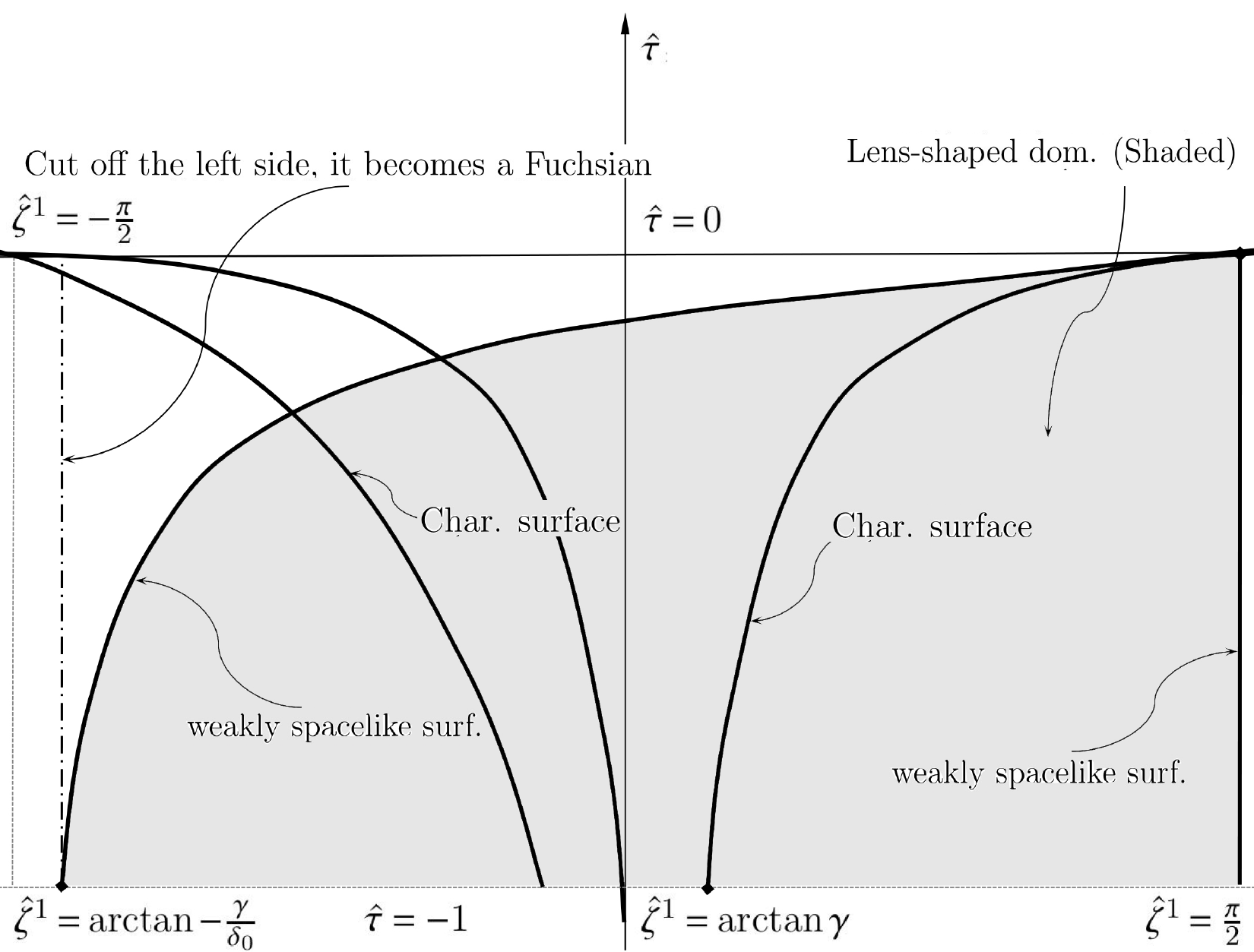}
	\caption{The region $\widehat{\mathit{\Lambda}_{\delta_0}}$ and the characteristic surface $\widehat{\mathcal{C}}$}
	\label{f:fig1}
\end{figure}

Due to the cutoff function $\phi(\txi^1)$ introduced earlier, the solution of the modified equation \eqref{e:mainsys6} no longer coincides with that of \eqref{e:mainsys8}. To restore this equivalence, we construct a lens-shaped region where both solutions agree, applying Theorem~4.5 in \cite{Lax2006}.

Specifically, we define a subregion of $[-1,0) \times \mathbb{T}^n_{[-\pi/2, \pi/2]}$ that is lens-shaped and intersects the singular set $\mathcal{I}$ at the blow-up point $p_m$ (see Theorem~\ref{t:mainthm2}). Following the method in \cite{Liu2024}, we build such a region as shown in Figure~\ref{f:fig1}. The key step is to construct a weakly spacelike hypersurface $\widehat{\Gamma}_{\delta_0}$, defined in \eqref{e:surf0}, which forms the boundary of the region. The enclosed domain is thus lens-shaped, and by Theorem~4.5 in \cite{Lax2006}, the solutions of \eqref{e:mainsys6} and \eqref{e:mainsys8} coincide within it.

Let $\Lambda_{\delta_0}$ (the shaded area in Figure~\ref{f:fig1}) denote this lens-shaped region. It is bounded by weakly spacelike hypersurfaces, with its upper boundary intersecting the characteristic cone $\mathcal{C}$ defined in \eqref{e:char1} at the future endpoint $p_m$ of a null geodesic. The hypersurface is given by
\begin{equation}\label{e:surf0}
	\widehat{\Gamma}_{\delta_0}
	:= \left\{ (\hat{\tau}, \hat{\zeta}) \in [-1,0) \times \mathbb{T}^n_{[-\pi/2, \pi/2]}
	\;\middle|\;
	\widehat{\Psi}(\hat{\tau}, \hat{\zeta})
	:= \hat{\tau} - \widehat{\mathfrak{T}}_{\delta_0}(\hat{\zeta}) = 0 \right\},
\end{equation}
where
\begin{align*}
	\widehat{\mathfrak{T}}_{\delta_0}(\hat{\zeta}) :=
	\begin{cases}
		\widehat{\mathfrak{T}}^{(1)}_{\delta_0}(\hat{\zeta}),
		& -\arctan\!\left(\tfrac{\gamma}{\delta_0}\right)
		\le \hat{\zeta}^1 \le
		\arctan\!\left( \tfrac{202b - 200 - 101\Xi_0}{202\Xi_0}
		\tfrac{\gamma}{\delta_0} \right), \\[4pt]
		\widehat{\mathfrak{T}}^{(2)}_{\delta_0}(\hat{\zeta}),
		& \hat{\zeta}^1 >
		\arctan\!\left( \tfrac{202b - 200 - 101\Xi_0}{202\Xi_0}
		\tfrac{\gamma}{\delta_0} \right),
	\end{cases}
\end{align*}
and
\begin{align*}
	\widehat{\mathfrak{T}}^{(1)}_{\delta_0}(\hat{\zeta}) &= -\exp\left(- \frac{\overline{\cc} \cdot 101A}{202b - 200} \left( 1 - \frac{1}{2\delta_0 \gamma^{-1} \tan \hat{\zeta}^1 + 4} \right) \left( \frac{1}{\gamma} \tan \hat{\zeta}^1 + \frac{1}{\delta_0} \right) \right), \\
	\widehat{\mathfrak{T}}^{(2)}_{\delta_0}(\hat{\zeta}) &= -\exp\left(-\frac{101A}{202b - 200} \left(1 - \frac{1}{2\delta_0 \gamma^{-1} \tan \hat{\zeta}^1 + 4} \right) \left( \frac{1}{\gamma} \tan \hat{\zeta}^1 + \frac{1}{2\delta_0} \right) \right),
\end{align*}
with $\delta_0 > 0$ a sufficiently small constant, $\overline{\cc} \in \left(0, \left(1 + \frac{101\Xi_0}{202b - 200}\right)^{-1} \right)$, and $\Xi_0 := \Xi(t_0) = \sup_{t \in [t_0, t_m)} \Xi(t)$ (Lemma \ref{t:Thpst} shows that $\Xi(t)$ is decreasing). Note that $\frac{202b - 200 - 101 \Xi_0}{202 \Xi_0} = \frac{101b - 100}{101 \Xi_0} - \frac{1}{2} > -\frac{1}{2}$ for $1 \leq b \leq 2$.

The lens-shaped region $\widehat{\mathit{\Lambda}}_{\delta_0}$ is defined as
\begin{equation*}\label{e:Lbddef}
	\widehat{\mathit{\Lambda}}_{\delta_0} := \left\{ (\hat{\tau}, \hat{\zeta}) \in [-1,0) \times \mathbb{T}^n_{[-\frac{\pi}{2}, \frac{\pi}{2}]} \;\middle|\; \hat{\tau} \leq \widehat{\mathfrak{T}}_{\delta_0}(\hat{\zeta}) \right\},
\end{equation*}
and its boundary is given by
\begin{equation*}\label{e:ttsurf1}
	\partial \widehat{\mathit{\Lambda}}_{\delta_0} = \widehat{\Sigma}_0 \cup \widehat{\Gamma}_{\delta_0} \cup \widehat{\Sigma}_1 \cup \left( \bigcup_{k=2}^n \widehat{\Sigma}_{\pm k} \right),
\end{equation*}
where
\begin{equation*}%\label{e:ttsurf2}
	\widehat{\Sigma}_{k}:=  \left\{(\hat{\tau},\hat\zeta^i) \;\Big|\; \hat{\zeta}^k=\frac{\pi}{2}\right\}\cap \widehat{\mathit{\Lambda}}_{\delta_0} , \quad \widehat{\Sigma}_{-k}:=  \left\{(\hat{\tau},\hat\zeta^i) \;\Big|\; \hat{\zeta}^k=-\frac{\pi}{2}\right\}\cap \widehat{\mathit{\Lambda}}_{\delta_0} \quad \text{and}\quad \widehat{\Sigma}_0:=\{\hat{\tau}=-1\} \cap \widehat{\mathit{\Lambda}}_{\delta_0}  . 
\end{equation*}

It should be noted that Lemmas \ref{t:Tsurf2}--\ref{t:Tsurf5} in Appendix \ref{s:ctbdry} provide the expressions for the hypersurface $\widehat{\Gamma}_{\delta_0}$ in other coordinate systems, such as $(\ttau,\txi)$, $(\tau,\zeta)$, and $(t,x)$. Moreover, Lemma \ref{t:inI} proves that $p_m \in \overline{\mathit{\Lambda}_{\delta_0} \cap \mathcal{I}}$.

\begin{lemma}\label{t:extori2}
	Let $\widehat{\mathfrak{U}}(\hat{\tau}, \hat{\zeta})$ be a solution of the modified equation \eqref{e:mainsys6} on the spacetime $[-1,0) \times \mathbb{T}^n_{[-\frac{\pi}{2}, \frac{\pi}{2}]}$, and denote its restriction to the lens-shaped region $\widehat{\mathit{\Lambda}}_{\delta_0}$ by $\widehat{\mathfrak{U}}(\hat{\tau}, \hat{\zeta})|_{\widehat{\mathit{\Lambda}}_{\delta_0}}$. Then $\widehat{\mathfrak{U}}(\hat{\tau}, \hat{\zeta})|_{\widehat{\mathit{\Lambda}}_{\delta_0}}$ is also a solution of the original equation \eqref{e:mainsys8} in the region $\widehat{\mathit{\Lambda}}_{\delta_0}$ with parameters given in \eqref{e:para}, and this solution is uniquely determined by the initial data on $\widehat{\Sigma}_0$.
\end{lemma}
\begin{proof}
	The proof follows the argument of \cite[Lemma~4.4]{Liu2024}, with a more refined construction of the weakly spacelike hypersurface $\widehat{\Gamma}_{\delta_0}$. We omit the full details and focus only on verifying that $\widehat{\Gamma}_{\delta_0}$ is indeed weakly spacelike.
	
	From \eqref{e:surf0}, the unit normal vector to $\widehat{\Gamma}_{\delta_0}$ in coordinates $(\hat{\tau}, \hat{\zeta})$ and $(\ttau, \txi)$ is given by
$(\hat{\nu}_\Lambda)_\mu = \partial_\mu \widehat{\Psi}(\hat{\tau}, \hat{\zeta})$ and $(\nu_\Lambda)_\mu = \partial_\mu \Psi(\tilde{\tau}, \tilde{\zeta})$. 
	Using the coordinate transformation \eqref{e:transf3}, the relation between these normal vectors in the two coordinate systems can be expressed as
		\begin{equation*}
		(\hat{\nu}_{\Lambda})_0 = (\nu_{\Lambda})_0 = 1 \quad \text{and} \quad (\hat{\nu}_{\Lambda})_i = (\nu_{\Lambda})_i \frac{\sec^2(\hat{\zeta}^1)}{\gamma}.
	\end{equation*}

	According to \cite{Lax2006}, a hypersurface is weakly spacelike if on $\widehat{\Gamma}_{\delta_0}$ the following condition holds,
	\begin{equation*}
		(\hat{\nu}_\Lambda)_0 \widehat{\mathfrak{A}}_\phi^0 + (\hat{\nu}_\Lambda)_i \frac{1}{A\hat{\tau}} \gamma \cos^2 \hat{\zeta}^i \widehat{\mathfrak{A}}^{i}_\phi = \left( (\nu_\Lambda)_0 \mathfrak{A}^0_\phi + (\nu_\Lambda)_i \frac{1}{A\ttau} \mathfrak{A}^i_\phi \right)^{\widehat{}} \geq 0.
	\end{equation*}
	We now prove that $(\nu_\Lambda)_0 \mathfrak{A}^0_\phi + (\nu_\Lambda)_i \frac{1}{A\ttau} \mathfrak{A}^i_\phi \geq 0$ holds on $\widehat{\Gamma}_{\delta_0}$.
	
	By Lemma \ref{t:Tsurf2}, in the coordinates $(\ttau, \txi)$, the outward unit normal to the upper boundary $\Gamma_{\delta_0}$ of the region $\mathit{\Lambda}_{\delta_0}$ is
	\begin{align*}
		\nu_\Lambda = \left(1, A \ttau \mathtt{q} \delta^1_i \right) = 
		\begin{cases}
			\nu_\Lambda^l = \left(1, A \ttau \mathtt{q}_l \delta^1_i \right), & -\frac{1}{\delta_0} < \txi^1 \leq \frac{202b - 200 - 101 \Xi_0}{202 \Xi_0} \frac{1}{\delta_0}, \\
			\nu_\Lambda^r = \left(1, A \ttau \mathtt{q}_r \delta^1_i \right), & \txi^1 > \frac{202b - 200 - 101 \Xi_0}{202 \Xi_0} \frac{1}{\delta_0}
		\end{cases},
	\end{align*}
	where $\overline{\cc} = \left(1 + \frac{101\Xi_0}{202b - 200} \right)^{-1}$ and
	\begin{align}\label{e:qlr}
		\mathtt{q} =
		\begin{cases}
			\mathtt{q}_l = \frac{\overline{\cc} \cdot 101}{202b - 200} \left( 1 - \frac{1}{2(2 + \txi^1 \delta_0)^2} \right) \in \left[\frac{51}{4}\cc, \frac{101A}{202b - 200} \overline{\cc} \right), & -\frac{1}{\delta_0} < \txi^1 \leq \frac{202b - 200 - 101 \Xi_0}{202 \Xi_0} \frac{1}{\delta_0}, \\
			\mathtt{q}_r = \frac{101}{202b - 200} \left(1 - \frac{3}{4 (\delta_0 \txi^1 + 2)^2} \right) \in \left[17, \frac{101A}{202b - 200} \right), & \txi^1 > \frac{202b - 200 - 101 \Xi_0}{202 \Xi_0} \frac{1}{\delta_0}
		\end{cases}.
	\end{align}

On the hypersurface $\Gamma_\Lambda$, we compute
\begin{align*}
	& (\nu_\Lambda)_0 \mathfrak{A}^0 + (\nu_\Lambda)_i \frac{1}{A\ttau} \mathfrak{A}^i = \mathfrak{A}^0 + A \ttau \mathtt{q} \delta^1_i \frac{1}{A\ttau} \mathfrak{A}^i \notag \\
	= & \begin{pmatrix}
		1 + \mathtt{q} \frac{200}{101} & \mathtt{q} S \frac{\widetilde{\underline{\chi_\uparrow}}}{B} \delta^{1j} + \mathscr{K}^j & 0 & 0 & 0 & 0 \\
		\mathtt{q} S \frac{\widetilde{\underline{\chi_\uparrow}}}{B} \delta^{1j} \delta_{jk} + \mathscr{K}^j \delta_{jk} & (1 + \mathtt{q} \frac{200}{101})(S + \phi \widetilde{\mathscr{L}}) \delta^j_k & 0 & 0 & 0 & 0 \\
		0 & 0 & 2 + \mathtt{q} \frac{400}{101} & 0 & 0 & 0 \\
		0 & 0 & 0 & (1 + \mathtt{q} \frac{200}{101}) \delta^l_s & 0 & 0 \\
		0 & 0 & 0 & 0 & 1 + \mathtt{q} \frac{200}{101} & 0 \\
		0 & 0 & 0 & 0 & 0 & 2 + \mathtt{q} \frac{400}{101}
	\end{pmatrix},
\end{align*}
where $\mathscr{K}^j := \mathscr{K}^j(\tau; \mu \mfu_0, \mu \mfv, \eta \mathfrak{z}, \mu \mathfrak{B}_i) = \mathtt{q} \phi \widetilde{\mathscr{Z}}^{1j} + \phi \mu \tilde{R} \mathfrak{B}_i \delta^{ij}$.

From Proposition \ref{t:verfuc}, we have $\|\widehat{\mathfrak{U}}\|_{L^\infty} \leq C\sigma$. Using Lemma \ref{t:Gest2} and Proposition \ref{t:fginv0}, together with \eqref{e:S1} and the hypersurface expression \eqref{e:surf}, there exists a constant $C > 0$ independent of $\delta_0$ such that for $\txi^1 > \frac{202b - 200 - 101 \Xi_0}{202 \Xi_0} \frac{1}{\delta_0} (> -\frac{1}{2\delta_0})$, the following estimates hold,
\begin{align}
	&\left| S - 1 - \frac{(1 - \cm^2)(4 - 6\cb)}{6\cb + \frac{\underline{\mathfrak{G}}}{B}} \right| \notag \\
	\leq & C(-\ttau)^{\frac{1}{2}} = C \exp\left(-\frac{1}{2} \cdot \frac{101A}{202b - 200} \left( 1 - \frac{1}{2\delta_0 \txi^1 + 4} \right) \left( \txi^1 + \frac{1}{2\delta_0} \right) \right) \notag \\
	< & C \exp\left(- \frac{101A}{4(202b - 200)} \left( \txi^1 + \frac{1}{2\delta_0} \right) \right) \notag \\
	= & C \exp\left(-\frac{1}{\delta_0} \frac{A}{8 \Xi_0} \right) \exp\left(- \frac{101A}{4(202b - 200)} \left( \txi^1 - \frac{202b - 200 - 101 \Xi_0}{202 \delta_0 \Xi_0} \right) \right), \label{e:Sest1a}
\end{align}
and similarly,
\begin{align}
	&\left| S \frac{\widetilde{\underline{\chi_\uparrow}}}{B} - 4 - 2\cm^2 \right| \leq C(-\ttau)^{\frac{1}{2}} \notag \\
	< & C \exp\left(-\frac{1}{\delta_0} \frac{A}{8 \Xi_0} \right) \exp\left(- \frac{101A}{4(202b - 200)} \left( \txi^1 - \frac{202b - 200 - 101 \Xi_0}{202 \delta_0 \Xi_0} \right) \right). \label{e:Sest1b}
\end{align}

Moreover, for $\txi^1 > -\frac{1}{\delta_0}$, we have the estimates,
\begin{equation}\label{e:KLest1}
	|\mathscr{K}^j| \leq C \sigma_0 e^{-\frac{303}{\delta_0}} \sigma e^{-\frac{303}{2} \txi^1} \quad \text{and} \quad |\widetilde{\mathscr{L}}| \leq C \sigma_0 e^{-\frac{303}{\delta_0}} \sigma e^{-\frac{303}{2} \txi^1},
\end{equation}
and from Lemma \ref{t:Thpst},
\begin{equation}\label{e:Sest2}
	S \frac{\chi_\uparrow}{B} = (2b + \Xi)^2 \frac{B}{\chi_\uparrow} \quad \text{and} \quad S = (2b + \Xi)^2 \left( \frac{B}{\chi_\uparrow} \right)^2.
\end{equation}

We now prove that all principal minors of the matrix $(\nu_\Lambda)_0\mathfrak{A}^0 + (\nu_\Lambda)_i \frac{1}{A\ttau} \mathfrak{A}^i$ are non-negative. First, from \eqref{e:cc}, for $\txi^1 > -\frac{1}{\delta_0}$, the following inequality holds:
\begin{align}\label{e:Xiest2}
 1 - \overline{\cc} \left( 1 + \frac{101\Xi}{202b - 200} \right) \left( 1 - \frac{1}{2(2 + \txi^1 \delta_0)^2} \right)  
	> 1 - \overline{\cc} \left( 1 + \frac{101\Xi}{202b - 200} \right)
	\geq 1 - \overline{\cc} \left( 1 + \frac{101\Xi_0}{202b - 200} \right) > 0.
\end{align}

We first verify the leading principal minors $\mathrm{D}_\ell$ for $\ell = 1, 2, \dots, 2n+4$. The diagonal elements are $2 + \mathtt{q}\frac{400}{101}$ and $1 + \mathtt{q}\frac{200}{101}$, and the remaining principal minors can be computed directly. From Proposition \ref{t:verfuc}.\ref{t:verfuc.2}, we have $\|\mathfrak{U}(\ttau)\|_{H^k} \leq C\sigma$. Choosing sufficiently small $\sigma$ and using the estimates \eqref{e:KLest1}, \eqref{e:Sest2}, and \eqref{e:Xiest2}, for $-\frac{1}{\delta_0} < \txi^1 \leq \frac{202b - 200 - 101 \Xi_0}{202 \delta_0 \Xi_0}$, we obtain
\begin{align*}
	\mathrm{D}_2 = & \left(1 + \mathtt{q}_l \frac{200}{101} \right)^2 (S + \phi \widetilde{\mathscr{L}}) - \left( \mathtt{q}_l S \frac{\widetilde{\underline{\chi_\uparrow}}}{B} + \mathscr{K}^1 \right)^2 \notag \\
	= & \left(1 + \frac{100 \overline{\cc}}{101b - 100} \left( 1 - \frac{1}{2(2 + \txi^1 \delta_0)^2} \right) \right)^2 \phi \widetilde{\mathscr{L}} \notag \\
	& + (2b + \Xi)^2 \left( \frac{B}{\chi_\uparrow} \right)^2 \left( 1 - \overline{\cc} \left( 1 + \frac{101\Xi}{202b - 200} \right) \left( 1 - \frac{1}{2(2 + \txi^1 \delta_0)^2} \right) - \frac{\chi_\uparrow}{B} \frac{\mathscr{K}^1}{(2b + \Xi)} \right) \notag \\
	& \times \left( 1 + \overline{\cc} \left( \frac{101b + 100}{101b - 100} + \frac{101\Xi}{202b - 200} \right) \left( 1 - \frac{1}{2(2 + \txi^1 \delta_0)^2} \right) + \frac{\chi_\uparrow}{B} \frac{\mathscr{K}^1}{(2b + \Xi)} \right).
\end{align*}
Using \eqref{e:KLest1}, we have
\begin{align*}
	\mathrm{D}_2 \geq & -C \sigma_0 e^{-\frac{303}{\delta_0}} \sigma e^{-\frac{303}{2} \txi^1} \notag \\
	& + \left( 1 - \overline{\cc} \left( 1 + \frac{101\Xi}{202b - 200} \right) \left( 1 - \frac{1}{2(2 + \txi^1 \delta_0)^2} \right) - C \sigma_0 e^{-\frac{303}{\delta_0}} \sigma e^{-\frac{303}{2} \txi^1} \right) \notag \\
	& \times \left( 1 + \overline{\cc} \left( \frac{101b + 100}{101b - 100} + \frac{101\Xi}{202b - 200} \right) \left( 1 - \frac{1}{2(2 + \txi^1 \delta_0)^2} \right) - C \sigma_0 e^{-\frac{303}{\delta_0}} \sigma e^{-\frac{303}{2} \txi^1} \right) \notag \\
	& \times (2b + \Xi)^2 \left( \frac{B}{\chi_\uparrow} \right)^2 \overset{\eqref{e:Xiest2}}{>} 0.
\end{align*}

Next, we estimate $\mathrm{D}_2$ for $\txi^1 > \frac{202b - 200 - 101 \Xi_0}{202 \delta_0 \Xi_0}$. We introduce a decreasing function $\mathrm{Q}$ that plays a crucial role in the estimate of $\mathrm{D}_2$.
\begin{equation*}
	\mathrm{Q}(\txi^1) := (\delta_0 \txi^1 + 2)^2 \exp\left(- \frac{101A}{4(202b - 200)} \left( \txi^1 - \frac{202b - 200 - 101 \Xi_0}{202 \delta_0 \Xi_0} \right) \right).
\end{equation*}
Direct differentiation of $\mathrm{Q}(\txi^1)$ shows that $\mathrm{Q}$ is decreasing for $\txi^1 > \frac{16(101b - 100) \delta_0 - 202 A}{101 A \delta_0}$. If $\delta_0$ satisfies $0 < \delta_0 < \frac{303}{32(101b - 100)} A$, then $\frac{16(101b - 100) \delta_0 - 202 A}{101 A \delta_0} < -\frac{1}{2 \delta_0} < \frac{202b - 200 - 101 \Xi_0}{202 \delta_0 \Xi_0}$. Hence, for $\txi^1 > \frac{202b - 200 - 101 \Xi_0}{202 \delta_0 \Xi_0}$, $\mathrm{Q}$ is decreasing and
\begin{equation}\label{e:Qest1}
	\mathrm{Q}(\txi^1) \leq \mathrm{Q}\left( \frac{202b - 200 - 101 \Xi_0}{202 \Xi_0} \frac{1}{\delta_0} \right) = \frac{(303 \Xi_0 + 202b - 200)^2}{40804 \Xi_0^2}.
\end{equation}

Similarly, if $0 < \delta_0 \leq \frac{909}{8}$, then for $\txi^1 > \frac{202b - 200 - 101 \Xi_0}{202 \delta_0 \Xi_0} > -\frac{1}{2\delta_0}$, the function $\mathrm{P}(\txi^1) = e^{-\frac{303}{2} \txi^1} (\delta_0 \txi^1 + 2)^2$ is decreasing, so
\begin{equation}\label{e:Qest2}
	\mathrm{P}(\txi^1) \leq \mathrm{P}\left( \frac{202b - 200 - 101 \Xi_0}{202 \Xi_0} \frac{1}{\delta_0} \right) = \frac{(303 \Xi_0 + 202b - 200)^2}{40804 \Xi_0^2} e^{\frac{3(101 \Xi_0 + 200 - 202b)}{4 \delta_0 \Xi_0}}.
\end{equation}

Using \eqref{e:Sest1a}--\eqref{e:Sest1b} and choosing $\delta_0$ sufficiently small, for $\txi^1 > \frac{2 - 51 \Xi_0}{102 \Xi_0} \frac{1}{\delta_0}$, we have
\begin{align*}
	\mathrm{D}_2 = & \left(1 + \mathtt{q}_r \frac{200}{101} \right)^2 (S + \phi \widetilde{\mathscr{L}}) - \left( \mathtt{q}_r S \frac{\widetilde{\underline{\chi_\uparrow}}}{B} + \mathscr{K}^1 \right)^2 \notag \\
	\geq & \left(1 + \mathtt{q}_r \frac{200}{101} \right)^2 - 4 \mathtt{q}_r^2 - C \sigma_0 \sigma \exp\left(- \frac{303}{\delta_0} - \frac{303}{2} \txi^1 \right) \notag \\
	& - C \exp\left(-\frac{1}{\delta_0} \frac{A}{8 \Xi_0} - \frac{101A}{4(202b - 200)} \left( \txi^1 - \frac{202b - 200 - 101 \Xi_0}{202 \delta_0 \Xi_0} \right) \right).
\end{align*}
Substituting \eqref{e:qlr} into the above inequality yields
\begin{align*}
	\mathrm{D}_2 \geq & \left( 1 + \frac{402}{101} \mathtt{q}_r \right) \left( 1 - \frac{1}{101b - 100} \left( 1 - \frac{3}{4(\delta_0 \txi^1 + 2)^2} \right) \right) \notag \\
	& - C \exp\left(-\frac{1}{\delta_0} \frac{A}{8 \Xi_0} - \frac{101A}{4(202b - 200)} \left( \txi^1 - \frac{202b - 200 - 101 \Xi_0}{202 \delta_0 \Xi_0} \right) \right) \notag \\
	& - C \sigma_0 \sigma \exp\left(- \frac{303}{\delta_0} - \frac{303}{2} \txi^1 \right).
\end{align*}
Since $b \geq 1$ (see its definition in \eqref{e:bdef!} and Lemma~\ref{t:Thpst}), we further obtain
\begin{align*}
	\mathrm{D}_2 \geq & \frac{3}{4(101b - 100)} \frac{1}{(\delta_0 \txi^1 + 2)^2} \left( 1 + \frac{402}{101} \mathtt{q}_r \right) - C \sigma_0 \sigma \exp\left(-\frac{303}{\delta_0} - \frac{303}{2} \txi^1 \right) \notag \\
	& - C \exp\left(-\frac{1}{\delta_0} \frac{A}{8 \Xi_0} - \frac{101A}{4(202b - 200)} \left( \txi^1 - \frac{202b - 200 - 101 \Xi_0}{202 \delta_0 \Xi_0} \right) \right).
\end{align*}
Setting $\mathtt{q}_r \geq \frac{101}{3(101b - 100)}$, we derive
\begin{align*}
	\mathrm{D}_2 \geq & \frac{1}{(\delta_0 \txi^1 + 2)^2} \left( \frac{303b + 102}{4(101b - 100)^2} - C (\delta_0 \txi^1 + 2)^2 \sigma_0 \sigma \exp\left(- \frac{303}{\delta_0} - \frac{303}{2} \txi^1 \right) \right. \notag \\
	& \left. - C \exp\left(-\frac{1}{\delta_0} \frac{A}{8 \Xi_0} \right) (\delta_0 \txi^1 + 2)^2 \exp\left(- \frac{101A}{4(202b - 200)} \left( \txi^1 - \frac{202b - 200 - 101 \Xi_0}{202 \delta_0 \Xi_0} \right) \right) \right).
\end{align*}
Using inequalities \eqref{e:Qest1} and \eqref{e:Qest2}, we finally obtain
\begin{align*}
	\mathrm{D}_2 \geq & \frac{1}{(\delta_0 \txi^1 + 2)^2} \left( \frac{303b + 102}{4(101b - 100)^2} - C \exp\left(-\frac{1}{\delta_0} \frac{A}{8 \Xi_0} \right) \frac{(303 \Xi_0 + 202b - 200)^2}{40804 \Xi_0^2} \right. \notag \\
	& \left. - C \sigma \sigma_0 \frac{(303 \Xi_0 + 202b - 200)^2}{40804 \Xi_0^2} \exp\left( \frac{3(101 \Xi_0 + 200 - 202b)}{4 \delta_0 \Xi_0} \right) \right) \geq 0.
\end{align*}

For $\ell = 2, \dots, n$, with sufficiently small $\sigma$ and $-\frac{1}{\delta_0} < \txi^1 \leq \frac{202b - 200 - 101 \Xi_0}{202 \Xi_0} \frac{1}{\delta_0}$, we compute
\begin{align*}
	\mathrm{D}_{\ell+1} = & \left( \left(1 + \mathtt{q}_l \frac{200}{101} \right)^2 (S + \phi \widetilde{\mathscr{L}}) - \left( \mathtt{q}_l S \frac{\widetilde{\underline{\chi_\uparrow}}}{B} + \mathscr{K}^1 \right)^2 - \sum_{i=2}^{\ell} (\mathscr{K}^i)^2 \right) \notag \\
	& \times \left(1 + \mathtt{q}_l \frac{200}{101} \right)^{\ell-1} (S + \phi \widetilde{\mathscr{L}})^{\ell-1} \notag \\
	\geq & \left( \mathrm{D}_2 - C \sigma_0 e^{-\frac{303}{\delta_0}} \sigma e^{-\frac{303}{2} \txi^1} \right) \left(1 + \frac{100 \overline{\cc}}{101b - 100} \left( 1 - \frac{1}{2(2 + \txi^1 \delta_0)^2} \right) \right)^{\ell-1} \notag \\
	& \times \left( (2b + \Xi)^2 \left( \frac{B}{\chi_\uparrow} \right)^2 - C \sigma_0 e^{-\frac{303}{\delta_0}} \sigma e^{-\frac{303}{2} \txi^1} \right)^{\ell-1} > 0.
\end{align*}
Similarly, for $\txi^1 > \frac{202b - 200 - 101 \Xi_0}{202 \Xi_0} \frac{1}{\delta_0}$ and sufficiently small $\delta_0$, we have
\begin{equation*}
	\mathrm{D}_{\ell+1} = \left( \mathrm{D}_2 - \sum_{i=2}^{\ell} (\mathscr{K}^i)^2 \right) \left(1 + \mathtt{q}_r \frac{200}{101} \right)^{\ell-1} (S + \phi \widetilde{\mathscr{L}})^{\ell-1} \geq 0.
\end{equation*}

Therefore, the matrix $(\nu_\Lambda)_0 \mathfrak{A}^0 + (\nu_\Lambda)_i \frac{1}{A\ttau} \mathfrak{A}^i$ is non-negative definite, proving that the hypersurface is weakly spacelike. We complete the proof. 
\end{proof}

\subsection{Verification of the Fuchsian Form}\label{s:verfuc}
This section provides a detailed proof of Proposition \ref{t:verfuc}, establishing that equation \eqref{e:mainsys6} is indeed a Fuchsian equation. To achieve this, we systematically verify conditions \ref{c:2}--\ref{c:7} from Appendix \ref{s:fuch}.

\begin{proof}[Proof of Proposition~\ref{t:verfuc}]
To prove Proposition~\ref{t:verfuc}(1), we verify that equation~\eqref{e:mainsys6} satisfies the six conditions in Section~\ref{s:fuch}. Except for Condition~\ref{c:5}, the arguments closely follow \cite[§4.6]{Liu2024}; thus, we focus on Condition~\ref{c:5} and \ref{c:7}, while the remaining verifications are omitted and can be obtained by analogy with \cite[§4.6]{Liu2024}.

	\textbf{Verification of Condition \ref{c:5}:}
	
	In essence, for all $(\hat{\tau}, \hat{\zeta}, \widehat{\mathfrak{U}}) \in [-1,0] \times \mathbb{T}^n_{[-\frac{\pi}{2},\frac{\pi}{2}]} \times B_R(\mathbb{R}^{4+2n})$, we need to find constants $\acute{\kappa}$, $\gamma_2$, and $\gamma_1$ such that
	\begin{equation}\label{e:pstv}
		\frac{1}{\gamma_1} \mathds{1} \leq B^0 \leq \frac{1}{\acute{\kappa}} \mathbf{B} \leq \gamma_2 \mathds{1} \quad \text{i.e.,} \quad \frac{1}{\gamma_1} X^T \mathds{1} X \leq X^T \widehat{\mathfrak{A}}^0_\phi X \leq \frac{1}{\acute{\kappa}} X^T \frac{1}{A} \widehat{\mathfrak{A}}_\phi X \leq \gamma_2 X^T \mathds{1} X,
	\end{equation}
	where $X := (X_1, \cdots, X_{4+2n})^T \in \mathbb{R}^{4+2n}$.
	
	First, using \eqref{e:Aph}, we rewrite $X^T \widehat{\mathfrak{A}}_\phi X$ in symmetric form
	\begin{equation}\label{e:eBe1}
		X^T \widehat{\mathfrak{A}}_\phi X = X^T \acute{\mathfrak{A}}_\phi X,
	\end{equation}
	where $\acute{\mathfrak{A}}_\phi$ is the symmetric matrix
	\begin{equation}\label{e:Aphi2}
		\begin{aligned}
			\acute{\mathfrak{A}}_\phi = &
			\begin{pmatrix}
				300 - 7\cb & \acute{\mathfrak{D}}_{12}^j & \hat{\phi}\hat{\mu}\acute{\mathfrak{D}}_{13} & 0 & \sigma_0\acute{\mathfrak{D}}_{15} & 0 \\
				\acute{\mathfrak{D}}_{21,k} & (6\cb + 300)\delta^j_k & 0 & \acute{\mathfrak{D}}_{24}\delta^l_k & 0 & 0 \\
				\hat{\phi}\hat{\mu}\acute{\mathfrak{D}}_{13} & 0 & 20\cb & 0 & \hat{\phi}\hat{\eta}\acute{\mathfrak{D}}_{35} & 0 \\
				0 & \acute{\mathfrak{D}}_{24}\delta^j_{s} & 0 & (\cb + 300)\delta^l_s & 0 & 0 \\
				\sigma_0\acute{\mathfrak{D}}_{15} & 0 & \hat{\phi}\hat{\eta}\acute{\mathfrak{D}}_{35} & 0 & 300 & \sigma_0\acute{\mathfrak{D}}_{35} \\
				0 & 0 & 0 & 0 & \sigma_0\acute{\mathfrak{D}}_{35} & 20\cb + 600
			\end{pmatrix} \\
			& + \widehat{\mathscr{Z}}_{(4+2n)\times (4+2n)}(\ttau; \hat{\phi}{\hat{\mu}} \mfu_0, \mfu, \hat{\phi}{\hat{\mu}} \mfu_i, \hat{\phi}{\hat{\eta}} \mathfrak{z}, \hat{\phi}{\hat{\mu}} \mathfrak{B}_\zeta),
		\end{aligned}
	\end{equation}
	with
\begin{equation*}\label{acute:D}
	\begin{aligned}
		\acute{\mathfrak{D}}_{12}^j=&(-2+(2-3\cb)\cm^2)q^j + 606\delta^{1j}  
		,\quad  \acute{\mathfrak{D}}_{21,k}=
		(-2+(2-3\cb)\cm^2)q^j+ 606 \delta^{1}_k ,\\ 
		\acute{\mathfrak{D}}_{13}=&- 6 \cb,\quad
		\acute{\mathfrak{D}}_{15}=\frac{(9\ca-16)\ell_0\cb}{2},\quad
		\acute{\mathfrak{D}}_{24}=3\cb(\cb+6\ca\cb)   +\frac{1}{3},\quad 	 \acute{\mathfrak{D}}_{35}=-6\cb\ell_0(2-\ca),
	\end{aligned}
\end{equation*}
	and $\widehat{\mathscr{Z}}_{(4+2n)\times (4+2n)}(\ttau; \hat{\phi}{\hat{\mu}} \mfu_0, \mfu, \hat{\phi}{\hat{\mu}} \mfu_i, \hat{\phi}{\hat{\eta}} \mathfrak{z}, \hat{\phi}{\hat{\mu}} \mathfrak{B}_\zeta)$ is a symmetric $(4+2n) \times (4+2n)$ matrix.
	
Expanding \eqref{e:Aphi2} yields
	\begin{align}\label{e:eBe2}
		X^T \acute{\mathfrak{A}}_\phi X = & (300 - 7\cb)X_1^2 + (300 + 6\cb)\sum_{\ell=2}^{n+1} X_\ell^2 + 20\cb X_{n+2}^2 + (300 + \cb)\sum_{\ell=n+3}^{2n+2} X_\ell^2 + 300 X_{2n+3}^2 \notag \\
		& + (600 + 20\cb) X_{2n+4}^2  \notag  \\
		& + 2 \cdot \left[(-2 + (2 - 3\cb)\cm^2)|q| + 606\right] X_1 X_2 - 2 \cdot 6\cb \hat{\phi}\hat{\mu} X_{n+2} X_1 \notag \\
		& + 2 \cdot \left(3\cb(\cb + 6\ca\cb) + \frac{1}{3}\right) \sum_{\ell=0}^{n-1} X_{n+3+\ell} X_{2+\ell} + 2 \cdot \frac{(9\ca - 16)\ell_0\cb}{2} \sigma_0 X_{2n+3} X_1 \notag \\
		& - 2 \cdot 6\cb\ell_0(2 - \ca)\hat{\phi} \hat{\eta} X_{2n+3} X_{n+2} - 2 \cdot 6\cb\ell_0(2 - \ca) \sigma_0 X_{2n+4} X_{2n+3} \notag \\
		& + X^T \widehat{\mathscr{Z}}_{(4+2n)\times (4+2n)} X.
	\end{align}

Next, we take
\begin{equation}\label{e:l0def}
	\ell_0 = 2 \AND 1 \leq \ca \leq \frac{16}{9}
\end{equation}
It should be noted that the choice $1 \leq \ca \leq \frac{16}{9}$ is made to simplify the calculations and clearly demonstrate the derivation process. In fact, for $\frac{16}{9} < \ca \leq 30$, constants $\acute{\kappa}$, $\gamma_2$, and $\gamma_1$ can also be found to satisfy the positive definiteness condition \ref{c:5}. The derivation for $\frac{16}{9} < \ca \leq 30$ is very similar to that for $1 \leq \ca \leq \frac{16}{9}$, with only minor computational differences. Here, we present the calculation process for the case $1 \leq \ca \leq \frac{16}{9}$ as an example.

Since $\widehat{\mathscr{Z}}_{(4+2n)\times (4+2n)}(\ttau; 0) = 0$ and by continuity, there exists a constant $\hat{R} > 0$ such that $\mathfrak{U} \in B_{\hat{R}}(\mathbb{R}^{4+2n})$. Then, for sufficiently small $\sigma_0$, using \eqref{e:eBe2}, \eqref{e:mueta2}, Assumption \ref{A:3} (i.e., $q^i = |q|\delta^i_1$ with $|q| = \frac{606}{2 - (2 - 3\cb)\cm^2} \in [303, \frac{1212}{4 - 3m^2}]$), Young's inequality for products (i.e., $2ab \geq -\frac{a^2}{p} - p b^2$), and the Cauchy–Schwarz inequality, we can compute a lower bound for $X^T \acute{\mathfrak{A}}_\phi X$,
\begin{align*}
	X^T \acute{\mathfrak{A}}_\phi X \geq & \left(300 - 7\cb - 6\cb\hat{\phi}\hat{\mu} + \cb(9\ca - 16)\sigma_0 \right) X_1^2 + \left(300 + 6\cb - \left(3\cb(\cb + 6\ca\cb) + \frac{1}{3}\right) r \right) X_2^2 \notag \\
	& + \sum_{\ell=3}^{n+1} \left(300 + 6\cb - \left(3\cb(\cb + 6\ca\cb) + \frac{1}{3}\right) \right) X_\ell^2 \\
	& + \left(20\cb - 6\cb\hat{\phi}\hat{\mu} - 12\cb(2 - \ca)\hat{\phi}\hat{\eta} \right) X_{n+2}^2 \notag \\
	& + \left(300 + \cb - \frac{\left(3\cb(\cb + 6\ca\cb) + \frac{1}{3} \right)}{r} \right) X_{n+3}^2 \\
	& + \left(300 + \cb - \left(3\cb(\cb + 6\ca\cb) + \frac{1}{3}\right) \right) \sum_{\ell=n+4}^{2n+2} X_\ell^2 \notag \\
	& + \left(300 + \cb(9\ca - 16)\sigma_0 - 12\cb(2 - \ca)\hat{\phi}\hat{\eta} - 12\cb(2 - \ca) \sigma_0 \right) X_{2n+3}^2 \notag \\
	& + \left(600 + 20\cb - 12\cb(2 - \ca) \sigma_0 \right) X_{2n+4}^2 - C\hat{R} \sum_{\ell=1}^{2n+4} X_\ell^2.
\end{align*}

Further simplifying the above inequality using $1 \leq \ca \leq \frac{16}{9}$ and $\frac{1}{6} \leq \cb \leq \frac{2}{3}$, and setting $r = \frac{1}{13}$, we estimate $X^T \acute{\mathfrak{A}}_\phi X$ as
\begin{align}\label{e:XAX1}
	X^T \acute{\mathfrak{A}}_\phi X > & \left(\frac{886}{3} - \frac{26}{3}\sigma_0 \right) X_1^2 + \frac{2698}{9} X_2^2 + \frac{2566}{9} \sum_{\ell=3}^{n+1} X_\ell^2 + \left(\frac{10}{3} - 12\sigma_0 \right) X_{n+2}^2 + \frac{1685}{18} X_{n+3}^2 \notag \\
	& + \frac{5117}{18} \sum_{\ell=n+4}^{2n+2} X_\ell^2 + \left(300 - \frac{62}{3}\sigma_0 \right) X_{2n+3}^2 + \left(\frac{1810}{3} - 8\sigma_0 \right) X_{2n+4}^2 - C\hat{R} \sum_{\ell=1}^{2n+4} X_\ell^2  
	>   \frac{1}{3} X^T \mathds{1} X.
\end{align}

Similarly, we compute an upper bound for $X^T \acute{\mathfrak{A}}_\phi X$,
\begin{align*}
	X^T \acute{\mathfrak{A}}_\phi X \leq & \left(300 - 7\cb + 6\cb\hat{\phi}\hat{\mu} - \cb(9\ca - 16)\sigma_0 \right) X_1^2 + \left(300 + 6\cb + \left(3\cb(\cb + 6\ca\cb) + \frac{1}{3}\right) r \right) X_2^2 \notag \\
	& + \sum_{\ell=3}^{n+1} \left(300 + 6\cb + \left(3\cb(\cb + 6\ca\cb) + \frac{1}{3}\right) \right) X_\ell^2   + \left(20\cb + 6\cb\hat{\phi}\hat{\mu} + 12\cb(2 - \ca)\hat{\phi}\hat{\eta} \right) X_{n+2}^2 \notag \\
	& + \left(300 + \cb + \frac{\left(3\cb(\cb + 6\ca\cb) + \frac{1}{3} \right)}{r} \right) X_{n+3}^2  + \left(300 + \cb + \left(3\cb(\cb + 6\ca\cb) + \frac{1}{3}\right) \right) \sum_{\ell=n+4}^{2n+2} X_\ell^2 \notag \\
	& + \left(300 - \cb(9\ca - 16)\sigma_0 + 12\cb(2 - \ca)\hat{\phi}\hat{\eta} + 12\cb(2 - \ca) \sigma_0 \right) X_{2n+3}^2 \notag \\
	& + \left(600 + 20\cb + 12\cb(2 - \ca) \sigma_0 \right) X_{2n+4}^2 + C\hat{R} \sum_{\ell=1}^{2n+4} X_\ell^2.
\end{align*}

Using the ranges of $\ca$, $\cb$ and the value of $r$, this inequality simplifies to
\begin{align}\label{e:XAX2}
	X^T \acute{\mathfrak{A}}_\phi X < & \left(\frac{1793}{6} + \frac{26}{3} \sigma_0 \right) X_1^2 + \frac{2747}{9} X_2^2 + \frac{2879}{9} \sum_{\ell=3}^{n+1} X_\ell^2 + \left(\frac{40}{3} + 12\sigma_0 \right) X_{n+2}^2 \notag \\
	& + \frac{4565}{9} X_{n+3}^2 + \frac{2849}{9} \sum_{\ell=n+4}^{2n+2} X_\ell^2 + \left(300 + \frac{62}{3}\sigma_0 \right) X_{2n+3}^2 \notag \\
	& + \left( \frac{1840}{3} + 8\sigma_0 \right) X_{2n+4}^2 + C\hat{R} \sum_{\ell=1}^{2n+4} X_\ell^2  
	<  600 X^T \mathds{1} X.
\end{align}

Furthermore, given $\cc = \frac{4}{3}$, $\ck = 1$, and $\frac{1}{6} \leq \cb \leq \frac{2}{3}$, the result from Lemma \ref{t:coef1},
\begin{equation*}
	\cm^2 + \frac{(\ck - \cm^2)(12 - 2\cb - 8\cc)}{2\cb + (3 - 2\cc)\frac{\underline{\mathfrak{G}}}{B}} < S(\tau) \leq \ck\left(1 + \frac{1}{\beta}\right) + \frac{(\ck - \cm^2)(12 - 2\cb - 8\cc)}{2\cb + (3 - 2\cc)\frac{\underline{\mathfrak{G}}}{B}},
\end{equation*}
can be further simplified to $
	\cm^2 < S(\tau) < 4 + \frac{1}{\beta}$. 
Then, from \eqref{e:A0b}, for all $\widehat{\mathfrak{U}} \in B_{\hat{R}}(\mathbb{R}^{4+2n})$, we have
\begin{equation}\label{e:XA0X}
	\frac{\cm^2}{2} X^T \mathds{1} X \leq X^T \widehat{\mathfrak{A}}^0_\phi X < \left(4 + \frac{1}{\beta} \right) X^T \mathds{1} X.
\end{equation}

Combining \eqref{e:eBe1}, \eqref{e:XAX1}, \eqref{e:XAX2}, and \eqref{e:XA0X}, we obtain
\begin{equation*}%\label{e:XAXcomp}
	\frac{\cm^2}{2} X^T \mathds{1} X \leq X^T \widehat{\mathfrak{A}}^0_\phi X < \frac{3(4\beta + 1)}{\beta} X^T \widehat{\mathfrak{A}}_\phi X \leq \frac{1800(4\beta + 1)}{\beta} X^T \mathds{1} X.
\end{equation*}

Therefore, the constants $\acute{\kappa}$, $\gamma_1$, and $\gamma_2$ in \eqref{e:pstv} are given by
\begin{equation*}
	\acute{\kappa} = \frac{\beta}{3A(4\beta + 1)}, \quad \gamma_1 = \frac{2}{\cm^2}, \quad \text{and} \quad \gamma_2 = \frac{1800(4\beta + 1)}{\beta}.
\end{equation*}

\textbf{Verification of Condition \ref{c:7}:} 
The verification of Condition~\ref{c:7} follows almost the same argument as in \cite[§4.6(4)]{Liu2024}; details are omitted, and only the key steps are highlighted below.

Since $\mathbf{P}^\perp = 0$, the terms $\mathbf{P} \mathrm{div} B \mathbf{P}^\perp$, $\mathbf{P}^\perp \mathrm{div} B \mathbf{P}$, and $\mathbf{P}^\perp \mathrm{div} B \mathbf{P}^\perp$ are identically zero, which trivially satisfies conditions \eqref{e:PhP2}--\eqref{e:PhP4}. Given $\mathbf{P} = \mathds{1}$ and $\mathbf{P}^\perp = 0$, Condition \ref{c:7} simplifies to: there exist constants $\theta$ and $\beta_\ell > 0$, $\ell = 0, 1$, such that for all $(\htau, \hat{\zeta}, \widehat{\mathfrak{U}}, \widehat{\mathfrak{W}}_i) \in [-1,0) \times \mathbb{T}^n_{[-\frac{\pi}{2},\frac{\pi}{2}]} \times B_{\tilde{R}}(\mathbb{R}^{4+2n}) \times B_{\tilde{R}}(\mathbb{M}_{(4+2n) \times n})$, the following holds, equivalent to \eqref{e:PhP1}:
\begin{equation}\label{e:divB1}
	\mathrm{div} \widehat{\mathfrak{A}}_\phi (\htau, \hat{\zeta}, \widehat{\mathfrak{U}}, \widehat{\mathfrak{W}}) = \mathcal{O}\left( \theta + |\htau|^{-\frac{1}{2}} \beta_0 + |\htau|^{-1} \beta_1 \right),
\end{equation}
where
\begin{align}\label{e:divB2}
	\mathrm{div} \widehat{\mathfrak{A}}_\phi (\htau, \hat{\zeta}, \widehat{\mathfrak{U}}, \widehat{\mathfrak{W}}) :=& \underbrace{\partial_{\htau} \widehat{\mathfrak{A}}^0_\phi(\htau, \hat{\zeta}, \widehat{\mathfrak{U}})}_{(d)\; (-\htau)^{-\frac{1}{2}} \text{ order term}} + \partial_{\widehat{\mathfrak{U}}} \widehat{\mathfrak{A}}^0_\phi(\htau, \hat{\zeta}, \widehat{\mathfrak{U}}) \cdot (\widehat{\mathfrak{A}}^0_\phi(\htau, \hat{\zeta}, \widehat{\mathfrak{U}}))^{-1} \notag \\
	&\times \Bigl[ \underbrace{-\frac{1}{A\htau} \gamma \cos^2(\hat{\zeta}^i) \widehat{\mathfrak{A}}^{i}_\phi (\htau, \hat{\zeta}, \widehat{\mathfrak{U}}) \cdot \widehat{\mathfrak{W}}_i}_{(a)\; \htau^{-1} \text{ order term}} + \underbrace{\frac{1}{A\htau} \widehat{\mathfrak{A}}_\phi (\htau, \hat{\zeta}, \widehat{\mathfrak{U}}) \widehat{\mathfrak{U}}}_{(b)\; \htau^{-1} \text{ order term}}  + \underbrace{\widehat{\mathfrak{F}}_\phi(\htau, \hat{\zeta}, \widehat{\mathfrak{U}})}_{(e)\; (-\htau)^{-\frac{1}{2}} \text{ order term}} \Bigr] \notag \\
	&+ \underbrace{\frac{1}{A\htau} \partial_{\hat{\zeta}^i} \left( \gamma \cos^2(\hat{\zeta}^i) \widehat{\mathfrak{A}}^{i}_\phi(\htau, \hat{\zeta}, \widehat{\mathfrak{U}}) \right) + \frac{\gamma \cos^2(\hat{\zeta}^i)}{A\htau} \partial_{\widehat{\mathfrak{U}}} \widehat{\mathfrak{A}}^{i}_\phi(\htau, \hat{\zeta}, \widehat{\mathfrak{U}}) \cdot \widehat{\mathfrak{W}}_i}_{(c)\; \htau^{-1} \text{ order term}}.
\end{align}

By arguments analogous to those in \cite[\S4.6(4)]{Liu2024}, we obtain that there exist constants $\tilde{\theta}$ and $\tilde{\beta}_{0}$ such that
\begin{equation}\label{e:H6.1}
	\partial_{\htau} \widehat{\mathfrak{A}}^0_\phi = \mathcal{O}\left(\tilde{\theta} + \tilde{\beta}_{0} |\htau|^{-\frac{1}{2}}\right).
\end{equation}

Through analysis, we establish the following three conclusions,

(i) $\partial_{\widehat{\mathfrak{U}}} \hat{R}$ depends on $\partial_{\widehat{\mathfrak{U}}} \hat{\mathscr{L}}$. From \eqref{e:L1} and \eqref{e:ZIj}, there exists a constant $\hat{\theta} > 0$ such that
\[
\partial_{\widehat{\mathfrak{U}}} \hat{\mathscr{L}} = \mathcal{O}(\hat{\theta}) \quad \text{and} \quad \partial_{\widehat{\mathfrak{U}}} \widehat{\mathscr{Z}^{ij}} = \mathcal{O}(\hat{\theta}).
\]

(ii) A direct computation yields
\begin{align}\label{e:dzcosA}
	\partial_{\hat{\zeta}^i} \left( \gamma \cos^2(\hat{\zeta}^i) \widehat{\mathfrak{A}}^{i}_\phi(\htau, \hat{\zeta}, \widehat{\mathfrak{U}})\right) = & -2 \gamma \cos(\hat{\zeta}^i) \sin(\hat{\zeta}^i) \widehat{\mathfrak{A}}^{i}_\phi(\htau, \hat{\zeta}, \widehat{\mathfrak{U}})  + \gamma \cos^2(\hat{\zeta}^i) \partial_{\hat{\zeta}^i} \widehat{\mathfrak{A}}^{i}_\phi(\htau, \hat{\zeta}, \widehat{\mathfrak{U}}).
\end{align}

(iii) The first term on the right-hand side of \eqref{e:dzcosA} is bounded; hence, it suffices to estimate the second term. To estimate \eqref{e:dzcosA}, we analyze $\gamma \cos^2(\hat{\zeta}^i) \partial_{\hat{\zeta}^i} \widehat{\mathscr{Z}}^{ij}_\phi$ and $\gamma \cos^2(\hat{\zeta}^i) \partial_{\hat{\zeta}^i} \widehat{\mathscr{L}}_\phi$. The key is to examine the factors $\gamma \cos^2(\hat{\zeta}^i) \partial_{\hat{\zeta}^i}(\hat{\phi}\hat{\mu})$ and $\gamma \cos^2(\hat{\zeta}^i) \partial_{\hat{\zeta}^i}(\hat{\phi}\hat{\eta})$.

Similar to \cite[\S4.6(4)]{Liu2024}, we obtain
\begin{align}\label{e:gamcos}
	\left| \gamma \cos^2(\hat{\zeta}^i) \partial_{\hat{\zeta}^i}(\hat{\phi}\hat{\mu}) \right| = \left| \hat{\mu} \widehat{\partial_{\tilde{\zeta}^i} \phi} + \hat{\phi} \widehat{\partial_{\tilde{\zeta}^i} \mu} \right| \leq \left| C\hat{\mu} - 101\delta^1_i \hat{\phi}\hat{\mu} \right| \leq C\hat{\mu} \leq C\sigma_0.
\end{align}
Thus, combining \eqref{e:L1} and \eqref{e:ZIj}, we conclude
\begin{equation*}
	|\partial_{\hat{\zeta}^i} \widehat{\mathscr{L}}| \lesssim \sigma_0 \quad \text{and} \quad |\partial_{\hat{\zeta}^i} \widehat{\mathscr{Z}^{ij}}| \lesssim \sigma_0.
\end{equation*}

Based on the above three conclusions, we directly compute $\partial_{\widehat{\mathfrak{U}}} \widehat{\mathfrak{A}}^0_\phi$, $\partial_{\widehat{\mathfrak{U}}} \widehat{\mathfrak{A}}^{i}_\phi$, and $\partial_{\hat{\zeta}^i} \widehat{\mathfrak{A}}^{i}_\phi$, and apply Lemmas C.1 and C.2 from \cite{Liu2018} to compute $(\widehat{\mathfrak{A}}^0_\phi)^{-1}$. We find that for all $\widehat{\mathfrak{U}} \in B_{\tilde{R}}(\mathbb{R}^{4+2n})$, there exist constants $\hat{\theta}$ and $\tilde{R}$ such that
\begin{gather}\label{e:H6.2}
	\partial_{\widehat{\mathfrak{U}}} \widehat{\mathfrak{A}}^0_\phi = \mathcal{O}(\hat{\theta}), \quad
	\partial_{\hat{\zeta}^i} \left( \gamma \cos^2(\hat{\zeta}^i) \widehat{\mathfrak{A}}^{i}_\phi \right) = \mathcal{O}(\hat{\theta} \sigma_0), \quad 
	\partial_{\widehat{\mathfrak{U}}} \widehat{\mathfrak{A}}^{i}_\phi = \mathcal{O}(\hat{\theta}), \quad \text{and} \quad
	(\widehat{\mathfrak{A}}^0)^{-1} = \mathcal{O}(\hat{\theta}).
\end{gather}

Furthermore, from arguments analogous to those in \cite[\S4.6(4)]{Liu2024}, there exist constants $\bar{\theta}$, $\hat{\beta}_0$, and $\tilde{R}$ such that for all $\widehat{\mathfrak{U}} \in B_{\tilde{R}}(\mathbb{R}^{4+2n})$
\begin{equation}\label{e:H6.5}
	\widehat{\mathfrak{F}}_\phi = \mathcal{O}\left( \bar{\theta} + \hat{\beta}_0 |\htau|^{-\frac{1}{2}} \right).
\end{equation}

Using \eqref{e:H6.1}--\eqref{e:H6.5}, we examine each term in \eqref{e:divB2} and conclude that there exist constants $\theta > 0$, $\beta_0 > 0$, and $\beta_1 > 0$ such that for all $(\htau, \hat{\zeta}, \widehat{\mathfrak{U}}, \widehat{\mathfrak{W}}_i) \in [-1,0) \times \mathbb{T}^n_{[-\frac{\pi}{2},\frac{\pi}{2}]} \times B_{\tilde{R}}(\mathbb{R}^{4+2n}) \times B_{\tilde{R}}(\mathbb{M}_{(4+2n) \times n})$, inequality \eqref{e:divB1} holds.

From \eqref{e:dzcosA} and \eqref{e:gamcos}, the $\htau^{-1}$-order terms $(a)$, $(b)$, and $(c)$ in \eqref{e:divB2} are controlled by $\widehat{\mathfrak{U}}$, $\widehat{\mathfrak{W}}_i$, or bounded by $\sigma_0$ and $\gamma$. By appropriately reducing the radius of the ball $B_{\tilde{R}}(\mathbb{R}^N) \times B_{\tilde{R}}(\mathbb{M}_{(4+2n) \times n}) \ni (\mathfrak{U}, \mathfrak{W}_i)$ and choosing sufficiently small $\sigma_0$, we can select a suitable $\gamma$ and set $\beta_1 = C_1 \gamma > 0$ in \eqref{e:divB1} such that
\begin{equation}\label{e:bt1}
	\beta_1 + 2k(k+1)C_0 \gamma = (C_1 + 2k(k+1)C_0)\gamma < \frac{\cm^2 \beta}{3A(4\beta + 1)},
\end{equation}
where $k \in \mathbb{Z}_{> \frac{n}{2} + 3}$ is a constant and $C_0 > 0$ is given by $\mathtt{b} \leq C_0 \gamma$ (see \cite[\S4.6(1)]{Liu2024} for details of this result). This completes the verification of Condition \ref{c:7}.

In summary, having verified Conditions \ref{c:2}--\ref{c:7}, we conclude that equation \eqref{e:mainsys3} constitutes a Fuchsian equation as defined in \S \ref{s:fuch}, thereby proving Proposition \ref{t:verfuc}.\ref{t:verfuc.1}.

Proposition \ref{t:verfuc}.\ref{t:verfuc.2} follows directly from Theorem \ref{t:fuch}. Since \eqref{e:mainsys5} is indeed in Fuchsian form, and with $\beta_1$ chosen according to \eqref{e:bt1}, the constants $\kappa$, $\gamma_1$, and $\beta_1$ satisfy \eqref{e:kpbt1} (see \cite[\S4.6]{Liu2024} for details). This completes the proof of Proposition \ref{t:verfuc}.
\end{proof}

\section{Proof of Main Theorem \ref{t:mainthm2}}\label{s:pfmthm}
This section presents the proof of Theorem~\ref{t:mainthm2}, carried out in four main steps. For simplicity, we omit parts similar to \cite[§5]{Liu2024} as well as routine calculations and minor modifications, emphasizing only the essential changes. Once Theorem~\ref{t:mainthm2} is proved, Theorem~\ref{t:mainthm1} follows immediately through exponential and logarithmic transformations and will not be discussed further. Let $\sigma_0 \in (0, \sigma_\star)$ and $\delta_0 \in (0, \delta_\star)$, where $\sigma_\star$ and $\delta_\star$ are sufficiently small positive constants.

\textbf{Step 1:} We first introduce a set of variable substitutions that will be used frequently in the subsequent analysis. Using the transformations defined in \eqref{e:v1}--\eqref{e:v5} and \eqref{e:fv1}--\eqref{e:fv6}, the coordinate transformation, 
\begin{equation}\label{e:coord6}
	\ttau = \tau = g(t, x^i) \quad \text{and} \quad \txi^i = \frac{200\delta^i_1}{101A} \ln(-\tau) + \zeta^i = \frac{200\delta^i_1}{101A} \ln(-g(t, x^i)) + x^i,
\end{equation}
along with \eqref{e:coord5}--\eqref{e:coordi5}, \eqref{e:coord2}, and \eqref{e:coordi2}, and the definitions of $\mu$ and $\eta$ in \eqref{e:eta2}, together with the relations $\mft = \mathsf{h}_\uparrow(\tau) = \mathsf{h}_\uparrow \circ g(t, x^i)$ and $f(\mft) = \uf(\tau) = \uf \circ \mfg(\mft) = f \circ \mathsf{h}_\uparrow \circ g(t, x^i)$, we obtain the following composite transformations
\begin{align}
	\hat{\mfu}_0(\htau, \hat{\zeta}) = & \mfu_0(\ttau, \txi) = \frac{e^{\frac{303}{\delta_0}} (-\tau)^{\frac{300}{A}} e^{\frac{303}{2} \zeta^1}}{\sigma_0} \frac{\underline{\varrho_0}(\tau, \zeta) - \underline{f_0}(\tau)}{\underline{f_0}(\tau)} \nonumber \\
	= & \frac{e^{\frac{303}{\delta_0}} (-g(t, x))^{\frac{300}{A}} e^{\frac{303}{2} x^1}}{\sigma_0} \frac{\varrho_0(t, x) - f_0 \circ \mathsf{h}_\uparrow \circ g(t, x)}{f_0 \circ \mathsf{h}_\uparrow \circ g(t, x)}, \label{e:fv1.b} \\
	\hat{\mfu}_i(\htau, \hat{\zeta}) = & \mfu_i(\ttau, \txi) = \frac{e^{\frac{303}{\delta_0}} (-\tau)^{\frac{300}{A}} e^{\frac{303}{2} \zeta^1}}{\sigma_0} \frac{\underline{\varrho_i}(\tau, \zeta)}{1 + \underline{f}(\tau)} \notag \\
	= & \frac{e^{\frac{153}{\delta_0}} (-g(t, x))^{\frac{100}{A}} e^{51 x^1}}{\sigma_0} \frac{\varrho_i(t, x)}{1 + f \circ \mathsf{h}_\uparrow \circ g(t, x)}, \label{e:fv2.b} \\
	\hat{\mfu}(\htau, \hat{\zeta}) = & \mfu(\ttau, \txi) = \frac{\underline{\varrho}(\tau, \zeta) - \uf(\tau)}{\uf(\tau)} = \frac{\varrho(t, x) - f \circ \mathsf{h}_\uparrow \circ g(t, x)}{f \circ \mathsf{h}_\uparrow \circ g(t, x)}, \label{e:fv3.b} \\
	\hat{\mfv}(\htau, \hat{\zeta}) = & \mfv(\ttau, \txi) = \frac{e^{\frac{303}{\delta_0}} (-\tau)^{\frac{300}{A}} e^{\frac{303}{2} \zeta^1}}{\sigma_0} \frac{\underline{\varrho}(\tau, \zeta) - \uf(\tau)}{\uf(\tau)} \notag \\
	= & \frac{e^{\frac{303}{\delta_0}} (-g(t, x))^{\frac{300}{A}} e^{\frac{303}{2} x^1}}{\sigma_0} \frac{\varrho(t, x) - f \circ \mathsf{h}_\uparrow \circ g(t, x)}{f \circ \mathsf{h}_\uparrow \circ g(t, x)}, \label{e:fv6.b} \\
	\hat{\mathfrak{B}}_j(\htau, \hat{\zeta}) = & \mathfrak{B}_j(\ttau, \txi) = \frac{e^{\frac{303}{\delta_0}} (-\tau)^{\frac{300}{A}} e^{\frac{303}{2} \zeta^1}}{\sigma_0} \frac{B \ufo(\tau)}{\underline{\chi_\uparrow}(\tau) \uf(\tau)} \mathsf{h}_j(\tau, \zeta), \label{e:fv4.b} \\
	\hat{\mathfrak{z}}(\htau, \hat{\zeta}) = & \mathfrak{z}(\ttau, \txi) = \frac{e^{\frac{303}{\delta_0}} (-\tau)^{\frac{300}{A}} e^{\frac{303}{2} \zeta^1}}{\sigma_0^2} \left[ \left( \frac{\mathsf{h}(\tau, \zeta)}{\mathsf{h}_\uparrow(\tau)} \right)^{\frac{1}{2}} - 1 \right]. \label{e:fv5.b}
\end{align}

Following arguments and calculations similar to those in \cite[\S5]{Liu2024}, we conclude that the initial data of $\widehat{\mathfrak{U}}$ satisfy
\begin{equation}\label{e:dt3}
	\|\widehat{\mathfrak{U}}_0\|_{C^3} \leq \|\widehat{\mathfrak{U}}_0\|_{H^k(B_r(0))} \leq \sigma.
\end{equation}

\textbf{Step 2:} In \S \ref{s:4}, through Lemmas \ref{t:mainsys1}, \ref{t:mainsys2}, and \ref{t:sigsys}, we transformed equation \ref{Eq2} into the singular equation \eqref{e:mainsys3}. By introducing the cutoff function $\phi$, we replaced the terms involving $\mu$ and $\eta$ in \eqref{e:mainsys3} with $\phi\mu$ and $\phi\eta$, respectively, resulting in the modified equation \eqref{e:mainsys5}. Then, using the coordinate transformations \eqref{e:coord6b} and \eqref{e:coordi6}, we compactified the spatial domain into a torus, ultimately converting the modified equation \eqref{e:mainsys5} into the standard Fuchsian form described in \S \ref{s:fuch} (see Proposition \ref{t:verfuc}).

This implies that on the spacetime $[-1,0) \times \mathbb{T}^n_{[-\frac{\pi}{2},\frac{\pi}{2}]}$, there exists a constant $\sigma = \sigma(\delta_0) > 0$ such that if the initial data satisfy $\|\widehat{\mathfrak{U}}_0\|_{H^k} \leq \sigma$, then there exists a unique solution $\widehat{\mathfrak{U}}$ satisfying the regularity conditions in Proposition \ref{t:verfuc}, with the energy estimate 
$
\|\widehat{\mathfrak{U}}(\htau)\|_{H^k(\mathbb{T}^n_{[-\frac{\pi}{2},\frac{\pi}{2}]})} \leq C \|\widehat{\mathfrak{U}}_0\|_{H^k(\mathbb{T}^n_{[-\frac{\pi}{2},\frac{\pi}{2}]})}$. 
From \eqref{e:dt3}, by choosing $\sigma_\star$ sufficiently small, there exists $\sigma > 0$ such that $\|\widehat{\mathfrak{U}}_0\|_{H^k} \leq \sigma$. Therefore, there exists a unique solution $\widehat{\mathfrak{U}}$ on $[-1,0) \times \mathbb{T}^n_{[-\frac{\pi}{2},\frac{\pi}{2}]}$ satisfying the regularity condition \eqref{e:solreg}, and for all $-1 \leq \hat{\tau} < 0$:
\begin{equation}\label{e:est1}
	\|\widehat{\mathfrak{U}}(\htau)\|_{C^3(\mathbb{T}^n_{[-\frac{\pi}{2},\frac{\pi}{2}]})} \leq C \|\widehat{\mathfrak{U}}(\htau)\|_{H^k(\mathbb{T}^n_{[-\frac{\pi}{2},\frac{\pi}{2}]})} \leq C \|\widehat{\mathfrak{U}}_0\|_{H^k(\mathbb{T}^n_{[-\frac{\pi}{2},\frac{\pi}{2}]})} \leq C\sigma.
\end{equation}

By Lemma \ref{t:extori2}, we obtain the unique solution $\widehat{\mathfrak{U}} := (\hat{\mfu}_0, \hat{\mfu}_j, \hat{\mfu}, \hat{\mathfrak{B}}_l, \hat{\mathfrak{z}}, \hat{\mfv})^T$ of equation \eqref{e:mainsys3} with parameters \eqref{e:para} in the region $\widehat{\mathit{\Lambda}}_{\delta_0}$. Applying the Sobolev embedding theorem, the estimate \eqref{e:est1} becomes
\begin{equation}\label{e:est2}
	\|(\hat{\mfu}_0, \hat{\mfu}_j, \hat{\mfu}, \hat{\mathfrak{B}}_l, \hat{\mathfrak{z}}, \hat{\mfv})\|_{C^3(\mathit{\Lambda_{\delta_0}})} = \|\widehat{\mathfrak{U}}(\ttau)\|_{C^3(\mathit{\Lambda_{\delta_0}})} \leq C\sigma.
\end{equation}

The following lemma controls the decay factor via simple exponential decay in $\zeta$:
\begin{lemma}\label{t:dcy}
	If $(\tau, \zeta) \in \underline{\mathit{\Lambda}_{\delta_0}}$, then the decay factor satisfies
	\begin{equation*}
		e^{-\frac{303}{\delta_0}} (-g(t,x))^{-\frac{300}{A}} e^{-\frac{303}{2} x^1} = e^{-\frac{303}{\delta_0}} (-\tau)^{-\frac{300}{A}} e^{-\frac{303}{2} \zeta^1} < e^{-\frac{100}{\delta_0}} e^{-\frac{3}{4} \zeta^1}.
	\end{equation*}
\end{lemma}
\begin{proof}
	If $(\tau, \zeta) \in \underline{\mathit{\Lambda}_{\delta_0}}$, then for all $\zeta$, $(-\tau)^{-\frac{300}{A}}$ attains its supremum on the hypersurface $\hat{\Gamma}_{\delta_0}$. Substituting the expression for $\tau$ from Lemma \ref{t:Tsurf3}, and since $\sqrt{\Diamond_r(\zeta^1)} > 0$ and $\sqrt{\Diamond_l(\zeta^1)} > 0$, the result follows by direct computation.
\end{proof}

Now we consider the original variables in the region $\mathit{\Lambda}_{\delta_0}$. For $(t, x) \in \mathit{\Lambda}_{\delta_0}$, using \eqref{e:fv1.b}--\eqref{e:fv5.b} and Lemma \ref{t:dcy}, estimate \eqref{e:est2} implies
\begin{align}
	(1 - C \sigma_0^2 e^{-\frac{100}{\delta_0}} e^{-\frac{3x^1}{4}}) f_0 \circ \mathsf{h}_\uparrow \circ g \leq \varrho_0 &\leq (1 + C \sigma_0^2 e^{-\frac{100}{\delta_0}} e^{-\frac{3x^1}{4}}) f_0 \circ \mathsf{h}_\uparrow \circ g, \\
	-C \sigma_0^2 e^{-\frac{100}{\delta_0}} e^{-\frac{3x^1}{4}} (1 + f \circ \mathsf{h}_\uparrow \circ g) \leq \varrho_i &\leq C \sigma_0^2 e^{-\frac{100}{\delta_0}} e^{-\frac{3x^1}{4}} (1 + f \circ \mathsf{h}_\uparrow \circ g), \\
	(1 - C \sigma_0) f \circ \mathsf{h}_\uparrow \circ g \leq \varrho &\leq (1 + C \sigma_0) f \circ \mathsf{h}_\uparrow \circ g, \label{e:rhoreg0} \\
	(1 - C \sigma_0^2 e^{-\frac{100}{\delta_0}} e^{-\frac{3x^1}{4}}) f \circ \mathsf{h}_\uparrow \circ g \leq \varrho &\leq (1 + C \sigma_0^2 e^{-\frac{100}{\delta_0}} e^{-\frac{3x^1}{4}}) f \circ \mathsf{h}_\uparrow \circ g, \\
	-\frac{\chi_\uparrow}{B} \frac{f}{f_0} C \sigma_0^2 e^{-\frac{100}{\delta_0}} e^{-\frac{3x^1}{4}} \leq \mathsf{h}_j &\leq \frac{\chi_\uparrow}{B} \frac{f}{f_0} C \sigma_0^2 e^{-\frac{100}{\delta_0}} e^{-\frac{3x^1}{4}}, \\
	(1 - C \sigma_0^3 e^{-\frac{100}{\delta_0}} e^{-\frac{3x^1}{4}})^2 \mathsf{h}_\uparrow \leq \mathsf{h} &\leq (1 + C \sigma_0^3 e^{-\frac{100}{\delta_0}} e^{-\frac{3x^1}{4}})^2 \mathsf{h}_\uparrow. \label{e:breg0}
\end{align}

Moreover, from Lemma \ref{t:Tsurf5} and \eqref{e:breg0}, we can further estimate the hypersurface (where $\tau_\Gamma$ is defined in Lemma \ref{t:Tsurf3}),
\begin{align}\label{e:lidest}
	(1 - C \sigma_0^3 e^{-\frac{100}{\delta_0}} e^{-\frac{3x^1}{4}})^2 \mathsf{h}_\uparrow(\tau_\Gamma(x; \delta_0)) \leq \mathsf{h}(\tau_\Gamma(x; \delta_0), x) \leq (1 + C \sigma_0^3 e^{-\frac{100}{\delta_0}} e^{-\frac{3x^1}{4}})^2 \mathsf{h}_\uparrow(\tau_\Gamma(x; \delta_0)).
\end{align}
Note that
\begin{gather*}
	\lim_{x^1 \to +\infty} \tau_\Gamma(x; \delta_0) = 0 \quad \text{and} \quad \lim_{x^1 \to +\infty} e^{-\frac{3x^1}{4}} = 0; \\
	\lim_{\delta_0 \to 0^+} \tau_\Gamma(x; \delta_0) = 0 \quad \text{and} \quad \lim_{\delta_0 \to 0^+} e^{-\frac{100}{\delta_0}} = 0.
\end{gather*}
Taking the limit $x^1 \to +\infty$ in \eqref{e:lidest}, and using \eqref{e:ctm2} and $\mathfrak{g}(t_m) = 0$ from Lemma \ref{t:gb2}.\ref{l:2.3}, we obtain
\begin{equation*}
	\lim_{a \to +\infty} \mathfrak{T}(a \delta^i_1, \delta_0) = \mathsf{h}_\uparrow(0) = t_m \quad \text{and} \quad \lim_{\delta_0 \to 0^+} \mathfrak{T}(x, \delta_0) = \mathsf{h}_\uparrow(0) = t_m.
\end{equation*}

\textbf{Step 3:} We now prove parts \ref{b1}--\ref{b4} of Theorem \ref{t:mainthm2}. To prove Theorem \ref{t:mainthm2}.\ref{b1}, we first establish estimates for the composite function $\mathsf{h}_\uparrow \circ g$.

\begin{lemma}\label{t:bgest}
	For all $(t, x) \in \mathit{\Lambda}_{\delta_0}$, the composite function $\mathsf{h}_\uparrow \circ g(t, x)$ satisfies
	\begin{equation*}
		t_0 + \left(1 - C \sigma_0^2 e^{-\frac{100}{\delta_0}} e^{-\frac{3x^1}{4}}\right)(t - t_0) \leq \mathsf{h}_\uparrow \circ g(t, x) \leq t_0 + \left(1 + C \sigma_0^2 e^{-\frac{100}{\delta_0}} e^{-\frac{3x^1}{4}}\right)(t - t_0).
	\end{equation*}
\end{lemma}
\begin{proof}
The proof is analogous to that of \cite[Lemma 5.2]{Liu2024}, and the details are omitted.
\end{proof}

Now, using \eqref{e:rhoreg0}, Lemma \ref{t:bgest}, and the fact that $f$ is increasing (by Lemma \ref{t:f0fg}), we obtain for all $(t, x) \in \mathit{\Lambda}_{\delta_0}$,
\begin{gather*}
	\mathbf{1}_{-}(x^1) f_0\left(t_0 + \mathbf{1}_{-}(x^1)(t - t_0)\right) \leq \varrho_0(t, x) \leq \mathbf{1}_{+}(x^1) f_0\left(t_0 + \mathbf{1}_{+}(x^1)(t - t_0)\right), \\
	-C \sigma_0^2 e^{-\frac{100}{\delta_0}} e^{-\frac{3x^1}{4}} \left(1 + f\left(t_0 + \mathbf{1}_{-}(x^1)(t - t_0)\right)\right) \leq \varrho_i   \leq C \sigma_0^2 e^{-\frac{100}{\delta_0}} e^{-\frac{3x^1}{4}} \left(1 + f\left(t_0 + \mathbf{1}_{+}(x^1)(t - t_0)\right)\right), \\
	\mathbf{1}_{-}(x^1) f\left(t_0 + \mathbf{1}_{-}(x^1)(t - t_0)\right) \leq \varrho \leq \mathbf{1}_{+}(x^1) f\left(t_0 + \mathbf{1}_{+}(x^1)(t - t_0)\right), \\
	-\frac{\chi_\uparrow}{B} \frac{f}{f_0} C\sigma_0^2 e^{-\frac{100}{\delta_0}} e^{-\frac{3x^1}{4}} \leq \mathsf{h}_j \leq \frac{\chi_\uparrow}{B} \frac{f}{f_0} C\sigma_0^2 e^{-\frac{100}{\delta_0}} e^{-\frac{3x^1}{4}}, \\
	(1 - C\sigma_0^3 e^{-\frac{100}{\delta_0}} e^{-\frac{3x^1}{4}})^2 \mathsf{h}_\uparrow \leq \mathsf{h} \leq (1 + C\sigma_0^3 e^{-\frac{100}{\delta_0}} e^{-\frac{3x^1}{4}})^2 \mathsf{h}_\uparrow, %\label{e:breg0a}
\end{gather*}
where
\begin{equation*}
	\mathbf{1}_{-}(x^1) := 1 - C \sigma_0^2 e^{-\frac{100}{\delta_0}} e^{-\frac{3x^1}{4}}, \quad \mathbf{1}_{+}(x^1) := 1 + C \sigma_0^2 e^{-\frac{100}{\delta_0}} e^{-\frac{3x^1}{4}}.
\end{equation*}
This completes the proof of Theorem \ref{t:mainthm2}.\ref{b1}. Part \ref{t:mainthm2}.\ref{b2} follows from Corollary \ref{t:homdom}. \ref{t:mainthm2}.\ref{b3} and \ref{t:mainthm2}.\ref{b4} are obtained via Theorem \ref{t:mainthm0}.		
		
\textbf{Step 4:} We prove that $\varrho \in C^2(\mathcal{K} \cup \mathcal{H})$. Since the proof is similar to \cite[\S5.1.4]{Liu2024} with minor calculation revision, we omit the details. 
 
This completes the proof of Theorem \ref{t:mainthm2}.

\appendix

\section{Key Identities}\label{s:iden1}
This appendix lists several identities frequently used in reformulating the original equation into Fuchsian form and in subsequent derivations.

As shown in \eqref{e:Gdef0} of Appendix \ref{t:refsol}, $\chi(\mft)$ is defined by
\begin{equation*}\label{e:Gdef1}
	\chi_\uparrow(\mft) := \frac{\mft^{2 - \ca} f_0(\mft)}{(1 + f(\mft))^{2 - \cc} f(\mft) (-\mfg(\mft))^{\frac{\cb}{A}}} \overset{\eqref{e:f0aa}}{=} \frac{(-\mfg(\mft))^{-\frac{2\cb}{A}} \mft^{2 - 2\ca}}{B f(\mft) (1 + f(\mft))^{2 - 2\cc}}.
\end{equation*}
In terms of the compactified time variable $\tau$, it becomes
\begin{equation}\label{e:Gdef2}
	\underline{\chi_\uparrow}(\tau) := \frac{\mathsf{h}_\uparrow^{2 - \ca} \underline{f_0}(\tau)}{(1 + \uf(\tau))^{2 - \cc} \uf(\tau) (-\tau)^{\frac{\cb}{A}}} = \frac{(-\tau)^{-\frac{2\cb}{A}} \mathsf{h}_\uparrow^{2 - 2\ca}}{B \uf(\tau) (1 + \uf(\tau))^{2 - 2\cc}}.
\end{equation}

\begin{lemma}\label{t:iden1}
	Let $f \in C^2([t_0, t_1))$ ($t_1 > t_0$) be a solution to the ODE \eqref{e:feq0}--\eqref{e:feq1}, let $\mfg(t)$ be defined by \eqref{e:ctm2a}, and let $f_0(\mft) := \partial_{\mft} f(\mft)$. Then the following identities hold
	\begin{gather}
		\frac{\mathsf{h}_\uparrow^{1 - \ca} (1 + \uf)^{\cc - 1}}{A B \uf (-\tau)^{\frac{\cb}{A} + 1}} = \frac{\underline{\chi_\uparrow}^{\frac{1}{2}}}{A (-\tau) B^{\frac{1}{2}} \uf^{\frac{1}{2}}}, \label{e:iden3} \\
		\frac{1}{\underline{f_0}} \frac{\mathsf{h}_\uparrow^{-\ca} (1 + \uf)^{\cc}}{A B (-\tau)^{\frac{\cb}{A} + 1}} = \frac{1}{A (-\tau)}, \label{e:iden1} \\
		\frac{\mathsf{h}_\uparrow^{2 - \ca} (1 + \uf)^{\cc - 2}}{A B \uf (-\tau)^{\frac{\cb}{A} + 1}} \underline{f_0} = \frac{\underline{\chi_\uparrow}(\tau)}{A B (-\tau)}, \label{e:iden2} \\
		\frac{\mathsf{h}_\uparrow \underline{f_0}}{(1 + \uf)} = \frac{\underline{\chi_\uparrow}^{\frac{1}{2}}}{B^{\frac{1}{2}}} \uf^{\frac{1}{2}}. \label{e:keyid3}
	\end{gather}
\end{lemma}
\begin{proof}
The proof is analogous to that of \cite[Lemma 2.3]{Liu2024}, and the details are omitted.
\end{proof}

\section{Supplementary Lemmas}\label{s:pfmthm1}

\subsection{Some Necessary Estimates}
\begin{lemma}\label{t:Thpst}
	If the initial data satisfy condition \ref{A:1} in \S \ref{s:mainthm}, and $S(\tau)$, $\underline{\chi}(\tau)$, and $\underline{\mathfrak{G}}(\tau)$ are given by \eqref{e:S1}, \eqref{e:Gdef2}, and \eqref{e:chig} respectively, then for all $\tau \in [-1,0]$,
	\begin{enumerate}[label=(\arabic*)]
		\item\label{r:1} $\underline{\mathfrak{G}}(\tau) \geq 0$;
		\item\label{r:2} If $\frac{1}{3} \leq \cb \leq \frac{2}{3}$ and $\cc = \frac{4}{3}$, then $\cm^2 < S \leq \ck\left(3 + \frac{1}{\beta}\right)$;
		\item\label{r:3} If $\ck = 1$, $\frac{1}{3} \leq \cb \leq \frac{2}{3}$, $\cc = \frac{4}{3}$, and define
		\begin{align*}
			b := \sqrt{9\cb^2 \cm^2 - 6\cb \cm^2 + 6\cb},
		\end{align*}
		then
		\begin{equation*}
		b\in[1,2], \quad	\sqrt{S} \frac{\underline{\chi_\uparrow}}{B} = 2\sqrt{\frac{\underline{\chi_\uparrow}}{B}\left(\frac{\cm^2}{4}\frac{\underline{\chi_\uparrow}}{B} + (1 - \cm^2)\frac{1 + \uf}{\uf}\right)} = 2b + \Xi, \quad \text{and} \quad \sqrt{S} \frac{\underline{\chi_\uparrow}}{B} \geq 2,
		\end{equation*}
		where $\Xi(\tau)$ is a decreasing function.
	\end{enumerate}
\end{lemma}
\begin{proof}
	From \eqref{e:keyid3}, we have $\underline{\chi_\uparrow} = \frac{\mathsf{h}^2_\uparrow \underline{f_0}^2}{(1 + \uf)^2 \uf} B$. Substituting into \eqref{e:chig} yields
	\begin{align*}
		\underline{\mathfrak{G}} = \frac{\mathsf{h}^2_\uparrow \underline{f_0}^2}{(1 + \uf)^2 \uf} B - \frac{2\cb}{3 - 2\cc} B.
	\end{align*}
	By condition \ref{A:1} in \S \ref{s:mainthm}, the initial data satisfy $\beta_0^2 \geq \frac{2\cb}{3 - 2\cc} \beta (1 + \beta)^2 t_0^{-2}$, hence
	\begin{equation*}
		\underline{\mathfrak{G}}(-1) = \underline{\chi_\uparrow}(-1) - \frac{2\cb}{3 - 2\cc} B = \frac{t_0^2 \beta_0^2}{(1 + \beta)^2 \beta} B - \frac{2\cb}{3 - 2\cc} B \geq 0.
	\end{equation*}
	Combining with the result from \cite{Liu2024} (Appendix B) that $\lim_{\tau \to 0} \underline{\mathfrak{G}}(\tau) = 0$ and $|\underline{\mathfrak{G}}|$ is decreasing, we conclude $\underline{\mathfrak{G}}(\tau) \geq 0$ for all $\tau \in [-1,0]$, proving \ref{r:1}.
	
	If $\frac{1}{3} \leq \cb \leq \frac{2}{3}$ and $\cc = \frac{4}{3}$, since $\ck \geq \cm^2$ and $\underline{\mathfrak{G}}(\tau) \geq 0$, we compute
	\begin{equation*}
		0 \leq \frac{(\ck - \cm^2)(12 - 2\cb - 8\cc)}{2\cb + (3 - 2\cc)\frac{\underline{\mathfrak{G}}}{B}} = \frac{(\ck - \cm^2)(4 - 6\cb)}{6\cb + \frac{\underline{\mathfrak{G}}}{B}} \leq 2(\ck - \cm^2).
	\end{equation*}
	From the estimate for $S$ in Lemma \ref{t:coef1}, we have $\cm^2 < S(\tau) \leq \ck + (\ck - \cm^2)\left(\frac{1}{f} + 2\right)$, which implies $
		\cm^2 < S \leq \ck\left(3 + \frac{1}{\beta}\right)$, 
	proving \ref{r:2}.
	
	If $\cc = \frac{4}{3}$ and $\ck = 1$, using \eqref{e:S1}, \eqref{e:chig}, Lemma \ref{t:char1}, and the non-negativity of $\underline{\mathfrak{G}}(\tau)$
	\begin{align*}
		\sqrt{S} \frac{\underline{\chi_\uparrow}}{B} &= 2\sqrt{\frac{\underline{\chi_\uparrow}}{B}\left(\frac{\cm^2}{4}\frac{\underline{\chi_\uparrow}}{B} + (1 - \cm^2)\left(\frac{1}{\uf} + \frac{3}{2}\cb\right)\right) + \left(\frac{\underline{\chi_\uparrow}}{B}\right)^2 \frac{(1 - \cm^2)(1 - \frac{3\cb}{2})}{6\cb + \frac{\underline{\mathfrak{G}}}{B}}} \\
		&= 2\sqrt{\frac{\underline{\chi_\uparrow}}{B}\left(\frac{\cm^2}{4}\frac{\underline{\chi_\uparrow}}{B} + (1 - \cm^2)\left(\frac{1}{\uf} + 1\right)\right)}  
	 = 2b + \Xi.
	\end{align*}
	Here, $\frac{\underline{\chi_\uparrow}}{B} = 6\cb + \frac{\underline{\mathfrak{G}}}{B}$ and $1 \leq b = \sqrt{9\cb^2 \cm^2 - 6\cb \cm^2 + 6\cb} \leq 2$. Since $\underline{\chi_\uparrow}$ and $\underline{\mathfrak{G}}$ share the same monotonicity, $\Xi(\tau)$ is a decreasing function, completing the proof.
\end{proof}

\subsection{Boundary Between Known and Unknown Regions}\label{s:ctbdry}
The following lemmas provide the expressions for the hypersurface $\widehat{\Gamma}_{\delta_0}$ defined in \eqref{e:surf0} of \S \ref{s:reorg} in the coordinate systems $(\ttau,\txi)$, $(\tau,\zeta)$, and $(t,x)$, respectively. The proofs are based on lenghy but direct computation; only key steps and results are presented here.

\begin{lemma}\label{t:Tsurf2}
	The hypersurface $\widehat{\mathfrak{T}}_{\delta_0}(\hat{\zeta})$ in the coordinates $(\ttau,\txi)$ is given by
	\begin{align}\label{e:surf}
		\ttau =& \widetilde{\mathfrak{T}}_{\delta_0}(\tilde{\zeta}) \notag\\
		=&
		\begin{cases}
			-\exp\left(-\dfrac{\overline{\cc} \cdot 101A}{202b - 200}\left( 1 - \dfrac{1}{2\delta_0\txi^1 + 4} \right) \left( \txi^1 + \dfrac{1}{\delta_0} \right)\right), & -\dfrac{1}{\delta_0} \leq \txi^1 \leq \dfrac{202b - 200 - 101 \Xi_0}{202 \Xi_0} \dfrac{1}{\delta_0} \\
			-\exp\left(-\dfrac{101A}{202b - 200}\left( 1 - \dfrac{1}{2\delta_0\txi^1 + 4} \right) \left( \txi^1 + \dfrac{1}{2\delta_0} \right)\right), & \txi^1 > \dfrac{202b - 200 - 101 \Xi_0}{202 \Xi_0} \dfrac{1}{\delta_0}
		\end{cases}.
	\end{align}
	Here,
	\begin{equation}\label{e:cc}
		\overline{\cc} = \left(1 + \dfrac{101\Xi_0}{202b - 200} \right)^{-1} \quad \text{and} \quad \Xi_0 := \Xi(t_0) = \sup_{t \in [t_0, t_m)} \Xi(t).
	\end{equation}
	The inverse transformation yields $\tilde{\zeta}$ as
	\begin{align*}
		\txi^1 =
		\begin{cases}
			\dfrac{1}{404} \sqrt{\triangle_l(\ttau)} - \dfrac{(101 \Xi_0 + 202b - 200) \ln(-\ttau)}{202 A} - \dfrac{5}{4\delta_0}, & -1 \leq \ttau \leq -\exp\left(-\dfrac{A (101\Xi_0 + 101b - 100)}{\delta_0 \Xi_0 (303 \Xi_0 + 202b - 200)}\right) \\
			\dfrac{1}{202} \sqrt{\triangle_r(\ttau)} - \dfrac{(101b - 100) \ln(-\ttau)}{101 A} - \dfrac{1}{\delta_0}, & \ttau > -\exp\left(-\dfrac{A (101\Xi_0 + 101b - 100)}{\delta_0 \Xi_0 (303 \Xi_0 + 202b - 200)}\right)
		\end{cases}.
	\end{align*}
	where
	\begin{align*}
		\triangle_r(\ttau) :=& \dfrac{4(101b - 100)^2 (\ln(-\ttau))^2}{A^2} - \dfrac{808(101b - 100) \ln(-\ttau)}{A \delta_0} + \dfrac{10201}{\delta_0^2}, \\
		\triangle_l(\ttau) :=& \dfrac{4(101 \Xi_0 + 202b - 200)^2 (\ln(-\ttau))^2}{A^2} - \dfrac{1212(101 \Xi_0 + 202b - 200) \ln(-\ttau)}{A \delta_0} + \dfrac{10201}{\delta_0^2}.
	\end{align*}
\end{lemma}

\begin{lemma}\label{t:Tsurf3}
	The hypersurface $\widehat{\mathfrak{T}}_{\delta_0}(\hat{\zeta})$ in the coordinates $(\tau,\zeta)$ is given by
	\begin{align*}
		\tau :=& \tau_\Gamma(\zeta; \delta_0) \\
		=&
		\begin{cases}
			-\exp\left(\dfrac{A}{400 \delta_0 (\Xi_0 + 2b)} \sqrt{\Diamond_l(\zeta^1)} - \dfrac{A (101 \Xi_0 + 202b + 50)}{200 \delta_0 (\Xi_0 + 2b)} - \dfrac{A \zeta^1 (101 \Xi_0 + 202b + 200)}{400 (\Xi_0 + 2b)}\right), \\
			\hspace{6cm} -\dfrac{1}{\delta_0} \leq \zeta^1 \leq \dfrac{-303 \Xi_0^2 + 404b \Xi_0 + (404b^2 - 400b)}{\delta_0 \Xi_0 (606\Xi_0 + 404b - 400)} \\
			-\exp\left(\dfrac{A}{400b \delta_0} \sqrt{\Diamond_r(\zeta^1)} - \dfrac{101 A}{200 \delta_0} - \dfrac{(100 + 101b) A \zeta^1}{400b} \right), \\
			\hspace{6cm} \zeta^1 > \dfrac{-303 \Xi_0^2 + 404b \Xi_0 + (404b^2 - 400b)}{\delta_0 \Xi_0 (606\Xi_0 + 404b - 400)}
		\end{cases}.
	\end{align*}
	Here,
	\begin{align*}
		\Diamond_r(\zeta^1) :=& \delta_0 \zeta^1 (101b - 100) \left(101b(4 + \delta_0 \zeta^1) - 100\delta_0 \zeta^1 \right) + (40804b - 30300)b, \\
		\Diamond_l(\zeta^1) :=& \delta_0 \zeta^1 (101 \Xi_0 + 202b - 200) \left(\delta_0 \zeta^1 (101 \Xi_0 + 202b - 200) + 4 (101 \Xi_0 + 202b - 50)\right) \\
		&+ 4 \left(101 \Xi_0 (101 \Xi_0 + 404b - 200) + 4(10201b^2 - 10100b + 625)\right).
	\end{align*}
	The inverse transformation yields $\zeta^1$ as
	\begin{align*}
		\zeta^1 =
		\begin{cases}
			\dfrac{1}{404} \sqrt{\triangle_l(\tau)} - \dfrac{(101 \Xi_0 + 202b + 200) \ln(-\tau)}{202 A} - \dfrac{5}{4\delta_0}, & -1 \leq \tau \leq -\exp\left(-\dfrac{A (101\Xi_0 + 101b - 100)}{\delta_0 \Xi_0 (303 \Xi_0 + 202b - 200)}\right) \\
			\dfrac{1}{202} \sqrt{\triangle_r(\tau)} - \dfrac{(101b + 100) \ln(-\tau)}{101 A} - \dfrac{1}{\delta_0}, & \tau > -\exp\left(-\dfrac{A (101\Xi_0 + 101b - 100)}{\delta_0 \Xi_0 (303 \Xi_0 + 202b - 200)}\right)
		\end{cases}.
	\end{align*}
\end{lemma}

\begin{lemma}\label{t:Tsurf5}
	Assume $g(t,x)$ and $\mathsf{h}(\tau,\zeta)$ exist and are given by \eqref{e:coord2} and \eqref{e:coordi2}, respectively. Then, in the coordinates $(t,x)$, the surface $\widehat{\mathfrak{T}}_{\delta_0}(\hat{\zeta})$ can be expressed as
	\begin{align*}
		t :=& \mathsf{h}(\tau_\Gamma(x; \delta_0), x) \notag \\
		=&\begin{cases}
			\mathsf{h}\left(-\exp\left(\dfrac{A}{400 \delta_0 (\Xi_0 + 2b)} \sqrt{\Diamond_l(\zeta^1)} - \dfrac{A (101 \Xi_0 + 202b + 50)}{200 \delta_0 (\Xi_0 + 2b)} - \dfrac{A \zeta^1 (101 \Xi_0 + 202b + 200)}{400 (\Xi_0 + 2b)}\right), x\right), \\
			\hspace{7cm} -\dfrac{1}{\delta_0} \leq x^1 \leq \dfrac{-303 \Xi_0^2 + 404b \Xi_0 + (404b^2 - 400b)}{\delta_0 \Xi_0 (606\Xi_0 + 404b - 400)} \\
			\mathsf{h}\left(-\exp\left(\dfrac{A}{400b \delta_0} \sqrt{\Diamond_r(\zeta^1)} - \dfrac{101 A}{200 \delta_0} - \dfrac{(100 + 101b) A \zeta^1}{400b}\right), x\right), \\
			\hspace{7cm} x^1 > \dfrac{-303 \Xi_0^2 + 404b \Xi_0 + (404b^2 - 400b)}{\delta_0 \Xi_0 (606\Xi_0 + 404b - 400)}
		\end{cases}.
	\end{align*}
	Furthermore, $x^1$ can be expressed in terms of $g(t,x)$ as
	\begin{align*}
		x^1 :=& \mathrm{X}(g(t,x)) \notag \\
		=&
		\begin{cases}
			\dfrac{1}{404} \sqrt{\triangle_l(g(t,x))} - \dfrac{\Xi_0 \ln(-g(t,x))}{2 A} - \dfrac{(101b + 100) \ln(-g(t,x))}{101 A} - \dfrac{5}{4 \delta_0}, \\
			\hspace{6cm} -1 \leq g(t,x) \leq -\exp\left(-\dfrac{A (101\Xi_0 + 101b - 100)}{\delta_0 \Xi_0 (303 \Xi_0 + 202b - 200)}\right) \\
			\dfrac{1}{202} \sqrt{\triangle_r(g(t,x))} - \dfrac{(101b + 100) \ln(-g(t,x))}{101 A} - \dfrac{1}{\delta_0}, \\
			\hspace{6cm} g(t,x) > -\exp\left(-\dfrac{A (101\Xi_0 + 101b - 100)}{\delta_0 \Xi_0 (303 \Xi_0 + 202b - 200)}\right)
		\end{cases}.
	\end{align*}
\end{lemma}

The following lemma states an important conclusion: $p_m \in \overline{\mathit{\Lambda}_{\delta_0} \cap \mathcal{I}}$.
	\begin{lemma}\label{t:inI}
		For all points $(t, a\delta_1^i) \in \Gamma_{\delta_0}$ with $a \geq 0$, we have $(t, a\delta_1^i) \in \mathcal{I}$, where $\mathcal{I}$ is defined in \eqref{e:cai}.
	\end{lemma}
	\begin{proof}
		First, since $g(t,x) \in [-1,0)$, we observe that 
		\begin{equation*}
			\frac{303}{\delta_0} - \frac{2 (101 \Xi_0 + 202b - 200) (\ln(-g(t,x)))}{A} > 0 \AND \frac{202}{\delta_0} - \frac{(202b - 200) (\ln(-g(t,x)))}{A} > 0
		\end{equation*}
		Hence,
		\begin{align*}
			\sqrt{\triangle_r(g(t,x))} &< \frac{202}{\delta_0} - \frac{(202b - 200) (\ln(-g(t,x)))}{A}, \\
			\sqrt{\triangle_l(g(t,x))} &< \frac{303}{\delta_0} - \frac{2 (101 \Xi_0 + 202b - 200) (\ln(-g(t,x)))}{A}.
		\end{align*}
		Using Lemma \ref{t:Tsurf5}, we obtain the inequality
		\begin{equation}
			x^1 = \mathrm{X}(g(t,x))
			< \begin{cases}
				-\dfrac{(\Xi_0 + 2b)}{A} \ln(-g(t,x)) - \dfrac{1}{2\delta_0}, & -1 \leq g(t,x) \leq -\exp\left(-\dfrac{A (101\Xi_0 + 101b - 100)}{\delta_0 \Xi_0 (303 \Xi_0 + 202b - 200)}\right) \\
				-\dfrac{2b}{A} \ln(-g(t,x)), & g(t,x) > -\exp\left(-\dfrac{A (101\Xi_0 + 101b - 100)}{\delta_0 \Xi_0 (303 \Xi_0 + 202b - 200)}\right)
			\end{cases}. \label{e:x1upbd}
		\end{equation}
		
		We now prove $(t, a\delta_1^i) \in \mathcal{I}$ by contradiction. Suppose there exists a point $(t_1, a_1\delta_1^i) \in \Gamma_{\delta_0}$ with $a_1 \geq 0$ such that $(t_1, a_1\delta_1^i) \notin \mathcal{I}$. Since $(t_1, a_1\delta_1^i)$ lies in the homogeneous region $\mathcal{H}$, we have $g(t_1, a_1\delta^i_1) = \mfg(t_1)$. From \eqref{e:cah}, \eqref{e:keyid3}, Lemmas \ref{t:Tsurf5} and \ref{t:Thpst}.\ref{r:2}, it follows that
		\begin{equation*}
			\mathrm{X}(\mfg(t_1)) \geq 1 + \int_{t_0}^{t_1} \sqrt{\frac{\cm^2 f_0^2(y)}{(1 + f(y))^2} + (4 - 4\cm^2)\frac{1 + f(y)}{y^2}} \, dy = 1 + \int_{t_0}^{t_1} \frac{f_0(y)}{1 + f(y)} \frac{B}{\chi_\uparrow} (2b + \Xi) \, dy.
		\end{equation*}
		Combining \eqref{e:x1upbd} and Lemma \ref{t:gb2}.\ref{l:2.1}, we derive
		\begin{align*}
			& 1 + \int_{t_0}^{t_1} \frac{f_0(y)}{1 + f(y)} \frac{B}{\chi_\uparrow} (2b + \Xi) \, dy \notag \\
			<&
			\begin{cases}
				-\dfrac{(\Xi_0 + 2b)}{A} \ln(-\mfg(t_1)) - \dfrac{1}{2\delta_0} = (\Xi_0 + 2b) \int^{t_1}_{t_0} \frac{f(y)(f(y) + 1)}{y^2 f_0(y)} \, dy - \dfrac{1}{2\delta_0}, \\
				\hspace{7cm} -1 \leq \mfg(t_1) \leq -\exp\left(-\dfrac{A (101\Xi_0 + 101b - 100)}{\delta_0 \Xi_0 (303 \Xi_0 + 202b - 200)}\right) \\
				-\dfrac{2b}{A} \ln(-\mfg(t_1)) = 2b \int^{t_1}_{t_0} \frac{f(y)(f(y) + 1)}{y^2 f_0(y)} \, dy, \quad \mfg(t_1) > -\exp\left(-\dfrac{A (101\Xi_0 + 101b - 100)}{\delta_0 \Xi_0 (303 \Xi_0 + 202b - 200)}\right)
			\end{cases} \notag \\
			\overset{\eqref{e:keyid3}}{=}&
			\begin{cases}
				(\Xi_0 + 2b) \int^{t_1}_{t_0} \frac{f_0(y)}{1 + f(y)} \frac{B}{\chi_\uparrow} \, dy - \dfrac{1}{2\delta_0}, & -1 \leq \mfg(t_1) \leq -\exp\left(-\dfrac{A (101\Xi_0 + 101b - 100)}{\delta_0 \Xi_0 (303 \Xi_0 + 202b - 200)}\right) \\
				2b \int^{t_1}_{t_0} \frac{f_0(y)}{1 + f(y)} \frac{B}{\chi_\uparrow} \, dy, & \mfg(t_1) > -\exp\left(-\dfrac{A (101\Xi_0 + 101b - 100)}{\delta_0 \Xi_0 (303 \Xi_0 + 202b - 200)}\right)
			\end{cases}.
		\end{align*}
		Since $\Xi > 0$ (by Lemma \ref{t:char1}) and $B/\chi_\uparrow < \frac{1}{6\cb}$ (from \eqref{e:chig} and Lemma \ref{t:Thpst}.\ref{r:1}), we deduce
		
		When $\mfg(t_1) > -\exp\left(-\dfrac{A (101\Xi_0 + 101b - 100)}{\delta_0 \Xi_0 (303 \Xi_0 + 202b - 200)}\right)$,
		\[
		1 < -\int^{t_1}_{t_0} \frac{f_0(y)}{1 + f(y)} \frac{B}{\chi_\uparrow} \Xi \, dy < 0;
		\]
		
		When $-1 \leq \mfg(t_1) \leq -\exp\left(-\dfrac{A (101\Xi_0 + 101b - 100)}{\delta_0 \Xi_0 (303 \Xi_0 + 202b - 200)}\right)$,
		\[
		1 < \int^{t_1}_{t_0} \frac{f_0(y)}{1 + f(y)} \frac{B}{\chi_\uparrow} (\Xi_0 - \Xi) \, dy - \frac{1}{2\delta_0} < \frac{\Xi_0}{6\cb} \ln \frac{1 + f(t_1)}{1 + \beta} - \frac{1}{2\delta_0}.
		\]
		
		For sufficiently small $\delta_0$, specifically $\delta_0 < \left(\frac{\Xi_0}{3\cb} \ln \frac{1 + f(t_1)}{1 + \beta}\right)^{-1}$, we have
		\[
		1 < -\int^{t_1}_{t_0} \frac{f_0(y)}{1 + f(y)} \frac{B}{\chi_\uparrow} \Xi \, dy < 0,
		\]
		which leads to the contradiction $1 < 0$. This completes the proof.
	\end{proof}

\section{Fuchsian Global Initial Value Problem on Compact Domains}\label{s:fuch}

The Fuchsian global initial value problem (GIVP), introduced by Oliynyk in \cite{Oliynyk2016a}, provides a framework for analyzing singular symmetric hyperbolic systems. This appendix presents the formulation from \cite{Beyer2020}, adapted to our setting and serving as a key tool in this work. For further developments and applications of the Fuchsian GIVP, see \cite{Liu2018,Liu2018c,Liu2018a,Beyer2020b,Fajman2021,LeFloch2021,Ames2022,Ames2022a,Marshall2023,Liu2022,Liu2022a,Liu2023,Beyer2023a,Beyer2023b,Fajman2023,Oliynyk2024}.  
Consider the singular symmetric hyperbolic system

\begin{align}
	B^{\mu}(t,x,u)\partial_{\mu}u =& \frac{1}{t}\mathbf{B}(t,x,u)\mathbf{P}u + H(t,x,u) \quad &&\text{on } [T_{0},T_{1})\times\mathbb{R}^{n}, \label{e:model1} \\
	u =& u_{0}(x) &&\text{on } \{T_{0}\}\times\mathbb{R}^{n}, \label{e:model2}
\end{align}
where $T_{0} < T_{1} \leq 0$, and the system satisfies the following conditions. 
\begin{enumerate}[label=(H\arabic*)]
	\item \label{c:2} $\mathbf{P}$ is a constant symmetric projection operator, i.e., $\mathbf{P}^{2} = \mathbf{P}$, $\mathbf{P}^{T} = \mathbf{P}$, and $\partial_\mu \mathbf{P} = 0$. Denote the complementary projection by $\mathbf{P}^\perp := \mathds{1} - \mathbf{P}$.
	
	\item \label{c:3} $u = u(t,x)$ and $H(t,x,u)$ are $\mathbb{R}^{N}$-valued maps, with $H \in C^{0}([T_{0},0), C^{\infty}(\mathbb{T}^n \times B_R(\mathbb{R}^{N}), \mathbb{R}^N))$, and $H$ admits the expansion
$
		H(t,x,u) = H_0(t,x,u) + |t|^{-\frac{1}{2}} H_1(t,x,u)$, 
	where $H_0, H_1 \in C^0([T_0,0], C^{\infty}(\mathbb{T}^n \times B_R(\mathbb{R}^{N}), \mathbb{R}^N))$, and there exist constants $\lambda_a \geq 0$, $a = 1,2$, such that for all $(t,x) \in [T_0,0) \times \mathbb{T}^n$,
	\begin{equation*}
		H_0(t,x,u) = \mathcal{O}(u), \quad \mathbf{P}H_1(t,x,u) = \mathcal{O}(\lambda_1 u), \quad \mathbf{P}^\perp H_1(t,x,u) = \mathcal{O}(\lambda_2 \mathbf{P} u).
	\end{equation*}
	
	\item \label{c:4} $B^{\mu} = B^{\mu}(t,x,u)$ and $\mathbf{B} = \mathbf{B}(t,x,u)$ are $\mathbb{M}_{N \times N}$-valued maps. There exists a constant $R > 0$ such that $B^{i} \in C^{0}([T_{0},0), C^{\infty}(\mathbb{T}^n \times B_R(\mathbb{R}^{N}), \mathbb{M}_{N \times N}))$, $\mathbf{B} \in C^{0}([T_{0},0], C^{\infty}(\mathbb{T}^n \times B_R(\mathbb{R}^{N}), \mathbb{M}_{N \times N}))$, $B^{0} \in C^{1}([T_{0},0), C^{\infty}(\mathbb{T}^n \times B_R(\mathbb{R}^{N}), \mathbb{M}_{N \times N}))$, and they satisfy $
		(B^{\mu})^{T} = B^{\mu}, \quad [\mathbf{P}, \mathbf{B}] = \mathbf{PB} - \mathbf{BP} = 0$, 
	and $B^i$ can be expanded as
$
		B^i(t,x,u) = B_0^i(t,x,u) + |t|^{-1} B_2^i(t,x,u)$, 
	where $B_0^i, B_2^i \in C^{0}([T_{0},0], C^{\infty}(\mathbb{T}^n \times B_R(\mathbb{R}^{N}), \mathbb{M}_{N \times N}))$. 
	Moreover, there exist $\tilde{B}^0, \tilde{\mathbf{B}} \in C^0([T_0,0], C^\infty(\mathbb{T}^n, \mathbb{M}_{N \times N}))$ such that for all $(t,x,u) \in [T_{0},0] \times \mathbb{T}^n \times B_R(\mathbb{R}^N)$,
	\begin{equation*}
		[\mathbf{P}, \tilde{\mathbf{B}}] = 0, \quad
		B^0(t,x,u) - \tilde{B}^0(t,x) = \mathcal{O}(u), \quad
		\mathbf{B}(t,x,u) - \tilde{\mathbf{B}}(t,x) = \mathcal{O}(u).
	\end{equation*}
	
	Additionally, there exists $\tilde{B}_2^i \in C^0([T_0,0], C^\infty(\mathbb{T}^n, \mathbb{M}_{N \times N}))$ such that for all $(t,x,u) \in [T_{0},0] \times \mathbb{T}^n \times B_R(\mathbb{R}^N)$,
	\begin{gather*}
		\mathbf{P} B_2^i(t,x,u) \mathbf{P}^\perp = \mathcal{O}(\mathbf{P} u), \quad
		\mathbf{P}^\perp B_2^i(t,x,u) \mathbf{P} = \mathcal{O}(\mathbf{P} u), \\
		\mathbf{P}^\perp B_2^i(t,x,u) \mathbf{P}^\perp = \mathcal{O}(\mathbf{P} u \otimes \mathbf{P} u), \quad
		\mathbf{P} (B_2^i(t,x,u) - \tilde{B}_2^i(t,x)) \mathbf{P} = \mathcal{O}(u).
	\end{gather*}
	
	\item \label{c:5} There exist constants $\kappa, \gamma_1, \gamma_2$ such that for all $(t,x,u) \in [T_{0},0] \times \mathbb{T}^n \times B_R(\mathbb{R}^N)$,
	\begin{equation*}
		\frac{1}{\gamma_1} \mathds{1} \leq B^{0} \leq \frac{1}{\kappa} \mathbf{B} \leq \gamma_2 \mathds{1}. \label{e:Bineq}
	\end{equation*}
	
	\item \label{c:6} For all $(t,x,u) \in [T_{0},0] \times \mathbb{T}^n \times B_R(\mathbb{R}^{N})$, $
		\mathbf{P}^{\perp} B^{0}(t, \mathbf{P}^\perp u) \mathbf{P} = \mathbf{P} B^{0}(t, \mathbf{P}^\perp u) \mathbf{P}^{\perp} = 0$. 
	
	\item \label{c:7} There exist constants $\theta$ and $\beta_{\ell} \geq 0$, $\ell = 0, \dots, 7$, such that
	\begin{align*}
		\mathrm{div} B(t,x, u,w) :=& \partial_{t} B^0(t,x,u) + D_u B^0(t,x,u) \cdot (B^0(t,x,u))^{-1} \Bigl[ -B^i(t,x,u) \cdot w_i \\
		&+ \frac{1}{t} \mathbf{B}(t,x,u) \mathbf{P} u + H_0(t,x,u) + |t|^{-\frac{1}{2}} H_2(t,x,u) \Bigr]  + \partial_{i} B^i(t,x,u) + D_u B^i(t,x,u) \cdot w_i,
	\end{align*}
	where $w = (w_i)$, $(t,x,u,w) \in [T_0,0) \times \mathbb{T}^n \times B_R(\mathbb{R}^N) \times B_R(\mathbb{M}_{N \times n})$, and
	\begin{align}
		\mathbf{P} \mathrm{div} B \mathbf{P} =& \mathcal{O}\left(\theta + |t|^{-\frac{1}{2}} \beta_0 + |t|^{-1} \beta_1 \right), \label{e:PhP1} \\
		\mathbf{P} \mathrm{div} B \mathbf{P}^\perp =& \mathcal{O}\left(\theta + |t|^{-\frac{1}{2}} \beta_2 + \frac{|t|^{-1} \beta_3}{R} \mathbf{P} u \right), \label{e:PhP2} \\
		\mathbf{P}^\perp \mathrm{div} B \mathbf{P} =& \mathcal{O}\left(\theta + |t|^{-\frac{1}{2}} \beta_4 + \frac{|t|^{-1} \beta_5}{R} \mathbf{P} u \right), \label{e:PhP3} \\
		\intertext{and}
		\mathbf{P}^\perp \mathrm{div} B \mathbf{P}^\perp =& \mathcal{O}\left(\theta + \frac{|t|^{-\frac{1}{2}} \beta_6}{R} \mathbf{P} u + \frac{|t|^{-1} \beta_7}{R^2} \mathbf{P} u \otimes \mathbf{P} u \right). \label{e:PhP4}
	\end{align}
\end{enumerate}

\begin{theorem}\label{t:fuch}
	Assume $k \in \mathbb{Z}_{> \frac{n}{2} + 3}$, $u_{0} \in H^{k}(\mathbb{T}^{n})$, the system \eqref{e:model1} satisfies conditions \ref{c:2}--\ref{c:7}, and the constants $\kappa, \gamma_1, \beta_1, \beta_3, \beta_5, \beta_7$ from conditions \ref{c:2}--\ref{c:7} satisfy
	\begin{equation}\label{e:kpbt1}
		\kappa > \frac{1}{2} \gamma_1 \max\left\{ \sum^3_{\ell=0} \beta_{2\ell+1}, \beta_1 + 2k(k+1) \mathtt{b} \right\},
	\end{equation}
	where
	\begin{equation*}
		\mathtt{b} := \sup_{T_0 \leq t < 0} \left( \left\| |\mathbf{P} \tilde{\mathbf{B}} D (\tilde{\mathbf{B}}^{-1} \tilde{B}^0) (\tilde{B}^0)^{-1} \mathbf{P} \tilde{B}_2^i \mathbf{P}|_{\mathrm{op}} \right\|_{L^\infty} + \left\| |\mathbf{P} \tilde{\mathbf{B}} D (\tilde{\mathbf{B}}^{-1} \tilde{B}_2^i) \mathbf{P}|_{\mathrm{op}} \right\|_{L^\infty} \right),
	\end{equation*}
	then there exist constants $\delta_0, \delta > 0$ with $\delta < \delta_0$ such that if the initial data satisfy $
		\|u_0\|_{H^k} \leq \delta$, 
	the initial value problem \eqref{e:model1}--\eqref{e:model2} has a unique solution
	\begin{equation*}
		u \in C^0([T_0,0), H^k(\mathbb{T}^n)) \cap C^1([T_0,0), H^{k-1}(\mathbb{T}^n)) \cap L^\infty([T_0,0), H^k(\mathbb{T}^n)),
	\end{equation*}
	and the limit $\mathbf{P} u(0) := \lim_{t \nearrow 0} \mathbf{P} u(t)$ exists in $H^{k-1}(\mathbb{T}^n)$. Moreover, for all $T_0 \leq t < 0$, the solution $u$ satisfies the energy estimate:
	\begin{equation*}\label{e:ineq1}
		\|u(t)\|_{H^k(\mathbb{T}^n)}^2 - \int^t_{T_0} \frac{1}{\tau} \|\mathbf{P} u(\tau)\|^2_{H^k(\mathbb{T}^n)} d\tau \leq C(\delta_0, \delta_0^{-1}) \|u_0\|^2_{H^k(\mathbb{T}^n)}.
	\end{equation*}
\end{theorem}

\section{Reference ODE}\label{s:ODE0}
This appendix presents an important class of ODEs studied in \cite{Liu2022b} along with related results. Detailed derivations can be found in \cite{Liu2022b} and \cite{Liu2023}. Note that in this paper, $\mfg$ corresponds to $g$ in \cite{Liu2022b}.

\subsection{Analysis of the Reference ODE}\label{s:ODE}
Consider the following ODE:
\begin{gather}
	f^{\prime\prime}(\mft) + \frac{\ca}{\mft} f^\prime(\mft) - \frac{\cb}{\mft^2} f(\mft)(1 + f(\mft)) - \frac{\cc (f^\prime(\mft))^2}{1 + f(\mft)} = 0, \label{e:feq0} \\
	f(t_0) = \mf > 0 \quad \text{and} \quad f^\prime(t_0) = \mf_0 > 0, \label{e:feq1}
\end{gather}
where $\mf, \mf_0 > 0$ are constants, and the parameters $\ca$, $\cb$, and $\cc$ satisfy
\begin{equation}\label{e:abcdk}
	\ca > 1, \quad \cb > 0, \quad \text{and} \quad 1 < \cc < 3/2.
\end{equation}

For notational simplicity, define
\begin{equation*}
	\triangle := \sqrt{(1 - \ca)^2 + 4\cb} > -\ba, \quad \ba := 1 - \ca < 0, \quad \bc := 1 - \cc < 0.
\end{equation*}
Introduce constants $\mathtt{A}$, $\mathtt{B}$, $\mathtt{C}$, $\mathtt{D}$, and $\mathtt{E}$, which depend on the parameters $\ca$, $\cb$, $\cc$ in \eqref{e:feq0}--\eqref{e:feq1} and the initial values $\mf$, $\mf_0$.
\begin{align*}
	\mathtt{A} :=& \frac{t_0^{-\frac{\ba - \triangle}{2}}}{\triangle} \left( \frac{t_0 \mf_0}{(1 + \mf)^2} - \frac{\ba + \triangle}{2} \frac{\mf}{1 + \mf} \right), \\
	\mathtt{B} :=& \frac{t_0^{-\frac{\ba + \triangle}{2}}}{\triangle} \left( \frac{\ba - \triangle}{2} \frac{\mf}{1 + \mf} - \frac{t_0 \mf_0}{(1 + \mf)^2} \right) < 0, \\
	\mathtt{C} :=& \frac{2}{2 + \ba + \triangle} \left( \max\left\{ \frac{\ba + \triangle}{2 \cb}, 1 \right\} \right)^{-1} \left( \ln(1 + \mf) + \frac{\ba + \triangle}{2\cb} \frac{t_0 \mf_0}{1 + \mf} \right) t_0^{-\frac{\ba + \triangle}{2}} > 0, \\
	\mathtt{D} :=& \frac{\ba + \triangle}{2 + \ba + \triangle} \left( \ln(1 + \mf) - \frac{1}{\cb} \frac{t_0 \mf_0}{1 + \mf} \right) t_0 \AND
	\mathtt{E} :=  \frac{\bc \mf_0 t_0^{1 - \ba}}{\ba (1 + \mf)} > 0.
\end{align*}

Next, define two critical time points $t_\star$ and $t^\star$:
\begin{definition}\label{t:tdef}
	Given $\mathtt{A}$, $\mathtt{B}$, $\mathtt{E}$, $\ba$, and $\triangle$ as above:
	\begin{enumerate}[label=(\arabic*)]
		\item Let $\mathcal{R} := \{t_r > t_0 \mid \mathtt{A} t_r^{\frac{\ba - \triangle}{2}} + \mathtt{B} t_r^{\frac{\ba + \triangle}{2}} + 1 = 0\}$, and define $t_\star := \min \mathcal{R}$.
		\item If $t_0^{\ba} > \mathtt{E}^{-1}$, define $t^\star := (t_0^{\ba} - \mathtt{E}^{-1})^{1/\ba} \in (0, \infty)$, i.e., $\mft = t^\star$ solves $1 - \mathtt{E} t_0^{\ba} + \mathtt{E} \mft^{\ba} = 0$.
	\end{enumerate}
\end{definition}

The main theorem concerning the ODE \eqref{e:feq0}--\eqref{e:feq1} is stated below; its proof can be found in \cite{Liu2022b}:
\begin{theorem}\label{t:mainthm0}
	Assume the constants $\ca$, $\cb$, $\cc$ satisfy \eqref{e:abcdk}, $t_\star$ and $t^\star$ are defined as above, and initial values $\mf, \mf_0 > 0$. Then
	\begin{enumerate}[label=(\arabic*)]
		\item There exists $t_\star \in [0, \infty)$ with $t_\star > t_0$.
		\item There exists a constant $t_m \in [t_\star, \infty]$ such that the initial value problem \eqref{e:feq0}--\eqref{e:feq1} has a unique solution $f \in C^2([t_0, t_m))$ satisfying $
			\lim_{\mft \rightarrow t_m} f(\mft) = +\infty$ and $ \lim_{\mft \rightarrow t_m} f_0(\mft) = +\infty$. 
		\item The solution $f$ satisfies the bounds,
		\begin{align*}
			1 + f(\mft) >& \exp\left( \mathtt{C} \mft^{\frac{\ba + \triangle}{2}} + \mathtt{D} \mft^{-1} \right), && \mft \in (t_0, t_m); \\
			1 + f(\mft) <& \left( \mathtt{A} \mft^{\frac{\ba - \triangle}{2}} + \mathtt{B} \mft^{\frac{\ba + \triangle}{2}} + 1 \right)^{-1}, && \mft \in (t_0, t_\star).
		\end{align*}
	\end{enumerate}
	Furthermore, if the initial data satisfy $\mf_0 > \ba(1 + \mf) / (\bc t_0)$, then
	\begin{enumerate}[label=(\arabic*)]
		\setcounter{enumi}{3}
		\item $t_\star$ and $t^\star$ exist and are finite, with $t_0 < t_\star < t^\star < \infty$.
		\item There exists a finite $t_m \in [t_\star, t^\star)$ such that the solution $f \in C^2([t_0, t_m))$ to \eqref{e:feq0}--\eqref{e:feq1} satisfies $
			\lim_{\mft \rightarrow t_m} f(\mft) = +\infty$ and $\lim_{\mft \rightarrow t_m} f_0(\mft) = +\infty$. 
		\item For all $\mft \in (t_0, t_m)$, the lower bound of $f$ can be improved to
		\begin{equation*}\label{e:ipvest}
			(1 + \mf) \left(1 - \mathtt{E} t_0^{\ba} + \mathtt{E} \mft^{\ba} \right)^{1/\bc} < 1 + f(\mft).
		\end{equation*}
	\end{enumerate}
\end{theorem}

\subsection{Analysis of the Reference Solution} \label{t:refsol}
\subsubsection{Time Transformation Function $\mfg(\mft)$ and Estimates of $f$ and $\partial_{\mft} f(\mft)$}\label{t:ttf}
Define
\begin{align}\label{e:gdef0a}
	\mfg(\mft) := -\exp\left( -A \int^{\mft}_{t_0} \frac{f(s)(f(s) + 1)}{s^2 f_0(s)}  ds \right) < 0,
\end{align}
where the constant $A \in (0, 2\cb/(3 - 2\cc))$.

\begin{lemma}\label{t:f0fg}
	Let $t_1 > t_0$, and let $f \in C^2([t_0, t_1))$ be a solution to the ODE \eqref{e:feq0}--\eqref{e:feq1}. Define $\mfg(\mft)$ by \eqref{e:gdef0a}, and let $f_0(\mft) := \partial_{\mft} f(\mft)$. Then
	\begin{enumerate}[label=(\arabic*)]
		\item For all $\mft \in [t_0, t_1)$, $f_0$ can be expressed as
		\begin{equation}\label{e:f0aa}
			f_0(\mft) = B^{-1} \mft^{-\ca} (-\mfg(\mft))^{-\frac{\cb}{A}} (1 + f(\mft))^{\cc} > 0,
		\end{equation}
		where $B := (1 + \mf)^\cc / (t_0^{\ca} \mf_0) > 0$ is a constant depending on the initial data.
		
		\item If the initial value $\mf > 0$, then $f(\mft) > 0$ for all $\mft \in [t_0, t_1)$.
		
		\item For all $\mft \in [t_0, t_1)$, $\mfg(\mft)$ can be written as
		\begin{equation}\label{e:ctm2a}
			\mfg(\mft) = -\left(1 + \cb B \int^\mft_{t_0} s^{\ca - 2} f(s)(1 + f(s))^{1 - \cc}  ds \right)^{-\frac{A}{\cb}} \in [-1, 0),
		\end{equation}
		and $\mfg(t_0) = -1$.
		
		\item The function $\mfg(t)$ is strictly increasing and invertible on $[t_0, t_1)$.
	\end{enumerate}
\end{lemma}

\begin{remark}\label{t:dtg1}
	Two important identities from \cite{Liu2022b} are
	\begin{align*}
		\partial_{\mft} \mfg(\mft) =& A B (-\mfg(\mft))^{\frac{\cb}{A} + 1} \mft^{\ca - 2} f(\mft) (1 + f(\mft))^{1 - \cc}, \\
		\partial_{\mft} (-\mfg(\mft))^{-\frac{\cb}{A}} =& \frac{\cb}{A} (-\mfg(\mft))^{-\frac{\cb}{A} - 1} \partial_{\mft} \mfg(\mft) = \cb B \mft^{\ca - 2} f(\mft) (1 + f(\mft))^{1 - \cc}.
	\end{align*}
\end{remark}

\subsubsection{Estimates of $\chi(\mft)$ and $\xi(\mft)$}
The quantities $\chi(\mft)$ and $\xi(\mft)$ are frequently used in the analysis of Chapter \ref{cha:thirdsection}. This subsection provides relevant estimates for these important functions; see \cite{Liu2022b} for details. Define
\begin{equation}\label{e:Gdef0}
	\chi(\mft) := \frac{\mft^{2 - \ca} f_0(\mft)}{(1 + f(\mft))^{2 - \cc} f(\mft) (-\mfg)^{\frac{\cb}{A}}(\mft)} \overset{\eqref{e:f0aa}}{=} \frac{(-\mfg)^{-\frac{2\cb}{A}}(\mft) \mft^{2(1 - \ca)}}{B f(\mft) (1 + f(\mft))^{2(1 - \cc)}} > 0.
\end{equation}

\begin{lemma}\label{t:gmap}
	Let $\mfg(\mft)$ be defined by \eqref{e:gdef0a}, $\cc \in (1, 3/2)$, and let $f \in C^2([t_0, t_m))$ be a solution to the ODE \eqref{e:feq0}--\eqref{e:feq1}, where $[t_0, t_m)$ is the maximal interval of existence from Theorem \ref{t:mainthm0}. Then 
$
	\lim_{\mft \rightarrow t_m} \mfg(\mft) = 0
$. 
\end{lemma}

\begin{remark}
	By the above lemma, $\mfg(\mft)$ can be continuously extended to $[t_0, t_m]$ by setting $\mfg(t_m) := \lim_{\mft \rightarrow t_m} \mfg(\mft) = 0$, so that $\mfg^{-1}(0) = t_m$.
\end{remark}

\begin{proposition}\label{t:limG}
	Assume $\cc \in (1, 3/2)$, $\cb > 0$, $\ca > 1$, and let $\chi$ be defined by \eqref{e:Gdef0}. If $f \in C^2([t_0, t_m))$ is a solution to the ODE \eqref{e:feq0}--\eqref{e:feq1}, then there exists a function $\mathfrak{G} \in C^1([t_0, t_m))$ such that for all $\mft \in [t_0, t_m)$,
	\begin{equation}\label{e:limG}
		\chi(\mft) = \frac{2\cb B}{3 - 2\cc} + \mathfrak{G}(\mft),
	\end{equation}
	and $\lim_{\mft \rightarrow t_m} \mathfrak{G}(\mft) = 0$. Moreover, by setting $\chi(t_m) := 2\cb B/(3 - 2\cc)$ and $\mathfrak{G}(t_m) := 0$, there exists a constant $C_\chi > 0$ such that $0 < \chi(\mft) \leq C_\chi$ for all $\mft \in [t_0, t_m)$, and both $\chi$ and $\mathfrak{G}$ can be continuously extended to $[t_0, t_m]$, i.e., $\chi \in C^0([t_0, t_m])$ and $\mathfrak{G} \in C^0([t_0, t_m])$.
\end{proposition}

Define $\xi(\mft)$ as follows:
\begin{equation}\label{e:xidef}
	\xi(\mft) := \frac{1}{-\mfg(\mft) (1 + f(\mft))}.
\end{equation}
The following proposition establishes the boundedness of $\xi$ and shows that $\xi$ tends to zero as $\mft$ approaches $t_m$.

\begin{proposition}\label{t:fginv0}
	Let $f \in C^2([t_0, t_m))$ be a solution to the ODE \eqref{e:feq0}--\eqref{e:feq1}, where $[t_0, t_m)$ is the maximal interval of existence from Theorem \ref{t:mainthm0}. Let $\mfg(t)$ and $\xi(\mft)$ be defined by \eqref{e:gdef0a} and \eqref{e:xidef}, respectively. Then $\xi \in C^1([t_0, t_m))$ and
$
	\lim_{\mft \rightarrow t_m} \xi(\mft) = 0$. 
	Furthermore, there exists a constant $C_\star > 0$ such that $0 < \xi(\mft) \leq C_\star$ for all $\mft \in [t_0, t_m)$. By setting $\chi(t_m) := 2\cb B/(3 - 2\cc)$ and $\xi(t_m) := 0$, the function $\xi$ can be continuously extended to $[t_0, t_m]$, i.e., $\xi \in C^0([t_0, t_m])$.
\end{proposition}

The next lemma gives an expression for $\partial_{\mft} \chi$; see \cite{Liu2023} for details.

\begin{lemma}\label{t:dtchi}
	If $\chi(\mft)$ and $\mathfrak{G}$ are defined by \eqref{e:Gdef0} and \eqref{e:limG}, respectively, then
	\begin{equation*}
		\partial_{\mft} \chi = \partial_{\mft} \mathfrak{G} = -\frac{(3 - 2\cc) \mathfrak{G} f^{\frac{1}{2}} \chi^{\frac{1}{2}}}{B^{\frac{1}{2}} \mft} - \frac{\chi^{\frac{3}{2}}}{B^{\frac{1}{2}} \mft f^{\frac{1}{2}}} + 2(1 - \ca) \frac{\chi}{\mft}.
	\end{equation*}
\end{lemma}

\begin{corollary}\label{t:dtchi2}
	Under the conditions of Lemma \ref{t:dtchi}, the function $|\mathfrak{G}(\mft)|$ is decreasing and satisfies $-|\mathfrak{G}(t_0)| < \mathfrak{G}(\mft) < |\mathfrak{G}(t_0)|$.
\end{corollary}
\begin{proof}
The proof is similar to \cite[Corollary B.1]{Liu2024}. We omit the details. 
\end{proof}

\begin{lemma}\label{t:Gest2}
	If $\mathfrak{G}$ is defined as in \eqref{e:chig}, then for all $\ttau \in [-1, 0)$, $\tilde{\underline{\mathfrak{G}}}(\ttau)$ satisfies the estimate
	$
	|\tilde{\underline{\mathfrak{G}}}(\ttau)| \lesssim (-\ttau)^{\frac{1}{2}}$, 
	and the function $(-\ttau)^{-\frac{1}{2}} \tilde{\underline{\mathfrak{G}}}(\ttau)$ can be continuously extended to $\ttau \in [-1, 0]$, i.e., $(-\ttau)^{-\frac{1}{2}} \tilde{\underline{\mathfrak{G}}}(\ttau) \in C^0([-1, 0])$.
\end{lemma}

\section*{Acknowledgement}
		%We would like to thank ... for their helpful discussions, comments and advice.
C.L. is partially supported by  NSFC (Grant No. $12571234$). 
		%We also thank the referee for their comments and criticisms, which have served to improve the content and exposition of this article.

\bigskip

\textbf{Data Availability} Data sharing is not applicable to this article as no datasets were
generated or analysed during the current study.

\bigskip

\textbf{Declarations}

\bigskip

\textbf{Conflict of interest} The authors declare that they have no conflict of interest.

\bibliographystyle{amsplain}
\bibliography{Reference_Chao}

\end{document}